\documentclass[letterpaper]{article} 
\usepackage[preprint]{aaai2027}  
\usepackage[hyphens]{url}  
\usepackage{graphicx} 
\usepackage{natbib}  
\usepackage{caption} 
\usepackage{amsmath}
\usepackage{amssymb}
\usepackage{amsfonts}
\usepackage{mathtools}
\usepackage{amsthm}
\usepackage{bm}

\usepackage{booktabs}
\usepackage{array}
\usepackage{tabularx}
\usepackage{multirow}
\usepackage{makecell}
\usepackage{adjustbox}
\usepackage{siunitx}

\usepackage{algorithm}
\usepackage{algorithmic}

\usepackage{tikz}
\usetikzlibrary{
    arrows.meta,
    positioning,
    calc,
    fit,
    backgrounds,
    shapes.geometric,
    shadows
}

\usepackage{textcomp}
\usepackage{xcolor}
\usepackage{enumitem}
\usepackage{microtype}
\usepackage[capitalize,noabbrev]{cleveref}

\newtheorem{theorem}{Theorem}[section]
\newtheorem{proposition}[theorem]{Proposition}
\newtheorem{lemma}[theorem]{Lemma}
\newtheorem{corollary}[theorem]{Corollary}

\theoremstyle{definition}

\newtheorem{assumption}[theorem]{Assumption}

\DeclareMathOperator{\Proj}{Proj}
\DeclareMathOperator{\Fix}{Fix}
\DeclareMathOperator{\dist}{dist}

\DeclareMathOperator*{\argmin}{arg\,min}

\newcolumntype{C}[1]{>{\centering\arraybackslash}m{#1}}

\title{Residual-Controlled Douglas--Rachford Splitting for Differentiable Solver Layers}
\author{Kang Liu, ~Jianchen Hu}
\affiliations{School of Future Technology, Xi'an Jiaotong University\\
School of Automation Science and Engineering, Xi'an Jiaotong University}

\begin{document}

\maketitle

\begin{abstract}
Differentiable solver layers embed constrained optimization into end-to-end learning systems, but fixed-depth unrolling must trade off solution quality, feasibility, and computational budget. We propose Residual-Controlled Douglas--Rachford Splitting (RCDRS), a differentiable solver layer for conic linear programs. RCDRS treats an unrolled solver as a feedback-controlled dynamical system, where a causal controller adapts the relaxation and objective-drive parameters while preserving the projection-splitting structure of Douglas--Rachford splitting. Theoretically, we show that each fixed admissible block remains an averaged relaxed DRS operator and admits finite-step fixed-point residual bounds. We further analyze safeguarded time-varying rollouts as summable perturbations of a limiting averaged operator, and recover terminal primal-dual diagnostics from the final splitting state. Experiments on mixed-cone benchmarks and engineering applications show that RCDRS improves solution quality, feasibility and downstream decision performance. The code is available at \url{https://anonymous.4open.science/r/RC-DRS-180C/}.
\end{abstract}

\section{Introduction}
\label{sec:introduction}

Differentiable optimization layers are increasingly used to embed constrained
decision making into end-to-end learning systems. Such layers appear in
decision-focused learning (DFL)~\cite{mandi2024dflsurvey}, constrained
prediction, and model predictive control (MPC), where model outputs are
determined through the solution of an optimization problem
\cite{amos2017optnet,agrawal2019cvxpylayers}. In these settings, the solver is
repeatedly invoked over a distribution of related instances and is often
included in the training loop. The main issue is the performance with limited budget, i.e., the optimization layer must produce useful
decisions within a fixed small number of iterations while remaining differentiable.

We focus on conic linear programming (CLP)
\cite{ye1997interior,shapiro2001duality}, a unified framework covering linear
programming, second-order cone programming, semidefinite programming, and many
convex decision layers. CLPs can be solved by separating their affine and
conic constraints, leading to projection-based iterations composed of affine
projection, cone projection, relaxation, and dual correction. The alternating
direction method of multipliers (ADMM) and its Douglas--Rachford splitting
(DRS) interpretation provide a standard foundation for these schemes
\cite{boyd2011admm,eckstein1992drs}. Since the associated operations are almost everywhere differentiable, a fixed number of
splitting steps can be unrolled as a differentiable solver layer
\cite{rockafellar1998variational,agrawal2019diffcp}.

Fixed-depth unrolling creates a trade-off among solution quality, feasibility,
and cost. Increasing the number of iterations can improve
solution accuracy, but it increases the computation and activation memory
required in the forward and backward passes. Shallow unrolling is more
efficient but may return decisions with poor objective value, weak
feasibility, or even unstable terminal behavior. Moreover, effective
finite-budget updates depend on both the problem instance and the observed
trajectory. Early iterations may favor objective progress, whereas later
iterations may require feasibility recovery. Fixed
parameters cannot adapt across instances, and a learned layer-wise schedule
remains open-loop because it cannot react to trajectory-level residuals and
progress.

To address this issue, we propose Residual-Controlled Douglas--Rachford
Splitting (\textsc{RCDRS}), a differentiable projection-splitting layer for
CLPs. The key idea is to view an unrolled DRS solver as a
feedback-controlled dynamical system. At iteration \(k\), a causal controller
observes normalized residual, action-history, objective-progress, and
iteration-budget features from the current trajectory. It then adapts two
transition variables: 1) the relaxation parameter \(\alpha_k\), which controls
the splitting update; 2) the objective-drive parameter \(\beta_k\), which
controls the objective shift before the affine projection. The controller
therefore adjusts how the finite iteration budget is allocated while
preserving the affine projection, cone projection, relaxation, and dual-update
structure of DRS. This structure also provides an interpretable terminal
state from which primal, dual, and slack quantities can be recovered and
evaluated using KKT-style diagnostics.

We provide an operator-theoretic analysis of the proposed layer. For any
fixed admissible pair \((\alpha,\beta)\), the corresponding transition is an
averaged relaxed DRS operator for a positively scaled conic program with the
same primal solution set as the original problem. This interpretation yields
a finite-step fixed-point residual bound. We further analyze safeguarded
time-varying rollouts as summable perturbations of a limiting averaged
operator and derive a finite-step residual bound for the controlled
trajectory. In addition, we develop a terminal-oriented self-supervised
objective that learns the controller from solver residuals and objective
progress without requiring optimal-solution labels.

The main contributions are summarized as follows:
\begin{itemize}
    \item We propose \textsc{RCDRS}, a residual-feedback-controlled DRS
    layer that adapts relaxation and objective-drive parameters from causal
    trajectory features under a fixed iteration budget.

    \item We establish that fixed admissible transitions remain
    solution-preserving averaged DRS operators and that safeguarded
    time-varying rollouts can be analyzed as summable perturbations of a
    limiting averaged operator.

    \item We develop a terminal-oriented self-supervised training framework
    and evaluate the resulting solver layer on fixed-budget conic
    optimization, decision-focused learning, and forecast-aware energy
    dispatch.
\end{itemize}

\section{Related Work}
\label{sec:related_work}

\subsection{Differentiable solver layers}

Embedding constrained optimization into neural networks has followed two main routes: implicit differentiation and algorithm unrolling. Both produce differentiable layers, but they differ in how gradients are obtained and in how the solving iteration is exposed to learning.

Implicit differentiation layers solve the optimization problem approximately in the forward pass and then differentiate it through the optimality conditions. OptNet pioneered this approach by backpropagating through the KKT conditions of quadratic programs~\cite{amos2017optnet}. CVXPYLayers extended the paradigm to disciplined convex programs~\cite{agrawal2019cvxpylayers}, and diffcp enabled differentiation through cone programs via the residual map of a homogeneous self-dual embedding~\cite{agrawal2019diffcp}. These methods provide principled gradients of the solution map, but they usually require an additional linear-system or KKT-type solution in the backward pass, which can be expensive and numerically sensitive in repetitive training loops~\cite{katyal2024diffoptlayers,mandi2025surrogate,dysnet2025solverfree}.

Algorithm unrolling offers a complementary route~\cite{monga2021algorithm}. Instead of differentiating through a converged solution, it unfolds a fixed number of iterative solver steps into a computation graph and backpropagates through the stored trajectory. This idea has been widely used in sparse coding and inverse problems, from LISTA~\cite{gregor2010lista,chen2018lista} to ADMM-based unrolled networks~\cite{yang2016deepadmmnet,xie2019dladmm}. These works show that optimization algorithms can serve as trainable neural layers when their computational structure is preserved. \textsc{RCDRS} also follows the unrolling paradigm, but addresses a different question: \textit{how to improve the decision quality of a fixed-depth solver layer by adapting its splitting parameters from trajectory feedback?}

\subsection{Trajectory control of iterative solvers}

The performance of splitting methods is highly sensitive to relaxation factors, penalties, step sizes, and related algorithmic parameters. Classical adaptive schemes tune these quantities using hand-designed rules, including residual balancing~\cite{wohlberg2017residual}, spectral penalty estimation~\cite{xu2017spectraladmm}, and joint penalty-relaxation updates~\cite{xu2017aradmm}. Recent variants further consider multi-constraint penalties and constraint-wise weighting~\cite{lozenski2025multiparameter,verheijen2025superadmm}. These methods demonstrate the importance of parameter adaptation, but their update mechanisms are usually fixed by manually designed heuristics.

Learning-based solver control provides a data-driven alternative. RLQP formulates solver-parameter tuning as a reinforcement-learning problem~\cite{ichnowski2021rlqp}, while context-aware adaptive methods learn parameter updates from trajectory features~\cite{jung2022caadmm}. Other recent work studies learned over-relaxation with convergence certificates~\cite{lin2026overrelaxation} or neural approximations within iterative optimization subproblems~\cite{chen2025inexactadmm}. In contrast to the above approaches, \textsc{RCDRS} learns a closed-loop residual controller while keeping the original DRS structure. It therefore uses feedback to improve finite-budget behavior without turning the solver layer into a black-box neural predictor.

\section{Preliminaries and Problem Statement}
\label{sec:prelim}

\subsection{CLPs and DRS solver layer}

We consider a family of CLPs
\begin{equation}
\label{eq:clp}
    \min_{x} c^\top x
    ~~
    \mathrm{s.t.}~~
    Ax=b,~~ x\in\mathcal K ,
\end{equation}
where $x\in\mathbb R^n$, \(A\in\mathbb R^{m\times n}\), \(b\in\mathbb R^m\),
\(c\in\mathbb R^n\), and \(\mathcal K\subseteq\mathbb R^n\) is a
Cartesian product of projection-computable closed convex cones. A problem
instance is denoted by \(d:=(A,b,c,\mathcal K)\), and the goal is to
construct a fixed-depth differentiable solver layer that maps \(d\) to an
approximate primal decision.

Let \(\mathcal A:=\{x\in\mathbb R^n:Ax=b\}\) denote the affine feasible
set. Problem~\eqref{eq:clp} can be written as
\(\min_x c^\top x+\mathbb I_{\mathcal A}(x)+\mathbb I_{\mathcal K}(x)\),
where \(\mathbb I_{\mathcal C}\) is the indicator function of
\(\mathcal C\). Its affine projection is
\(\Proj_{\mathcal A}(v)=v-A^\top(AA^\top)^\dagger(Av-b)\), while
\(\Proj_{\mathcal K}\) is evaluated blockwise according to the product
structure of \(\mathcal K\).

We use the scaled-state form of relaxed DRS
\cite{eckstein1992drs,boyd2011admm}. The state is
\(y^k=(x^k,z^k,u^k)\), where \(x^k\) and \(z^k\) are the affine-side and
cone-side variables, and \(u^k\) is the scaled dual variable
associated with the consensus constraint \(x=z\). Given a relaxation
parameter \(\alpha\in(0,2)\) and an objective-shift scale \(\gamma>0\), one
relaxed DRS step is
\begin{equation}
\label{eq:standard}
\begin{aligned}
    x^{k+1}
    &=
    \Proj_{\mathcal A}(z^k-u^k-\gamma c),\\
    \bar x^{k+1}
    &=
    \alpha x^{k+1}+(1-\alpha)z^k,\\
    z^{k+1}
    &=
    \Proj_{\mathcal K}(\bar x^{k+1}+u^k),\\
    u^{k+1}
    &=
    u^k+\bar x^{k+1}-z^{k+1}.
\end{aligned}
\end{equation}
Here \(\gamma\) controls the objective shift, \(\alpha\) controls the
relaxed point passed to the cone projection, and
\(z^{k+1}\in\mathcal K\) holds by construction. The terminal variable
\(z^K\) is therefore used as the primal output, while \(Az^K-b\) and
\(x^K-z^K\) measure affine infeasibility and splitting inconsistency.

Let \(F_{\alpha,\gamma}\) denote the transition in
\eqref{eq:standard}. Running \(K\) steps generates
\(y^{k+1}=F_{\alpha,\gamma}(y^k;d)\), \(k=0,\ldots,K-1\), and the
corresponding solver layer returns \(\mathcal S_K(d)=z^K\). Its terminal
quality depends on the parameters used along the finite trajectory.

\subsection{Problem statement}

A fixed-depth splitting solver produces an observable trajectory
\(y^0,y^1,\ldots,y^K\) containing information about objective progress,
affine feasibility, splitting consistency, and terminal stabilization. We
seek a differentiable solver layer of the form
\(y^{k+1}=F_{\tau_k}(y^k;d)\), with
\(\tau_k=\pi_\theta(\mathcal H_k)\), where
\(\mathcal H_k\) denotes the information available up to iteration \(k\),
\(\tau_k\) is the next splitting action, and \(\pi_\theta\) is a causal
feedback policy. The objective is to improve the terminal decision \(z^K\)
under a fixed iteration budget while preserving the affine projection, cone
projection, relaxation, and dual-update structure of DRS.

\section{Methodology}
\label{sec:method}

\subsection{Residual-controlled DRS layer}
\label{subsec:method_overview}

\begin{figure}[!t]
    \centering
    \IfFileExists{fig1.png}{%
        \includegraphics[width=\columnwidth]{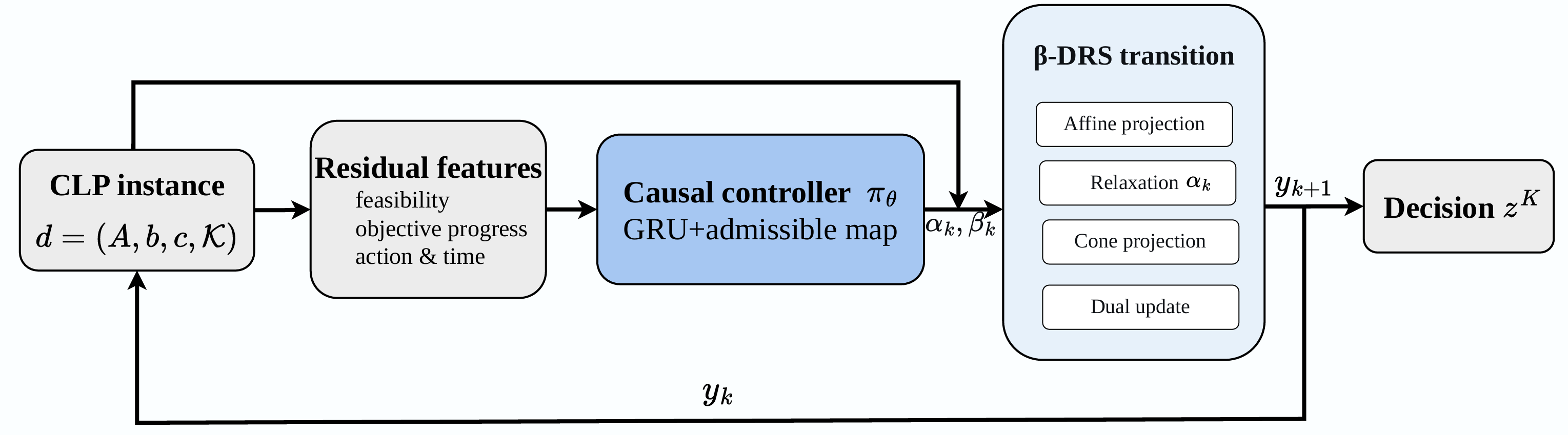}%
    }{%
        \fbox{\parbox[c][5.0cm][c]{0.78\columnwidth}{\centering
        Placeholder for the RCDRS overview figure\\[2mm]
        Replace this box by placing \texttt{fig1.png} in the manuscript folder.}}%
    }
    \caption{Overview of the proposed \textsc{RCDRS} layer. The layer
    couples a \(\beta\)-parameterized DRS transition with a causal residual
    feedback controller. The transition preserves the projection-splitting
    operations, while the controller adapts the finite-depth trajectory
    through residual, action-history, and time information.}
    \label{fig:rcdrs_overview}
\end{figure}

Fig.~\ref{fig:rcdrs_overview} illustrates the proposed \textsc{RCDRS}
layer. The key idea is to treat a fixed-depth DRS layer as a finite-horizon
controlled dynamical system. Conceptually, \textsc{RCDRS} designs a feedback
controller for the trajectory-dependent splitting parameters of DRS, with
the goal of improving terminal solution quality under a fixed iteration
budget.

The closed-loop rollout consists of two coupled maps:
\begin{subequations}
\label{eq:rcdrs_two_maps}
\begin{align}
    y^{k+1}
    &=
    F_{\vartheta_k}(y^k;d),
    ~~
    \vartheta_k=(\alpha_k,\beta_k),\\
    \omega_k
    &:=
    (\rho_k,\alpha_k,\beta_k)
    =
    \pi_\theta(\phi^{0:k}).
\end{align}
\end{subequations}
In~\eqref{eq:rcdrs_two_maps}, \(F_{\vartheta_k}\) is the numerical DRS
transition controlled by \((\alpha_k,\beta_k)\), while \(\pi_\theta\) is
a causal feedback policy that maps the feature history
\(\phi^{0:k}:=(\phi^0,\ldots,\phi^k)\) to the control vector
\(\omega_k\). The coordinate \(\rho_k>0\) is the auxiliary
controller-side scale memory for action-history encoding and regularization,
whereas \((\alpha_k,\beta_k)\) are passed to the numerical transition.

\subsection{\(\beta\)-parameterized DRS transition}
\label{subsec:beta_drs_block}

We first define the controlled transition \(F_{\vartheta_k}\). In standard
relaxed DRS~\eqref{eq:standard}, the affine projection uses a fixed
objective shift \(-\gamma c\). A fixed-depth layer benefits from a
trajectory-dependent objective drive because early iterations may favor
objective progress, whereas later iterations may require feasibility
recovery and terminal stabilization. We therefore use the normalized
objective direction \(\bar c:=c/(\|c\|_2+\varepsilon_c)\), where
\(\varepsilon_c>0\), and let \(\beta_k>0\) control the strength of the
objective shift. Given \(\vartheta_k=(\alpha_k,\beta_k)\), with
\(\alpha_k\in(0,2)\) and \(\beta_k>0\), the \(k\)-th
\(\beta\)-parameterized DRS transition is
\begin{equation}
\label{eq:rc_drs_block}
\begin{aligned}
    w^k
    &=
    z^k-u^k-\beta_k\bar c,\\
    x^{k+1}
    &=
    \Proj_{\mathcal A}(w^k),\\
    \bar x^{k+1}
    &=
    \alpha_kx^{k+1}+(1-\alpha_k)z^k,\\
    z^{k+1}
    &=
    \Proj_{\mathcal K}(\bar x^{k+1}+u^k),\\
    u^{k+1}
    &=
    u^k+\bar x^{k+1}-z^{k+1}.
\end{aligned}
\end{equation}
Compared with the fixed-core DRS layer, the fixed shift \(\gamma c\) is
replaced by the adaptive normalized shift \(\beta_k\bar c\). Thus
\(\beta_k\) controls objective motion before the affine projection, while
\(\alpha_k\) controls relaxation before the cone projection. The affine
projection, cone projection, and dual-correction operations remain
unchanged.

For a \(K\)-step rollout, the primal decision is \(z^K\). Since the cone
projection is applied at every iteration, \(z^K\in\mathcal K\) up to
numerical precision. The retained primal--dual structure also provides a
terminal diagnostic readout. With
\(\kappa_{K-1}:=\beta_{K-1}/(\|c\|_2+\varepsilon_c)\), we recover the
terminal slack and equality multiplier as
\begin{equation}
\label{eq:dualdiag}
    s^K=-u^K/\kappa_{K-1},
    ~~
    \lambda^K=(AA^\top)^\dagger A(c-s^K).
\end{equation}
The layer can therefore return the primal decision \(z^K\) together with the
diagnostic triplet
\(\operatorname{RCDRS}_\theta^K(d)=(z^K,\lambda^K,s^K)\).

\subsection{Causal residual feedback controller}
\label{subsec:causal_controller}

We now describe the controller used to select trajectory-dependent splitting
parameters. At iteration \(k\), the controller only uses information
available from the current and previous solver states and outputs
\(\omega_k=(\rho_k,\alpha_k,\beta_k)\). The transition in
\eqref{eq:rc_drs_block} uses \((\alpha_k,\beta_k)\), while \(\rho_k\) is
maintained as a positive controller-side scale memory.

The controller is organized into four functional modules:
\[
    \operatorname{Feat}
    \;\rightarrow\;
    \mathcal E_\theta
    \;\rightarrow\;
    g_\theta
    \;\rightarrow\;
    \operatorname{Adm}.
\]
Here \(\operatorname{Feat}\) converts raw trajectory information into a
scale-normalized observation vector
\[\phi^k=\operatorname{Feat}(y^k,z^{k-1},d,\omega_{k-1},k,K).\]
The features summarize three types of information: residual information
describing the current optimization state, action-history information
describing the previous update, and time-budget information describing the
current position within the fixed rollout horizon.

The encoder \(\mathcal E_\theta\) maps the feature history to a trajectory
state and is implemented as a lightweight GRU. The control head
\(g_\theta\) maps this state to unconstrained raw coordinates, and
\(\operatorname{Adm}\) transforms the raw outputs into admissible controls
satisfying \(\rho_k>0\), \(\alpha_k\in(0,2)\), and \(\beta_k>0\).
This construction gives a recurrent feedback policy followed by output corrections.

We use two controller variants. \textsc{RCDRS-NoEnv} applies the admissible
action map directly and allows more flexible finite-horizon adaptation.
\textsc{RCDRS-Env} additionally restricts the controls around a
validation-selected base action through a decaying envelope. The safeguarded
variant is used in the time-varying analysis in Section~\ref{sec:theory},
where the controlled rollout is related to a limiting averaged DRS operator.

\paragraph{Controller specification.}
For completeness, we give the exact features and control parameterization
used in the experiments. With the initialization conventions
\(z^{-1}=z^0\) and
\(\omega_{-1}=(\rho_0,\alpha_0,\beta_0)\), the normalized trajectory
quantities are
\(
    \eta_{\rm con}^k
    :=
    \frac{\|x^k-z^k\|_2}
    {1+\|z^k\|_2},
    \eta_{\rm eq}^k
    :=
    \frac{\|Az^k-b\|_2}
    {1+\|b\|_2},
    \eta_{\rm dz}^k
    :=
    \frac{\|z^k-z^{k-1}\|_2}
    {1+\|z^k\|_2},
    \eta_{\rm obj}^k
    :=
    \frac{|c^\top z^k|}
    {1+\|c\|_2\|z^k\|_2},
    \Delta_{\rm obj}^k
    :=
    \operatorname{asinh}
    \left(
    \frac{c^\top z^k-c^\top z^{k-1}}
    {1+|c^\top z^{k-1}|}
    \right).
\)
Here \(\eta_{\rm con}^k\), \(\eta_{\rm eq}^k\), and
\(\eta_{\rm dz}^k\) describe splitting consistency, affine feasibility,
and iterate movement, while \(\eta_{\rm obj}^k\) and
\(\Delta_{\rm obj}^k\) describe objective magnitude and signed progress.

The previous-action features are
\(\xi_\rho^{k-1}:=\log(\rho_{k-1}/\rho_0)\),
\(\xi_\alpha^{k-1}:=\alpha_{k-1}\), and
\(\xi_\beta^{k-1}:=\log(\beta_{k-1}/\beta_0)\). The complete observation
vector is
\begin{align*}
    &\phi^k=\big[
    \log(1+\eta_{\rm con}^k),
    \log(1+\eta_{\rm eq}^k),
    \log(1+\eta_{\rm dz}^k),\\
    &\log(1+\eta_{\rm obj}^k),
    \Delta_{\rm obj}^k,
    \xi_\rho^{k-1},
    \xi_\alpha^{k-1},
    \xi_\beta^{k-1},
    k/K,\,
    (K-k)/K
    \big].
\end{align*}
The logarithmic transformation compresses nonnegative quantities that may
vary by several orders of magnitude, while
\(\operatorname{asinh}\) compresses objective progress while preserving
its sign. Cone violation is omitted because \(z^k\in\mathcal K\) is
enforced by the cone projection.

The GRU updates its hidden state as
\(h^k=\mathcal E_\theta(\phi^k,h^{k-1})\), and the control head produces
\(r^k=(r_\rho^k,r_\alpha^k,r_\beta^k)=g_\theta(h^k)\). For
\textsc{RCDRS-NoEnv}, the raw outputs are mapped to admissible intervals by
\begin{align*}
    \log\rho_k
    &=
    \log\rho_{\min}
    +
    \sigma(r_\rho^k)
    \bigl(\log\rho_{\max}-\log\rho_{\min}\bigr),\\
    \alpha_k
    &=
    \alpha_{\min}
    +
    \sigma(r_\alpha^k)
    (\alpha_{\max}-\alpha_{\min}),\\
    \log\beta_k
    &=
    \log\beta_{\min}
    +
    \sigma(r_\beta^k)
    \bigl(\log\beta_{\max}-\log\beta_{\min}\bigr),
\end{align*}
where \(0<\rho_{\min}<\rho_{\max}\),
\(0<\alpha_{\min}<\alpha_{\max}<2\), and
\(0<\beta_{\min}<\beta_{\max}\).

For \textsc{RCDRS-Env}, let
\(\bar\omega=(\bar\rho,\bar\alpha,\bar\beta)\) be the base action. The safeguarded parameterization is
\begin{small}
\begin{equation}
\label{eq:envelope_param}
    \begin{aligned}
    \log\rho_k
    &=
    \log\bar\rho+\delta_k\tanh(r_\rho^k),
    \alpha_k
    =
    \bar\alpha+s_\alpha\delta_k\tanh(r_\alpha^k),\\
    \log\beta_k
    &=
    \log\bar\beta+\delta_k\tanh(r_\beta^k),
    \delta_k
    =
    \frac{\delta_0}{1+(k/k_0)^p}.
\end{aligned}
\end{equation}
\end{small}
The constants are selected so that
\(0<\bar\alpha-s_\alpha\delta_0
\leq\bar\alpha+s_\alpha\delta_0<2\). Since \(p>1\), the envelope in
\eqref{eq:envelope_param} satisfies
\(\sum_{k=0}^{\infty}\delta_k<\infty\), which is used in the
time-varying analysis. Optional multiplicative action filters and numerical
parameter settings are provided in the supplementary material.

\subsection{Terminal-oriented training and solver layer}
\label{subsec:terminal_training}

The proposed layer is trained with a terminal-oriented self-supervised
objective. Since deployment uses a prescribed depth \(K\), the loss is
evaluated at the terminal state rather than requiring monotone improvement
at every iteration.

For each training instance, we run both the controlled layer and a fixed-core
reference layer for the same depth. Let
\(y_{\rm rc}^K=(x_{\rm rc}^K,z_{\rm rc}^K,u_{\rm rc}^K)\) and
\(y_{\rm base}^K=(x_{\rm base}^K,z_{\rm base}^K,u_{\rm base}^K)\)
denote their terminal states. The training objective targets terminal affine
feasibility, splitting consistency, iterate stabilization, objective quality,
action smoothness, and finite-budget dominance over the reference rollout:
\[
\begin{aligned}
    \min_\theta~
    \mathbb E_{d\sim\mathcal D}
    \big[
        &\lambda_{\rm eq}\mathcal L_{\rm eq}
        +\lambda_{\rm con}\mathcal L_{\rm con}
        +\lambda_{\rm mov}\mathcal L_{\rm mov}\\
        &+\lambda_{\rm obj}\mathcal L_{\rm obj}
        +\lambda_{\rm smooth}\mathcal L_{\rm smooth}
        +\lambda_{\rm dom}\mathcal L_{\rm dom}
    \big].
\end{aligned}
\]

\paragraph{Loss definitions.}
The terminal residual losses are \(\mathcal L_{\rm eq}
    :=
    \frac{\|Az_{\rm rc}^K-b\|_2^2}
    {(1+\|b\|_2)^2},
    \mathcal L_{\rm con}
    :=
    \frac{\|x_{\rm rc}^K-z_{\rm rc}^K\|_2^2}
    {(1+\|z_{\rm rc}^K\|_2)^2},
    \mathcal L_{\rm mov}
    :=
    \frac{\|z_{\rm rc}^K-z_{\rm rc}^{K-1}\|_2^2}
    {(1+\|z_{\rm rc}^K\|_2)^2}.\)
Here \(\mathcal L_{\rm con}\) measures splitting consistency rather than
cone infeasibility, since the cone projection already enforces
\(z_{\rm rc}^K\in\mathcal K\). The baseline-relative objective loss is
\(
    \mathcal L_{\rm obj}
    :=
    \left[
    \frac{
    c^\top z_{\rm rc}^K-c^\top z_{\rm base}^K
    }
    {1+|c^\top z_{\rm base}^K|}
    \right]_+ ,
\)
which is positive only when the controlled rollout has a worse terminal
objective than the fixed-core reference at the same depth.

Let \(\ell_k:=(\log\rho_k,\alpha_k,\log\beta_k)\). The action-smoothness
term is
\(\mathcal L_{\rm smooth}:=\sum_{k=1}^{K-1}
\|\ell_k-\ell_{k-1}\|_2^2\). To define the dominance term, let
\[
    \mathcal M_{\rm rc}
    :=
    \lambda_{\rm eq}\mathcal L_{\rm eq}
    +
    \lambda_{\rm con}\mathcal L_{\rm con}
    +
    \lambda_{\rm mov}\mathcal L_{\rm mov}
    +
    \lambda_{\rm obj}\mathcal L_{\rm obj},
\]
and define the corresponding reference merit as
\(\mathcal M_{\rm base}:=
\lambda_{\rm eq}\mathcal L_{\rm eq}^{\rm base}
+\lambda_{\rm con}\mathcal L_{\rm con}^{\rm base}
+\lambda_{\rm mov}\mathcal L_{\rm mov}^{\rm base}\), where the base losses
use the same definitions evaluated on \(y_{\rm base}^K\). The optional
dominance loss is
\(\mathcal L_{\rm dom}:=
[\mathcal M_{\rm rc}-\mathcal M_{\rm base}+m]_+\), \(m\geq0\).

All loss terms are computed from the finite-depth trajectories. No reference
optimal solution, optimal value, or KKT multiplier is used as a supervised
label. After training, \textsc{RCDRS} defines the fixed-depth map
\(\operatorname{RCDRS}_\theta^K:(A,b,c,\mathcal K)\mapsto z^K\).
The map can be used directly as an optimizer or inserted into an end-to-end
learning pipeline, where an upstream model predicts problem data and the task
loss is backpropagated through the unrolled solver layer.

\section{Theoretical Analysis}
\label{sec:theory}

This section establishes the operator-theoretic properties of
\textsc{RCDRS}. We call a pair \((\alpha,\beta)\) admissible if
\(\alpha\in(0,2)\) and \(\beta>0\). The analysis has three goals. First,
we show that each fixed admissible block is a solution-preserving averaged
DRS operator and derive a terminal residual bound. Second, we analyze
safeguarded feedback rollouts as summable perturbations of a limiting
averaged operator. Third, we establish the consistency of the terminal
primal--dual readout. The fixed-block results apply to both controller
variants, whereas the time-varying analysis applies only to
\textsc{RCDRS-Env}. All proofs and auxiliary lemmas are provided in the
supplementary material.

\subsection{Fixed-block DRS interpretation}
\label{subsec:theory_fixed_block}

For any fixed \(\beta>0\), define
\(f_\beta(x):=\beta\bar c^\top x+\mathbb I_{\mathcal A}(x)\),
\(g(x):=\mathbb I_{\mathcal K}(x)\), and
\(\bar c=c/(\|c\|_2+\varepsilon_c)\). Let
\(J_{f_\beta}:=(I+\partial f_\beta)^{-1}\) and
\(J_g:=(I+\partial g)^{-1}\) be the corresponding resolvents, with
reflected resolvents \(R_{f_\beta}:=2J_{f_\beta}-I\) and
\(R_g:=2J_g-I\). We define
\(T_\beta:=\frac12(I+R_{f_\beta}R_g)\) and, for
\(\alpha\in(0,2)\),
\[
    T_{\alpha,\beta}
    :=
    (1-\alpha)I+\alpha T_\beta
    =
    I+\alpha(J_{f_\beta}R_g-J_g).
\]
The state correspondence is initialized by \(v^0=z^0+u^0\) and
\(z^0=J_g(v^0)\), which is satisfied by the standard initialization
\(z^0=u^0=0\).

\begin{proposition}[Fixed-block interpretation]
\label{prop:fixed_block_interpretation}
For any \(\beta>0\) and \(\alpha\in(0,2)\), the following statements
hold.
\begin{enumerate}
    \item The objective-scaled CLP has the same primal solution set as the
    original CLP:
    \(
        \arg\min_{x\in\mathcal A\cap\mathcal K}\beta\bar c^\top x
        =
        \arg\min_{x\in\mathcal A\cap\mathcal K}c^\top x .
    \)
    \item If \(v^k=z^k+u^k\) and \(z^k=J_g(v^k)\), then the state update
    in~\eqref{eq:rc_drs_block} with fixed \((\alpha,\beta)\) satisfies
    \(v^{k+1}:=z^{k+1}+u^{k+1}=T_{\alpha,\beta}v^k\) and
    \(z^{k+1}=J_g(v^{k+1})\).
    \item The operator \(T_{\alpha,\beta}\) is
    \(\alpha/2\)-averaged and nonexpansive.
\end{enumerate}
\end{proposition}

Proposition~\ref{prop:fixed_block_interpretation} shows that the
\(\beta\)-parameterized transition preserves the primal solution set and
the averaged DRS structure. The next result strengthens the usual
best-iterate estimate to a terminal residual bound.

\begin{proposition}[Terminal residual bound for a fixed block]
\label{prop:finite_step_fixed_residual}
Let \(\alpha\in(0,2)\), \(\beta>0\), and suppose that
\(\Fix(T_{\alpha,\beta})\neq\emptyset\). Generate
\(v^{k+1}=T_{\alpha,\beta}v^k\). Then, for every \(N\ge1\),
\[
    \|v^N-T_{\alpha,\beta}v^N\|_2^2
    \le
    \frac{\alpha}{(2-\alpha)N}
    \operatorname{dist}^2
    \bigl(v^0,\Fix(T_{\alpha,\beta})\bigr).
\]
\end{proposition}

The bound follows from the descent inequality for averaged operators and the
monotonicity of \(\|v^k-T_{\alpha,\beta}v^k\|_2\) under repeated
application of the fixed nonexpansive operator.

\subsection{Safeguarded feedback rollout}
\label{subsec:theory_time_varying}

The learned controller produces a time-varying sequence
\((\alpha_k,\beta_k)\), for which the fixed-block result does not directly
apply. We therefore analyze the safeguarded parameterization in
\eqref{eq:envelope_param}. Let \(T_k:=T_{\alpha_k,\beta_k}\) and
\(T_\infty:=T_{\bar\alpha,\bar\beta}\). With
\(q_{\alpha,k}:=\tanh(r_\alpha^k)\) and
\(q_{\beta,k}:=\tanh(r_\beta^k)\), the envelope gives
\(\alpha_k=\bar\alpha+s_\alpha\delta_kq_{\alpha,k}\) and
\(\log\beta_k=\log\bar\beta+\delta_kq_{\beta,k}\), where
\(q_{\alpha,k},q_{\beta,k}\in[-1,1]\) and
\(\sum_{k=0}^{\infty}\delta_k<\infty\). Thus \(T_k\) approaches the
limiting admissible operator \(T_\infty\).

\begin{theorem}[Residual bounds under the safeguarded envelope]
\label{thm:finite_step_envelope}
Let \(v^{k+1}=T_kv^k\), suppose that
\(\Fix(T_\infty)\neq\emptyset\), the trajectory \(\{v^k\}\) is bounded,
and the envelope in~\eqref{eq:envelope_param} holds. Then there exist
summable nonnegative sequences \(\{\epsilon_k\}\) and \(\{\xi_k\}\) such
that \(\|T_kv^k-T_\infty v^k\|_2\le\epsilon_k\). For any
\(v^\star\in\Fix(T_\infty)\), let
\(c_\infty:=(2-\bar\alpha)/\bar\alpha\) and
\(D_\infty:=\|v^0-v^\star\|_2^2+\sum_{k=0}^{\infty}\xi_k\). Then, for
every \(N\ge1\), \(\min_{0\le k<N}\|v^k-T_\infty v^k\|_2^2\le\frac{D_\infty}{c_\infty N},\|v^N-T_\infty v^N\|_2\le
\sqrt{\frac{2D_\infty}{c_\infty N}}+2\sum_{k=\lfloor N/2\rfloor}^{\infty}\epsilon_k .\)

\end{theorem}

The perturbation satisfies \(\epsilon_k=O(\delta_k)\). For
\(\delta_k=\delta_0/(1+(k/k_0)^p)\), this gives
\(\|v^N-T_\infty v^N\|_2
=O(N^{-1/2}+N^{1-p})\); in particular, the squared terminal residual is
\(O(1/N)\) when \(p\ge3/2\). The theorem shows that safeguarded adaptation
retains an asymptotic connection to a limiting averaged DRS operator. This
rollout-level result applies to \textsc{RCDRS-Env};
\textsc{RCDRS-NoEnv} retains only the fixed-block properties.

\subsection{Terminal primal--dual diagnostics}
\label{subsec:theory_primal_dual_certificate}

The retained DRS structure also provides a terminal primal--dual readout.
For a \(K\)-step rollout, the last transition uses \(\beta_{K-1}\), and
the terminal slack and equality multiplier are recovered by
\eqref{eq:dualdiag}. The solver layer returns \(z^K\) as the primal
decision and \((\lambda^K,s^K)\) as diagnostic quantities.

We use the normalized residuals
\begin{small}
\begin{equation}
\label{eq:diags}
\begin{aligned}
    r_{\rm p}^K
    &:=
    \frac{\|Az^K-b\|_2}{1+\|b\|_2},~~
    r_{\rm d}^K
    :=
    \frac{\|A^\top\lambda^K+s^K-c\|_2}{1+\|c\|_2},\\
    r_{\rm gap}^K
    &:=
    \frac{|c^\top z^K-b^\top\lambda^K|}
    {1+|c^\top z^K|+|b^\top\lambda^K|}.
\end{aligned}
\end{equation}
\end{small}
These quantities measure primal affine feasibility, dual stationarity, and
the primal--dual gap.

\begin{theorem}[Consistency of the terminal primal--dual readout]
\label{thm:terminal_primal_dual_certificate}
Assume that \(\mathcal K\) is a closed convex cone and that the cone
projection is evaluated exactly. For every terminal state generated by
\eqref{eq:rc_drs_block}, the recovered slack satisfies
\(s^K\in\mathcal K^\ast\) and
\(\langle z^K,s^K\rangle=0\). Moreover, for a fixed CLP, consider a
bounded sequence of terminal readouts
\(\{(z_j,\lambda_j,s_j)\}_{j\ge1}\). If the corresponding residuals
satisfy
\(
r_{{\rm p},j}\to0,\;
r_{{\rm d},j}\to0,\;
r_{{\rm gap},j}\to0,
\)
then every cluster point
\((z^\star,\lambda^\star,s^\star)\) is primal--dual optimal for the
original CLP.
\end{theorem}

Together, these results establish the theoretical role of the proposed
controller. Fixed admissible actions remain solution-preserving averaged
DRS blocks, safeguarded feedback rollouts remain controlled perturbations of
a limiting averaged operator, and vanishing terminal diagnostics recover
primal--dual optimality.

\section{Computational Experiments}
\label{sec:experiments}

The experiments examine \textit{whether residual feedback improves fixed-budget
optimization, which components produce the gains, whether these gains
transfer to downstream decision learning, and whether the resulting layer
improves operational decisions.} Unless otherwise stated, results are
mean\(\pm\)standard deviation over three seeds. All methods use the same data
splits, initial states, and evaluation depths. Parameters and checkpoints are
selected on validation data, and learned layers are trained and evaluated at
the same depth. Reference optima are used only for validation and evaluation,
never as training labels. Further settings and results are provided in the
supplementary material.

\paragraph{Compared methods.}
The baselines include fixed \textsc{F-DRS} and validation-tuned
\textsc{Grid-DRS}, both based on relaxed DRS~\cite{eckstein1992drs};
spectral \textsc{S-ADMM}~\cite{xu2017spectraladmm}; safeguarded
Anderson-accelerated \textsc{AA-DRS}~\cite{fu2019a2dr}; learned
\textsc{L-PDHG}, based on primal--dual splitting and learned primal--dual
unrolling~\cite{chambolle2011first,adler2018learned}; and open-loop
\textsc{Layerwise-DRS}, following the algorithm-unrolling
paradigm~\cite{monga2021algorithm}. \textsc{RCDRS-Env} uses a
validation-centered decaying envelope, whereas \textsc{RCDRS-NoEnv} uses
the same feedback controller without the envelope. In DFL, \textsc{PTO}
denotes the prediction-then-optimization baseline~\cite{elmachtoub2022spo};
suffixes \textsc{-F} and \textsc{-J} denote a frozen controller and joint
predictor--controller fine-tuning, respectively.

\paragraph{Common diagnostics.}
For terminal decision \(z^K\) and primal--dual readout
\((\lambda^K,s^K)\), we report
\(\mathrm{ObjErr}:=|c^\top z^K-p^\star|/(1+|p^\star|)\) together with
\(r_{\rm p},r_{\rm d},r_{\rm gap}\) from~\eqref{eq:diags}. Since infeasible
points may have favorable objectives, objective metrics are interpreted
jointly with these residuals. Online time is measured in milliseconds per instance
and includes features, controller evaluation, projections, and state updates,
but excludes training and validation tuning.

\subsection{Fixed-Budget Conic Optimization}
\label{subsec:synthetic_experiments}

We evaluate two hard CLP families: an SOCP benchmark and a
mixed-cone benchmark combining nonnegative, second-order, rotated
second-order, and positive-semidefinite blocks. We report \(K=20\) here; other results are given in supplementary material.

\begin{table*}[!t]
\centering
\scriptsize
\setlength{\tabcolsep}{2.2pt}
\renewcommand{\arraystretch}{0.92}
\caption{Fixed-budget conic optimization at \(K=20\). Values are
mean\(\pm\)standard deviation over three random seeds. The best
solution-quality value within each benchmark is bold.}
\label{tab:synthetic_main_k20}
\resizebox{0.8\textwidth}{!}{
\begin{tabular}{llccccc}
\toprule
Benchmark & Method
& \(\mathrm{ObjErr}\downarrow\)
& \(r_{\rm p}\downarrow\)
& \(r_{\rm d}\downarrow\)
& \(r_{\rm gap}\downarrow\)
& Time (ms) \(\downarrow\) \\
\midrule

\multirow{8}{*}{SOCP}
& \textsc{F-DRS}
& \(3.15\mathrm{e}{-1}\!\pm\!3.07\mathrm{e}{-2}\)
& \(2.28\mathrm{e}{-3}\!\pm\!1.18\mathrm{e}{-4}\)
& \(1.35\mathrm{e}{-1}\!\pm\!8.99\mathrm{e}{-3}\)
& \(2.08\mathrm{e}{-1}\!\pm\!8.13\mathrm{e}{-2}\)
& \(0.042\!\pm\!0.015\) \\
& \textsc{Grid-DRS}
& \(1.56\mathrm{e}{-3}\!\pm\!6.11\mathrm{e}{-5}\)
& \(3.08\mathrm{e}{-3}\!\pm\!1.71\mathrm{e}{-4}\)
& \(9.49\mathrm{e}{-3}\!\pm\!5.15\mathrm{e}{-4}\)
& \(1.90\mathrm{e}{-2}\!\pm\!2.93\mathrm{e}{-3}\)
& \(0.053\!\pm\!0.031\) \\
& \textsc{S-ADMM}
& \(2.07\mathrm{e}{-3}\!\pm\!1.60\mathrm{e}{-4}\)
& \(3.20\mathrm{e}{-3}\!\pm\!1.10\mathrm{e}{-4}\)
& \(4.99\mathrm{e}{-3}\!\pm\!1.39\mathrm{e}{-3}\)
& \(1.23\mathrm{e}{-2}\!\pm\!7.98\mathrm{e}{-4}\)
& \(0.096\!\pm\!0.022\) \\
& \textsc{AA-DRS}
& \(1.13\mathrm{e}{-3}\!\pm\!1.14\mathrm{e}{-5}\)
& \(1.60\mathrm{e}{-3}\!\pm\!6.04\mathrm{e}{-5}\)
& \(5.61\mathrm{e}{-3}\!\pm\!1.23\mathrm{e}{-4}\)
& \(1.69\mathrm{e}{-2}\!\pm\!1.33\mathrm{e}{-3}\)
& \(0.286\!\pm\!0.048\) \\
& \textsc{L-PDHG}
& \(2.42\mathrm{e}{-1}\!\pm\!4.62\mathrm{e}{-3}\)
& \(3.99\mathrm{e}{-2}\!\pm\!2.30\mathrm{e}{-3}\)
& \(3.33\mathrm{e}{-1}\!\pm\!1.96\mathrm{e}{-2}\)
& \(4.11\mathrm{e}{-1}\!\pm\!3.37\mathrm{e}{-2}\)
& \(0.197\!\pm\!0.043\) \\
& \textsc{Layerwise-DRS}
& \(5.66\mathrm{e}{-4}\!\pm\!1.31\mathrm{e}{-4}\)
& \(1.71\mathrm{e}{-3}\!\pm\!3.05\mathrm{e}{-4}\)
& \(3.56\mathrm{e}{-3}\!\pm\!7.06\mathrm{e}{-4}\)
& \(7.54\mathrm{e}{-3}\!\pm\!1.50\mathrm{e}{-3}\)
& \(0.121\!\pm\!0.043\) \\
& \textsc{RCDRS-Env}
& \(2.38\mathrm{e}{-4}\!\pm\!1.91\mathrm{e}{-6}\)
& \(6.92\mathrm{e}{-4}\!\pm\!3.00\mathrm{e}{-5}\)
& \(5.13\mathrm{e}{-4}\!\pm\!7.69\mathrm{e}{-6}\)
& \(\mathbf{1.42\mathrm{e}{-3}\!\pm\!1.92\mathrm{e}{-4}}\)
& \(0.143\!\pm\!0.059\) \\
& \textsc{RCDRS-NoEnv}
& \(\mathbf{2.33\mathrm{e}{-4}\!\pm\!7.31\mathrm{e}{-6}}\)
& \(\mathbf{6.84\mathrm{e}{-4}\!\pm\!1.48\mathrm{e}{-5}}\)
& \(\mathbf{5.10\mathrm{e}{-4}\!\pm\!2.37\mathrm{e}{-5}}\)
& \(\mathbf{1.42\mathrm{e}{-3}\!\pm\!1.97\mathrm{e}{-4}}\)
& \(0.153\!\pm\!0.049\) \\

\midrule

\multirow{8}{*}{Mixed cone}
& \textsc{F-DRS}
& \(8.64\mathrm{e}{-1}\!\pm\!1.29\mathrm{e}{-2}\)
& \(1.69\mathrm{e}{-3}\!\pm\!1.35\mathrm{e}{-4}\)
& \(1.65\mathrm{e}{-1}\!\pm\!2.35\mathrm{e}{-3}\)
& \(2.43\mathrm{e}{-1}\!\pm\!5.03\mathrm{e}{-2}\)
& \(0.279\!\pm\!0.064\) \\
& \textsc{Grid-DRS}
& \(2.44\mathrm{e}{-1}\!\pm\!5.28\mathrm{e}{-3}\)
& \(2.50\mathrm{e}{-3}\!\pm\!2.08\mathrm{e}{-4}\)
& \(6.33\mathrm{e}{-2}\!\pm\!5.41\mathrm{e}{-4}\)
& \(2.15\mathrm{e}{-1}\!\pm\!8.17\mathrm{e}{-3}\)
& \(0.274\!\pm\!0.066\) \\
& \textsc{S-ADMM}
& \(1.63\mathrm{e}{-1}\!\pm\!3.59\mathrm{e}{-3}\)
& \(3.13\mathrm{e}{-3}\!\pm\!1.94\mathrm{e}{-4}\)
& \(4.81\mathrm{e}{-2}\!\pm\!2.43\mathrm{e}{-4}\)
& \(1.86\mathrm{e}{-1}\!\pm\!7.26\mathrm{e}{-3}\)
& \(0.339\!\pm\!0.088\) \\
& \textsc{AA-DRS}
& \(2.56\mathrm{e}{-1}\!\pm\!4.75\mathrm{e}{-3}\)
& \(1.87\mathrm{e}{-3}\!\pm\!1.03\mathrm{e}{-4}\)
& \(6.15\mathrm{e}{-2}\!\pm\!7.06\mathrm{e}{-4}\)
& \(2.71\mathrm{e}{-1}\!\pm\!2.74\mathrm{e}{-3}\)
& \(1.067\!\pm\!0.283\) \\
& \textsc{L-PDHG}
& \(3.12\mathrm{e}{-1}\!\pm\!4.21\mathrm{e}{-3}\)
& \(4.72\mathrm{e}{-2}\!\pm\!3.39\mathrm{e}{-3}\)
& \(2.88\mathrm{e}{-1}\!\pm\!5.17\mathrm{e}{-3}\)
& \(4.36\mathrm{e}{-1}\!\pm\!5.55\mathrm{e}{-2}\)
& \(0.501\!\pm\!0.104\) \\
& \textsc{Layerwise-DRS}
& \(2.35\mathrm{e}{-1}\!\pm\!9.52\mathrm{e}{-3}\)
& \(1.72\mathrm{e}{-3}\!\pm\!1.38\mathrm{e}{-4}\)
& \(6.25\mathrm{e}{-2}\!\pm\!1.10\mathrm{e}{-3}\)
& \(2.00\mathrm{e}{-1}\!\pm\!1.78\mathrm{e}{-2}\)
& \(0.426\!\pm\!0.086\) \\
& \textsc{RCDRS-Env}
& \(1.67\mathrm{e}{-1}\!\pm\!6.02\mathrm{e}{-3}\)
& \(7.19\mathrm{e}{-4}\!\pm\!5.41\mathrm{e}{-5}\)
& \(\mathbf{7.77\mathrm{e}{-3}\!\pm\!1.63\mathrm{e}{-3}}\)
& \(\mathbf{9.12\mathrm{e}{-2}\!\pm\!6.02\mathrm{e}{-3}}\)
& \(0.456\!\pm\!0.100\) \\
& \textsc{RCDRS-NoEnv}
& \(\mathbf{1.60\mathrm{e}{-1}\!\pm\!4.53\mathrm{e}{-3}}\)
& \(\mathbf{5.99\mathrm{e}{-4}\!\pm\!4.66\mathrm{e}{-5}}\)
& \(9.14\mathrm{e}{-3}\!\pm\!1.40\mathrm{e}{-3}\)
& \(9.55\mathrm{e}{-2}\!\pm\!2.68\mathrm{e}{-3}\)
& \(0.400\!\pm\!0.058\) \\

\bottomrule
\end{tabular}}
\end{table*}

Table~\ref{tab:synthetic_main_k20} shows that residual-controlled DRS
improves fixed-budget conic optimization under both homogeneous and
heterogeneous cone structures. On SOCP, both \textsc{RCDRS} variants
outperform all baselines across the four solution-quality metrics;
\textsc{RCDRS-NoEnv} attains the lowest mean objective, primal, and dual
errors, while the two variants are tied in the rounded gap residual. On
mixed-cone instances, \textsc{RCDRS-NoEnv} gives the lowest mean objective
and primal residuals, whereas \textsc{RCDRS-Env} gives the lowest dual and
gap residuals. This division reflects the finite-budget trade-off between
more aggressive objective progress and the safeguarded primal--dual
trajectory. The controller introduces additional online cost relative to
fixed DRS, but both variants remain faster than \textsc{AA-DRS} and
\textsc{L-PDHG} on these benchmarks.

\subsection{Ablation Studies}
\label{subsec:ablation_experiments}

We isolate the controller, controlled variables, features, and training loss
at \(K=10\). Full ablations are
provided in the supplementary material.

\begin{table*}[!t]
\centering
\scriptsize
\setlength{\tabcolsep}{2.4pt}
\renewcommand{\arraystretch}{0.90}
\caption{Main ablation summary at \(K=10\). Values are
mean\(\pm\)standard deviation over three random seeds.}
\label{tab:main_ablation_summary}
\begin{tabular}{llcccc}
\toprule
Group & Variant
& \multicolumn{2}{c}{SOCP}
& \multicolumn{2}{c}{Mixed cone} \\
\cmidrule(lr){3-4}
\cmidrule(lr){5-6}
& & \(\mathrm{ObjErr}\downarrow\) & \(r_{\rm p}\downarrow\)
& \(\mathrm{ObjErr}\downarrow\) & \(r_{\rm p}\downarrow\) \\
\midrule

\multirow{3}{*}{Controller}
& Layerwise schedule
& \(0.01457\pm0.00033\)
& \(0.01092\pm0.00050\)
& \(0.10324\pm0.00053\)
& \(0.02015\pm0.00105\) \\

& MLP-current
& \(0.00604\pm0.00034\)
& \(0.00554\pm0.00017\)
& \(0.02552\pm0.00153\)
& \(0.01866\pm0.00117\) \\


& GRU full
& \(0.00413\pm0.00033\)
& \(0.00553\pm0.00008\)
& \(0.01564\pm0.00173\)
& \(0.01863\pm0.00121\) \\

\midrule

\multirow{4}{*}{Control}
& Fixed core
& \(0.08661\pm0.00105\)
& \(0.01376\pm0.00088\)
& \(0.11907\pm0.00110\)
& \(0.02338\pm0.00129\) \\

& \(\alpha\)-only
& \(0.02648\pm0.00023\)
& \(0.01303\pm0.00078\)
& \(0.02561\pm0.00167\)
& \(0.01980\pm0.00115\) \\

& \(\beta\)-only
& \(0.00459\pm0.00023\)
& \(0.00601\pm0.00018\)
& \(0.01991\pm0.00091\)
& \(0.01968\pm0.00123\) \\

& \(\alpha+\beta\)
& \(0.00413\pm0.00032\)
& \(0.00553\pm0.00010\)
& \(0.01563\pm0.00074\)
& \(0.01863\pm0.00122\) \\

& Full \(\alpha+\beta+\rho\)
& \(0.00401\pm0.00062\)
& \(0.00417\pm0.00008\)
& \(0.01520\pm0.00167\)
& \(0.01795\pm0.00109\) \\
\midrule

\multirow{4}{*}{Feature}
& Residual-only
& \(0.00525\pm0.00027\)
& \(0.00565\pm0.00014\)
& \(0.02844\pm0.00175\)
& \(0.01882\pm0.00123\) \\

& No time
& \(0.00449\pm0.00022\)
& \(0.00555\pm0.00017\)
& \(0.02547\pm0.00170\)
& \(0.01870\pm0.00126\) \\

& No previous action
& \(0.00417\pm0.00036\)
& \(0.00555\pm0.00004\)
& \(0.02533\pm0.00165\)
& \(0.01864\pm0.00123\) \\

 
& Full
& \(0.00413\pm0.00033\)
& \(0.00553\pm0.00008\)
& \(0.01564\pm0.00167\)
& \(0.01863\pm0.00120\) \\

\midrule

\multirow{4}{*}{Loss}
& Full loss
& \(0.00413\pm0.00033\)
& \(0.00553\pm0.00008\)
& \(0.01564\pm0.00173\)
& \(0.01863\pm0.00121\) \\

& No consistency term
& \(0.00399\pm0.00037\)
& \(0.01001\pm0.00008\)
& \(0.04599\pm0.01341\)
& \(0.02154\pm0.00083\) \\

& No objective term
& \(0.00507\pm0.00034\)
& \(0.00553\pm0.00008\)
& \(0.03514\pm0.00213\)
& \(0.01861\pm0.00121\) \\


& No movement term
& \(0.00513\pm0.00039\)
& \(0.00553\pm0.00008\)
& \(0.02554\pm0.00162\)
& \(0.01862\pm0.00121\) \\

& No smoothness term
& \(0.00492\pm0.00041\)
& \(0.00553\pm0.00009\)
& \(0.02604\pm0.00170\)
& \(0.01863\pm0.00120\) \\

\bottomrule
\end{tabular}
\end{table*}

Table~\ref{tab:main_ablation_summary} shows the results. A layer-wise open-loop schedule is
substantially weaker than the feedback controllers, and the recurrent
controller further improves the mixed-cone objective over the current-state
MLP. The \(\beta\)-controlled objective drive produces most of the objective
improvement, while joint \(\alpha+\beta\) control gives the strongest
cross-family performance. The full feature set is particularly beneficial
on mixed-cone instances: removing time or previous-action information
increases objective error while leaving primal residuals nearly unchanged.
Finally, the full training loss provides the best overall
objective--feasibility balance; removing consistency substantially degrades
feasibility, whereas removing the objective or dominance term mainly
degrades objective quality.

\subsection{Decision-Focused Learning}
\label{subsec:dfl_experiments}

We next embed each solver layer in a DFL pipeline. An upstream model predicts
the cost of a mixed-cone problem, and the fixed-depth layer returns a decision
evaluated under the true downstream cost. We report \(K=20\) here and provide
\(K=5,10,15\) results in the supplementary material.

\begin{table*}[!t]
\centering
\scriptsize
\setlength{\tabcolsep}{2.8pt}
\renewcommand{\arraystretch}{0.96}
\caption{Decision-focused learning results at \(K=20\). Regret and
Dist. evaluate the decision under the true cost, whereas \(r_{\rm p}\),
\(r_{\rm d}\), and \(r_{\rm gap}\) evaluate the predicted optimization
problem. Values are mean\(\pm\)standard deviation over three random seeds.}
\label{tab:dfl_main_k20}
\resizebox{\textwidth}{!}{
\begin{tabular}{lcccccc}
\toprule
Method
& Regret \(\downarrow\)
& \(r_{\rm p}\downarrow\)
& \(r_{\rm d}\downarrow\)
& \(r_{\rm gap}\downarrow\)
& Dist. \(\downarrow\)
& Time (ms) \(\downarrow\) \\
\midrule

\textsc{PTO}
& \(0.1053\pm0.0016\)
& \(2.88\mathrm{e}{-3}\pm1.65\mathrm{e}{-4}\)
& \(3.29\mathrm{e}{-3}\pm6.86\mathrm{e}{-4}\)
& \(1.04\mathrm{e}{-2}\pm5.21\mathrm{e}{-3}\)
& \(0.4751\pm0.0330\)
& \(0.0718\pm0.0066\) \\

\textsc{F-DRS}
& \(0.0511\pm0.0037\)
& \(1.56\mathrm{e}{-3}\pm1.66\mathrm{e}{-4}\)
& \(3.13\mathrm{e}{-3}\pm4.18\mathrm{e}{-4}\)
& \(6.12\mathrm{e}{-3}\pm1.86\mathrm{e}{-3}\)
& \(0.2813\pm0.0281\)
& \(0.0741\pm0.0066\) \\

\textsc{Grid-DRS}
& \(0.0642\pm0.0236\)
& \(1.62\mathrm{e}{-3}\pm1.83\mathrm{e}{-4}\)
& \(3.11\mathrm{e}{-3}\pm1.68\mathrm{e}{-4}\)
& \(7.25\mathrm{e}{-3}\pm5.95\mathrm{e}{-4}\)
& \(0.3062\pm0.0522\)
& \(\mathbf{0.0717\pm0.0098}\) \\

\textsc{AA-DRS}
& \(0.3366\pm0.0630\)
& \(3.30\mathrm{e}{-2}\pm2.43\mathrm{e}{-3}\)
& \(3.45\mathrm{e}{-2}\pm8.95\mathrm{e}{-4}\)
& \(4.35\mathrm{e}{-2}\pm1.08\mathrm{e}{-2}\)
& \(0.5447\pm0.0349\)
& \(0.1637\pm0.0077\) \\

\textsc{L-PDHG}
& \(0.1088\pm0.0099\)
& \(2.67\mathrm{e}{-2}\pm2.37\mathrm{e}{-3}\)
& \(5.28\mathrm{e}{-2}\pm1.08\mathrm{e}{-3}\)
& \(1.40\mathrm{e}{-1}\pm4.91\mathrm{e}{-2}\)
& \(0.4428\pm0.0504\)
& \(0.0757\pm0.0068\) \\

\textsc{RCDRS-Env-F}
& \(0.0501\pm0.0082\)
& \(1.48\mathrm{e}{-3}\pm2.27\mathrm{e}{-4}\)
& \(3.24\mathrm{e}{-3}\pm2.57\mathrm{e}{-4}\)
& \(6.93\mathrm{e}{-3}\pm1.36\mathrm{e}{-3}\)
& \(0.2809\pm0.0352\)
& \(0.1193\pm0.0046\) \\

\textsc{RCDRS-NoEnv-F}
& \(0.0465\pm0.0019\)
& \(1.58\mathrm{e}{-3}\pm1.46\mathrm{e}{-4}\)
& \(4.58\mathrm{e}{-3}\pm1.18\mathrm{e}{-4}\)
& \(1.18\mathrm{e}{-2}\pm7.07\mathrm{e}{-4}\)
& \(0.2616\pm0.0214\)
& \(0.1243\pm0.0070\) \\

\textsc{RCDRS-Env-J}
& \(0.0418\pm0.0017\)
& \(\mathbf{1.30\mathrm{e}{-3}\pm4.16\mathrm{e}{-5}}\)
& \(\mathbf{1.46\mathrm{e}{-3}\pm4.67\mathrm{e}{-5}}\)
& \(\mathbf{3.58\mathrm{e}{-3}\pm1.76\mathrm{e}{-4}}\)
& \(0.2292\pm0.0162\)
& \(0.1261\pm0.0158\) \\

\textsc{RCDRS-NoEnv-J}
& \(\mathbf{0.0406\pm0.0018}\)
& \(1.39\mathrm{e}{-3}\pm5.31\mathrm{e}{-5}\)
& \(1.68\mathrm{e}{-3}\pm2.59\mathrm{e}{-4}\)
& \(4.55\mathrm{e}{-3}\pm1.06\mathrm{e}{-4}\)
& \(\mathbf{0.2274\pm0.0112}\)
& \(0.1112\pm0.0102\) \\

\bottomrule
\end{tabular}}
\end{table*}

Table~\ref{tab:dfl_main_k20} shows that solver-layer quality matters for
end-to-end decision learning. The jointly trained \textsc{RCDRS} variants
give the two lowest mean regret and decision-distance values.
\textsc{RCDRS-NoEnv-J} attains the lowest means for these downstream
metrics, while \textsc{RCDRS-Env-J} gives the lowest primal, dual, and gap
residuals for the predicted optimization problem. Joint fine-tuning also
improves both variants over their frozen-controller counterparts. These
results show that residual-controlled finite-depth solving can improve
decision quality while retaining interpretable primal--dual diagnostics,
at the cost of additional controller evaluation time.

\subsection{Forecast-Aware Energy Dispatch}
\label{subsec:energy_experiments}

Finally, we consider 24-hour dispatch on a six-bus, seven-line network with
two conventional generators, two batteries, and three renewable sites.
Schedules use forecasted load and generation but are evaluated on realized
profiles under network, storage, reserve, curtailment, and load-shedding
constraints. We report positive normalized cost gap, equality residual,
shedding, curtailment, reserve shortfall, emergency shedding, operational
pass rate, and runtime. Pass is the fraction of realized profiles satisfying
all physical acceptance tests. The full formulation and \(K=5,10\) results
are provided in the supplementary material.

\begin{table*}[!t]
\centering
\scriptsize
\setlength{\tabcolsep}{2.2pt}
\renewcommand{\arraystretch}{0.96}
\caption{Forecast-aware energy-dispatch results at \(K=15\). Values are
mean\(\pm\)standard deviation over three random seeds.}
\label{tab:energy_main_k15}
\resizebox{\textwidth}{!}{
\begin{tabular}{lcccccccc}
\toprule
Method
& \(\mathrm{CostGap}\downarrow\)
& Eq. \(\downarrow\)
& Shed \(\downarrow\)
& Curt. \(\downarrow\)
& Reserve \(\downarrow\)
& Emerg. \(\downarrow\)
& Pass \(\uparrow\)
& Time (ms) \(\downarrow\) \\
\midrule

\textsc{F-DRS}
& \(0.074\pm0.002\)
& \(1.37\mathrm{e}{-2}\pm9.9\mathrm{e}{-5}\)
& \(2.94\mathrm{e}{-2}\pm3.3\mathrm{e}{-4}\)
& \(0.821\pm0.003\)
& \(2.12\mathrm{e}{-2}\pm5.4\mathrm{e}{-4}\)
& \(2.59\mathrm{e}{-2}\pm2.4\mathrm{e}{-4}\)
& \(0.996\pm0.002\)
& \(0.031\pm0.015\) \\

\textsc{Grid-DRS}
& \(0.029\pm0.000\)
& \(1.15\mathrm{e}{-2}\pm6.7\mathrm{e}{-5}\)
& \(2.18\mathrm{e}{-2}\pm3.6\mathrm{e}{-5}\)
& \(0.781\pm0.004\)
& \(1.89\mathrm{e}{-2}\pm2.5\mathrm{e}{-4}\)
& \(2.13\mathrm{e}{-2}\pm1.7\mathrm{e}{-5}\)
& \(\mathbf{1.000\pm0.000}\)
& \(\mathbf{0.028\pm0.005}\) \\

\textsc{S-ADMM}
& \(0.139\pm0.002\)
& \(1.43\mathrm{e}{-2}\pm1.3\mathrm{e}{-4}\)
& \(3.35\mathrm{e}{-2}\pm2.6\mathrm{e}{-4}\)
& \(0.842\pm0.003\)
& \(2.46\mathrm{e}{-2}\pm5.4\mathrm{e}{-4}\)
& \(2.50\mathrm{e}{-2}\pm2.1\mathrm{e}{-4}\)
& \(0.989\pm0.006\)
& \(0.041\pm0.005\) \\

\textsc{AA-DRS}
& \(0.029\pm0.000\)
& \(1.20\mathrm{e}{-2}\pm8.3\mathrm{e}{-5}\)
& \(2.22\mathrm{e}{-2}\pm6.8\mathrm{e}{-5}\)
& \(0.782\pm0.004\)
& \(1.82\mathrm{e}{-2}\pm2.1\mathrm{e}{-4}\)
& \(2.18\mathrm{e}{-2}\pm5.1\mathrm{e}{-5}\)
& \(\mathbf{1.000\pm0.000}\)
& \(0.064\pm0.003\) \\

\textsc{RCDRS-Env}
& \(\mathbf{8.6\mathrm{e}{-3}\pm6.0\mathrm{e}{-4}}\)
& \(1.05\mathrm{e}{-2}\pm6.2\mathrm{e}{-5}\)
& \(\mathbf{5.30\mathrm{e}{-3}\pm1.8\mathrm{e}{-4}}\)
& \(\mathbf{0.685\pm0.003}\)
& \(\mathbf{2.00\mathrm{e}{-3}\pm1.3\mathrm{e}{-4}}\)
& \(\mathbf{5.29\mathrm{e}{-3}\pm1.7\mathrm{e}{-4}}\)
& \(\mathbf{1.000\pm0.000}\)
& \(0.074\pm0.015\) \\

\textsc{RCDRS-NoEnv}
& \(1.05\mathrm{e}{-2}\pm3.0\mathrm{e}{-4}\)
& \(\mathbf{1.02\mathrm{e}{-2}\pm5.2\mathrm{e}{-5}}\)
& \(5.76\mathrm{e}{-3}\pm6.0\mathrm{e}{-5}\)
& \(0.695\pm0.001\)
& \(2.55\mathrm{e}{-3}\pm1.9\mathrm{e}{-4}\)
& \(5.75\mathrm{e}{-3}\pm6.8\mathrm{e}{-5}\)
& \(\mathbf{1.000\pm0.000}\)
& \(0.072\pm0.005\) \\

\bottomrule
\end{tabular}}
\end{table*}

Table~\ref{tab:energy_main_k15} shows that residual feedback improves
operational decisions under forecast uncertainty and network constraints.
\textsc{RCDRS-Env} achieves the best operating-cost gap, load-shedding
rate, curtailment rate, reserve-shortfall rate, and emergency-shedding rate,
while \textsc{RCDRS-NoEnv} gives the lowest equality residual.
\textsc{Grid-DRS}, \textsc{AA-DRS}, and both \textsc{RCDRS} variants
reach a perfect pass rate; the contribution of residual feedback is therefore
the reduction in cost and violation magnitudes while preserving operational
acceptability. This improvement requires additional computation---approximately
\(0.072\)--\(0.074\) milliseconds per instance versus \(0.028\) milliseconds for
\textsc{Grid-DRS}.

Overall, the experiments support three conclusions. First, \textsc{RCDRS}
improves fixed-depth conic optimization under terminal primal--dual
diagnostics, with moderate controller overhead. Second, closed-loop
trajectory feedback is more effective than fixed and open-loop schedules;
the objective-drive variable accounts for most of the gain, while joint
relaxation and objective-drive control gives the most robust configuration.
Third, the resulting layer transfers to downstream tasks: joint training
reduces regret and decision distance in DFL, while residual feedback reduces
cost and violation magnitudes in forecast-aware dispatch.

\section{Conclusion and Limitation}
\label{sec:conclusion}

This paper presented \textsc{RCDRS}, a fixed-depth differentiable conic solver layer that combines $\beta$-parameterized DRS with residual-based feedback control to improve finite-budget decision quality. The current theory covers fixed admissible blocks and safeguarded time-varying rollouts. The controller is trained for a prescribed depth and problem distribution, so cross-depth generalization and out-of-distribution robustness remain important open issues. The implementation also relies on projection-computable cones and efficient affine projections, which may limit scalability for very large sparse systems. Future work will focus on stronger robustness guarantees, more scalable projection implementations, and adaptive deployment across broader problem classes.

\bibliography{aaai2027}

\appendix
\onecolumn

\section{Self-Contained Problem Setup}
\label{sec:supp_setup}

We consider the conic linear program (CLP)
\begin{equation}
\label{eq:clp}
    \min_{x\in\mathbb R^n} c^\top x
    ~~
    \mathrm{s.t.}~~
    Ax=b,~~ x\in\mathcal K,
\end{equation}
where \(A\in\mathbb R^{m\times n}\), \(b\in\mathbb R^m\),
\(c\in\mathbb R^n\), and \(\mathcal K\subseteq\mathbb R^n\) is a
Cartesian product of projection-computable closed convex cones. We write
\(\mathcal A:=\{x\in\mathbb R^n:Ax=b\}\),
\(d:=(A,b,c,\mathcal K)\), and use \(\mathbb I_{\mathcal C}\) for the
indicator of a closed convex set \(\mathcal C\). The affine projection is
\(\Proj_{\mathcal A}(v)
=v-A^\top(AA^\top)^\dagger(Av-b)\), with
\((AA^\top)^\dagger\) the Moore--Penrose inverse; the cone projection
\(\Proj_{\mathcal K}\) is evaluated blockwise. For a cone
\(\mathcal K\), its dual is
\(\mathcal K^\ast:=\{s:\langle s,x\rangle\ge0
\ \forall x\in\mathcal K\}\).

Let \(y^k=(x^k,z^k,u^k)\) be the scaled splitting state and define
\(\bar c:=c/(\|c\|_2+\varepsilon_c)\), where
\(\varepsilon_c>0\). Given an admissible action
\(\vartheta_k=(\alpha_k,\beta_k)\), with
\(\alpha_k\in(0,2)\) and \(\beta_k>0\), one controlled transition is
\begin{equation}
\label{eq:rc_drs_block}
\begin{aligned}
    w^k&=z^k-u^k-\beta_k\bar c,&
    x^{k+1}&=\Proj_{\mathcal A}(w^k),\\
    \bar x^{k+1}
    &=\alpha_kx^{k+1}+(1-\alpha_k)z^k,&
    z^{k+1}&=\Proj_{\mathcal K}(\bar x^{k+1}+u^k),\\
    u^{k+1}&=u^k+\bar x^{k+1}-z^{k+1}.&&
\end{aligned}
\end{equation}
A \(K\)-step layer returns \(z^K\), which belongs to
\(\mathcal K\) up to numerical precision. The causal controller observes
the feature history \(\phi^{0:k}\) and produces
\(\omega_k=(\rho_k,\alpha_k,\beta_k)=\pi_\theta(\phi^{0:k})\).
Only \((\alpha_k,\beta_k)\) enter~\eqref{eq:rc_drs_block};
\(\rho_k>0\) is controller-side scale memory. For the operator analysis we
use the shadow state \(v^k:=z^k+u^k\), initialized so that
\(z^0=\Proj_{\mathcal K}(v^0)\); the standard initialization
\(z^0=u^0=0\) satisfies this relation.

\section{Methodological Details}
\label{app:method_details}

This section gives the complete implementation-level specification of
\textsc{RCDRS}, including the controller features, admissible action
parameterization, safeguarded envelope, terminal primal--dual diagnostics,
and training losses.

\subsection{Controller feature construction}
\label{app:controller_details}

The controller input at iteration \(k\) is constructed from the current
splitting state \(y^k=(x^k,z^k,u^k)\), the previous primal candidate
\(z^{k-1}\), the problem data \(d=(A,b,c,\mathcal K)\), the previous control
vector \(\omega_{k-1}\), and the normalized time index. We use the convention
\(z^{-1}=z^0\) and initialize the previous action by the base control
\(\omega_{-1}=(\rho_0,\alpha_0,\beta_0)\).

The residual components are
\begin{small}
\[
\begin{aligned}
    \eta_{\rm con}^k
    :=
    \frac{\|x^k-z^k\|_2}{1+\|z^k\|_2},
    \eta_{\rm eq}^k
    :=
    \frac{\|Az^k-b\|_2}{1+\|b\|_2},
    \eta_{\rm dz}^k
    :=
    \frac{\|z^k-z^{k-1}\|_2}{1+\|z^k\|_2},\\
    \eta_{\rm obj}^k
    :=
    \frac{|c^\top z^k|}{1+\|c\|_2\|z^k\|_2},
    \Delta_{\rm obj}^k
    :=
    \operatorname{asinh}
    \left(
    \frac{c^\top z^k-c^\top z^{k-1}}
    {1+|c^\top z^{k-1}|}
    \right).
\end{aligned}
\]
\end{small}
Here \(\eta_{\rm con}^k\) measures consensus or splitting consistency
between the affine-side and cone-side variables.
The quantity \(\eta_{\rm eq}^k\) measures affine feasibility,
\(\eta_{\rm dz}^k\) measures terminal movement of the primal candidate,
\(\eta_{\rm obj}^k\) measures objective scale, and
\(\Delta_{\rm obj}^k\) measures signed objective progress. The
normalizations make the same controller usable across instances with
different problem scales.

The previous-action features are
\(\xi_{\rho}^{k-1}:=\log(\rho_{k-1}/\rho_0)\),
\(\xi_{\alpha}^{k-1}:=\alpha_{k-1}\), and
\(\xi_{\beta}^{k-1}:=\log(\beta_{k-1}/\beta_0)\). The logarithmic encoding
is used for positive scale coordinates. The time-budget features are
\(t_k=k/K\) and \(\bar t_k=(K-k)/K\). The complete feature vector is
\(
    \phi^k =
    \Big[
    \log(1+\eta_{\rm con}^k),
    \log(1+\eta_{\rm eq}^k),
    \log(1+\eta_{\rm dz}^k),
    \log(1+\eta_{\rm obj}^k),
    \Delta_{\rm obj}^k,
    \xi_{\rho}^{k-1},
    \xi_{\alpha}^{k-1},
    \xi_{\beta}^{k-1},
    t_k,
    \bar t_k
    \Big].
\)
The transform \(\log(1+\cdot)\) compresses nonnegative residual magnitudes,
which may vary by orders of magnitude during a rollout. The transform
\(\operatorname{asinh}\) compresses objective progress while preserving its
sign. Cone feasibility is not included in \(\phi^k\), because
\(z^k\in\mathcal K\) is enforced by the cone projection in~\eqref{eq:rc_drs_block}.

\subsection{Admissible action parameterization}
\label{app:action_parameterization}

The causal controller is implemented as
\(h^k=\mathcal E_\theta(\phi^0,\ldots,\phi^k)\) and
\(r^k=g_\theta(h^k)\), where \(h^k\) is the hidden representation and
\(r^k=(r_\rho^k,r_\alpha^k,r_\beta^k)\) are raw control coordinates. In the
experiments, \(\mathcal E_\theta\) is a lightweight GRU. The recurrent
structure is used because the next useful action depends on trajectory
trends: similar residual magnitudes may correspond to improving,
oscillatory, or stagnant behavior.

The raw coordinates are mapped to admissible ranges by
\[
\begin{aligned}
    \log\rho_k
    &=
    \log\rho_{\min}
    +
    \sigma(r_\rho^k)
    \bigl(\log\rho_{\max}-\log\rho_{\min}\bigr),
    \\
    \alpha_k
    &=
    \alpha_{\min}
    +
    \sigma(r_\alpha^k)
    \bigl(\alpha_{\max}-\alpha_{\min}\bigr),
    \\
    \log\beta_k
    &=
    \log\beta_{\min}
    +
    \sigma(r_\beta^k)
    \bigl(\log\beta_{\max}-\log\beta_{\min}\bigr),
\end{aligned}
\]
where \(\sigma(\cdot)\) is the sigmoid function,
\(0<\rho_{\min}<\rho_{\max}\),
\(0<\alpha_{\min}<\alpha_{\max}<2\), and
\(0<\beta_{\min}<\beta_{\max}\). This gives the control vector
\(\omega_k=(\rho_k,\alpha_k,\beta_k)\), with
\(\rho_k>0\), \(\alpha_k\in(0,2)\), and \(\beta_k>0\) by construction.

The positive scale coordinates can be optionally filtered to limit abrupt
multiplicative changes. Let
\((\tilde\rho_k,\tilde\alpha_k,\tilde\beta_k)\) denote the outputs of the
bounded map above. The filtered scales are obtained by
\[
\begin{aligned}
    \rho_k
    &=
    \operatorname{clip}
    \left(
    \tilde\rho_k,
    \rho_{k-1}/\chi_\rho,
    \chi_\rho\rho_{k-1}
    \right),
    \\
    \beta_k
    &=
    \operatorname{clip}
    \left(
    \tilde\beta_k,
    \beta_{k-1}/\chi_\beta,
    \chi_\beta\beta_{k-1}
    \right),
\end{aligned}
\]
where \(\chi_\rho\ge1\) and \(\chi_\beta\ge1\). The clipped values are also
intersected with the global admissible ranges
\([\rho_{\min},\rho_{\max}]\) and \([\beta_{\min},\beta_{\max}]\). The
relaxation coordinate \(\alpha_k\) is already bounded in \((0,2)\) by the
admissible map.

For the safeguarded rollouts used in the time-varying analysis, the same
controller can be parameterized around a fixed admissible base control
\(\bar\omega=(\bar\rho,\bar\alpha,\bar\beta)\). Let
\(q_{\rho,k}=\tanh(r_\rho^k)\), \(q_{\alpha,k}=\tanh(r_\alpha^k)\), and
\(q_{\beta,k}=\tanh(r_\beta^k)\). The envelope form is
\begin{equation}
\label{eq:envelope_param}
\begin{aligned}
    \log\rho_k
    &=
    \log\bar\rho+\delta_k q_{\rho,k},
    ~~
    \alpha_k
    =
    \bar\alpha+s_\alpha\delta_k q_{\alpha,k},\\
    \log\beta_k
    &=
    \log\bar\beta+\delta_k q_{\beta,k}.
\end{aligned}
\end{equation}
Here \(s_\alpha>0\), and the envelope radius is
\[
    \delta_k=\frac{\delta_0}{1+(k/k_0)^p},
    ~~
    \delta_0>0,~~ k_0>0,~~ p>1 .
\]
The parameters are chosen so that all generated actions remain admissible;
in particular,
\[
    0<\bar\alpha-s_\alpha\delta_0
    \le
    \bar\alpha+s_\alpha\delta_0
    <2.
\]
The logarithmic parameterization keeps \(\rho_k>0\) and \(\beta_k>0\). The
condition \(p>1\) makes \(\sum_k\delta_k<\infty\), which is used in the
perturbed averaged-operator analysis. The envelope preserves finite-horizon
adaptability while making \((\alpha_k,\beta_k)\) approach the fixed
admissible pair \((\bar\alpha,\bar\beta)\) along long rollouts. This is the
form used in the time-varying analysis in Section~\ref{app:theory_proofs}.

\subsection{Terminal primal--dual diagnostics}
\label{app:terminal_diagnostics}

This diagnostic readout is stated for the conic setting, where
\(\mathcal K\) is a closed convex cone and \(\mathcal K^\ast\) denotes its
dual cone. For a \(K\)-step rollout, the terminal primal candidate is
\(z^K\). Since the last transition uses \(\beta_{K-1}\), define
\[
    \kappa_{K-1}
    :=
    \frac{\beta_{K-1}}{\|c\|_2+\varepsilon_c}.
\]
The cone projection residual in~\eqref{eq:rc_drs_block} provides
a cone-side normal vector. We use the rescaled vector
\(s^K:=-u^K/\kappa_{K-1}\) as a cone-dual-feasible diagnostic slack
for the original CLP. At finite depth, stationarity is assessed separately
through the least-squares equality multiplier and the residual \(r_{\rm d}^K\). The equality multiplier is
recovered by the least-squares stationarity fit
\[
    \lambda^K
    :=
    \arg\min_{\lambda}
    \|A^\top\lambda+s^K-c\|_2^2
    =
    (AA^\top)^\dagger A(c-s^K).
\]
Thus the terminal state yields the primal--dual triplet
\[
    \operatorname{RCDRS}_\theta^K(d)
    =
    (z^K,\lambda^K,s^K).
\]
In downstream tasks, the primal decision is \(z^K\), while
\((\lambda^K,s^K)\) are used as diagnostic readouts.

We evaluate finite-depth optimality through the normalized KKT-style
diagnostics
\[
\begin{aligned}
    r_{\rm p}^K
    &:=
    \frac{\|Az^K-b\|_2}{1+\|b\|_2},
    &
    r_{\rm d}^K
    &:=
    \frac{\|A^\top\lambda^K+s^K-c\|_2}{1+\|c\|_2},
    \\
    r_{\rm comp}^K
    &:=
    \frac{|\langle z^K,s^K\rangle|}
    {1+|c^\top z^K|},
    &
    r_{\rm gap}^K
    &:=
    \frac{|c^\top z^K-b^\top\lambda^K|}
    {1+|c^\top z^K|+|b^\top\lambda^K|}.
\end{aligned}
\]
Here \(r_{\rm p}^K\) measures primal affine feasibility,
\(r_{\rm d}^K\) measures dual stationarity,
\(r_{\rm comp}^K\) measures complementarity, and
\(r_{\rm gap}^K\) measures the primal--dual gap.

\subsection{Terminal-oriented training objective}
\label{app:training_details}

The training objective compares the learned rollout with a fixed-core
reference rollout at the same depth \(K\). The reference layer uses the base
control \(\omega_0=(\rho_0,\alpha_0,\beta_0)\) and the transition
\(F_{\alpha_0,\beta_0}\). Let
\(y_{\rm rc}^K=(x_{\rm rc}^K,z_{\rm rc}^K,u_{\rm rc}^K)\) be the terminal
state of the learned layer, and let
\(y_{\rm base}^K=(x_{\rm base}^K,z_{\rm base}^K,u_{\rm base}^K)\) be the
terminal state of the fixed-core reference layer.

The terminal residual losses are
\[
    \mathcal L_{\rm eq}
    :=
    \frac{\|Az_{\rm rc}^K-b\|_2^2}{(1+\|b\|_2)^2},
    ~~
    \mathcal L_{\rm con}
    :=
    \frac{\|x_{\rm rc}^K-z_{\rm rc}^K\|_2^2}
    {(1+\|z_{\rm rc}^K\|_2)^2},
\]
and the terminal movement loss is
\[
    \mathcal L_{\rm mov}
    :=
    \frac{\|z_{\rm rc}^K-z_{\rm rc}^{K-1}\|_2^2}
    {(1+\|z_{\rm rc}^K\|_2)^2}.
\]
Here \(\mathcal L_{\rm con}\) measures consensus or splitting consistency,
rather than cone infeasibility. The baseline-relative objective penalty is
\[
    \mathcal L_{\rm obj}
    :=
    \left[
    \frac{
    c^\top z_{\rm rc}^K-c^\top z_{\rm base}^K
    }
    {1+|c^\top z_{\rm base}^K|}
    \right]_+ .
\]
For the minimization problem in \eqref{eq:clp}, this term is positive only
when the learned terminal objective is larger than the fixed-core reference
objective under the same rollout depth.

The terminal merit of the learned rollout is
\[
    \mathcal M_{\rm rc}
    =
    \lambda_{\rm eq}\mathcal L_{\rm eq}
    +
    \lambda_{\rm con}\mathcal L_{\rm con}
    +
    \lambda_{\rm mov}\mathcal L_{\rm mov}
    +
    \lambda_{\rm obj}\mathcal L_{\rm obj},
\]
where all weights are nonnegative. We also define the reference merit
\[
    \mathcal M_{\rm base}
    =
    \lambda_{\rm eq}\mathcal L_{\rm eq}^{\rm base}
    +
    \lambda_{\rm con}\mathcal L_{\rm con}^{\rm base}
    +
    \lambda_{\rm mov}\mathcal L_{\rm mov}^{\rm base}.
\]
The base losses \(\mathcal L_{\rm eq}^{\rm base}\),
\(\mathcal L_{\rm con}^{\rm base}\), and
\(\mathcal L_{\rm mov}^{\rm base}\) are defined by the same formulas with
\(y_{\rm rc}^K\) replaced by \(y_{\rm base}^K\). The objective term is
excluded from \(\mathcal M_{\rm base}\), because \(\mathcal L_{\rm obj}\)
is already defined relative to the fixed-core reference. The optional
finite-budget dominance term is
\[
    \mathcal L_{\rm dom}
    =
    [\mathcal M_{\rm rc}-\mathcal M_{\rm base}+m]_+,
    ~~
    m\ge0.
\]
The margin \(m\) can be set to zero for non-worsening dominance or to a
small positive value to encourage strict finite-budget improvement.

The action sequence is regularized by
\[
    \mathcal L_{\rm smooth}
    =
    \sum_{k=1}^{K-1}
    [
    (\log\rho_k-\log\rho_{k-1})^2
    +(\alpha_k-\alpha_{k-1})^2
    +(\log\beta_k-\log\beta_{k-1})^2
    ].
\]
The total self-supervised training objective is
\[
    \min_\theta~
    \mathbb E_{d\sim\mathcal D}
    \left[
        \mathcal M_{\rm rc}
        +
        \lambda_{\rm dom}\mathcal L_{\rm dom}
        +
        \lambda_{\rm smooth}\mathcal L_{\rm smooth}
    \right].
\]
No reference optimizer solution is required during training; the fixed-core
rollout provides the baseline for the objective comparison, while the
residual terms directly penalize finite-depth feasibility and stability.


\section{Proofs and Auxiliary Results}
\label{app:theory_proofs}

This section states and proves the operator-theoretic results used for
\textsc{RCDRS}. It first establishes the fixed-block interpretation of the
\(\beta\)-parameterized transition, then derives terminal residual bounds,
analyzes safeguarded time-varying rollouts, and finally establishes the
terminal primal--dual diagnostics and additional finite-depth quality
results.

Throughout this section, unless otherwise stated, we assume that
\(\mathcal A\) and \(\mathcal K\) are nonempty closed convex sets, so that
the projection operators are well defined. We
also assume that the CLP has a nonempty solution set whenever an optimality
statement is made. A fixed pair \((\alpha,\beta)\) is called admissible if
\(\alpha\in(0,2)\) and \(\beta>0\).

The section is organized around the self-contained statements below.
The first group of results shows that one fixed \((\alpha,\beta)\) block is
not a new algorithmic object: it is exactly a relaxed DRS operator applied
to a positively scaled but solution-equivalent CLP. The second group uses
this averaged-operator interpretation to obtain finite-step residual bounds.
The third group extends the argument to safeguarded time-varying rollouts by
viewing the learned rollout as a summable perturbation of a limiting averaged
operator. The fourth group shows that the terminal scaled-dual state yields
a valid conic primal--dual readout. The final group gives auxiliary facts
explaining how the terminal training losses are connected to finite-depth
solution quality.

\subsection{Statements of the Theoretical Results}
\label{app:restated_theory}

For fixed \(\beta>0\), define
\(f_\beta(x):=\beta\bar c^\top x+\mathbb I_{\mathcal A}(x)\) and
\(g(x):=\mathbb I_{\mathcal K}(x)\). Let
\(J_{f_\beta}:=(I+\partial f_\beta)^{-1}\),
\(J_g:=(I+\partial g)^{-1}\),
\(R_{f_\beta}:=2J_{f_\beta}-I\), and \(R_g:=2J_g-I\). The DRS operator
and its relaxed version are
\[
    T_\beta:=\tfrac12(I+R_{f_\beta}R_g),
    ~
    T_{\alpha,\beta}:=(1-\alpha)I+\alpha T_\beta
    =I+\alpha(J_{f_\beta}R_g-J_g).
\]

\begin{proposition}[Fixed-block interpretation]
\label{prop:fixed_block_interpretation}
For any \(\beta>0\) and \(\alpha\in(0,2)\), the following statements
hold.
\begin{enumerate}[leftmargin=*,nosep]
    \item The scaled and original programs have the same primal solution
    set:
    \[
        \argmin_{x\in\mathcal A\cap\mathcal K}
        \beta\bar c^\top x
        =
        \argmin_{x\in\mathcal A\cap\mathcal K}c^\top x.
    \]
    \item If \(v^k=z^k+u^k\) and \(z^k=J_g(v^k)\), then one fixed
    \((\alpha,\beta)\) step of~\eqref{eq:rc_drs_block} satisfies
    \(v^{k+1}:=z^{k+1}+u^{k+1}=T_{\alpha,\beta}v^k\) and
    \(z^{k+1}=J_g(v^{k+1})\).
    \item \(T_{\alpha,\beta}\) is \(\alpha/2\)-averaged and
    nonexpansive.
\end{enumerate}
\end{proposition}

\begin{proposition}[Terminal residual bound for a fixed block]
\label{prop:finite_step_fixed_residual}
Let \(\alpha\in(0,2)\), \(\beta>0\), and
\(\Fix(T_{\alpha,\beta})\neq\emptyset\). If
\(v^{k+1}=T_{\alpha,\beta}v^k\), then for every \(N\ge1\),
\[
    \|v^N-T_{\alpha,\beta}v^N\|_2^2
    \le
    \frac{\alpha}{(2-\alpha)N}
    \dist^2\!\bigl(v^0,\Fix(T_{\alpha,\beta})\bigr).
\]
\end{proposition}

\begin{theorem}[Residual bounds under the safeguarded envelope]
\label{thm:finite_step_envelope}
Let \(T_k:=T_{\alpha_k,\beta_k}\),
\(T_\infty:=T_{\bar\alpha,\bar\beta}\), and
\(v^{k+1}=T_kv^k\). Suppose that
\(\Fix(T_\infty)\neq\emptyset\), the trajectory \(\{v^k\}\) is bounded,
and the envelope in~\eqref{eq:envelope_param} holds. Then there exist
summable nonnegative sequences \(\{\epsilon_k\}\) and \(\{\xi_k\}\) such
that \(\|T_kv^k-T_\infty v^k\|_2\le\epsilon_k\). For any
\(v^\star\in\Fix(T_\infty)\), define
\(c_\infty:=(2-\bar\alpha)/\bar\alpha\) and
\(D_\infty:=\|v^0-v^\star\|_2^2+\sum_{k=0}^{\infty}\xi_k\). Then, for
every \(N\ge1\),
\[
    \min_{0\le k<N}\|v^k-T_\infty v^k\|_2^2
    \le \frac{D_\infty}{c_\infty N},~~
    \|v^N-T_\infty v^N\|_2
    \le
    \sqrt{\frac{2D_\infty}{c_\infty N}}
    +2\sum_{k=\lfloor N/2\rfloor}^{\infty}\epsilon_k .
\]
Moreover, \(\epsilon_k=O(\delta_k)\). If
\(\delta_k=\delta_0/(1+(k/k_0)^p)\) with \(p>1\), the terminal residual
is \(O(N^{-1/2}+N^{1-p})\); its square is \(O(1/N)\) when
\(p\ge3/2\).
\end{theorem}

\begin{theorem}[Consistency of the terminal primal--dual readout]
\label{thm:terminal_primal_dual_certificate}
Assume that \(\mathcal K\) is a closed convex cone and the cone projection
is evaluated exactly. Every terminal state generated
by~\eqref{eq:rc_drs_block} yields a recovered slack satisfying
\(s^K\in\mathcal K^\ast\) and
\(\langle z^K,s^K\rangle=0\). Moreover, fix a CLP and consider a bounded
sequence of terminal readouts
\(\{(z_j,\lambda_j,s_j)\}_{j\ge1}\). If the residuals defined in
Section~\ref{app:terminal_diagnostics} satisfy
\(r_{{\rm p},j}\to0\), \(r_{{\rm d},j}\to0\), and
\(r_{{\rm gap},j}\to0\), then every cluster point
\((z^\star,\lambda^\star,s^\star)\) is primal--dual optimal for the
original CLP.
\end{theorem}

\subsection{Auxiliary facts for the fixed-block DRS interpretation}
\label{app:fixed_block_auxiliary}

The first step is to identify the controlled transition in
\eqref{eq:rc_drs_block} with a standard DRS operator. This requires three
ingredients. We first compute the resolvents of the scaled affine-objective
term and the cone indicator. We then show that scaling the objective by
\(\beta\bar c\) preserves the primal minimizer set. Finally, we verify that
the state variables \((z^k,u^k)\) realize the usual DRS shadow iteration in
the variable \(v^k=z^k+u^k\). These ingredients together prove
Proposition~\ref{prop:fixed_block_interpretation}.

Recall that, for any fixed \(\beta>0\), we define
\[
    f_\beta(x)
    :=
    \beta\bar c^\top x+\mathbb I_{\mathcal A}(x),
    ~~
    g(x)
    :=
    \mathbb I_{\mathcal K}(x),
    ~~
    \bar c
    :=
    \frac{c}{\|c\|_2+\varepsilon_c}.
\]
Let \(J_{f_\beta}:=(I+\partial f_\beta)^{-1}\) and
\(J_g:=(I+\partial g)^{-1}\) be the corresponding resolvents. The
reflected resolvents are \(R_{f_\beta}:=2J_{f_\beta}-I\) and
\(R_g:=2J_g-I\). The DRS operator and its relaxed version are
\[
    T_\beta:=\frac12(I+R_{f_\beta}R_g),
    ~~
    T_{\alpha,\beta}
    :=
    (1-\alpha)I+\alpha T_\beta
    =
    I+\alpha(J_{f_\beta}R_g-J_g).
\]

\begin{lemma}[\(\beta\)-parameterized resolvents]
\label{lem:beta_resolvent}
For any \(\beta>0\),
\[
    J_{f_\beta}(v)=\Proj_{\mathcal A}(v-\beta\bar c),
    ~~
    J_g(v)=\Proj_{\mathcal K}(v).
\]
\end{lemma}

\begin{proof}
By the definition of the resolvent,
\(J_{f_\beta}(v)=\argmin_{x\in\mathcal A}
\{\beta\bar c^\top x+\frac12\|x-v\|_2^2\}\). Completing the square gives
\[
    \beta\bar c^\top x+\frac12\|x-v\|_2^2
    =
    \frac12\|x-(v-\beta\bar c)\|_2^2
    +
    C(v,\beta,\bar c),
\]
where \(C(v,\beta,\bar c)\) is independent of \(x\). Hence
\(J_{f_\beta}(v)=\Proj_{\mathcal A}(v-\beta\bar c)\). Since
\(g=\mathbb I_{\mathcal K}\), its resolvent is
\(J_g(v)=\argmin_x\{\mathbb I_{\mathcal K}(x)+\frac12\|x-v\|_2^2\}
=\Proj_{\mathcal K}(v)\).
\end{proof}

Lemma~\ref{lem:beta_resolvent} shows that the affine step used by
\textsc{RCDRS} is exactly the resolvent of the scaled affine-objective
term. Thus the only remaining question at the optimization-problem level is
whether the scaling changes the target solution set. The next proposition
shows that it does not.

\begin{proposition}[Solution-set preservation]
\label{prop:solution_preservation}
Assume that the original CLP has a nonempty primal solution set. For any
\(\beta>0\),
\[
    \arg\min_{x\in\mathcal A\cap\mathcal K}
    \beta\bar c^\top x
    =
    \arg\min_{x\in\mathcal A\cap\mathcal K}
    c^\top x .
\]
\end{proposition}

\begin{proof}
If \(c=0\), both objectives are constant on the feasible set, so the two
minimizer sets are equal to \(\mathcal A\cap\mathcal K\). If \(c\neq0\),
then \(\beta\bar c=\theta_\beta c\), where
\(\theta_\beta:=\beta/(\|c\|_2+\varepsilon_c)>0\). Therefore, for any
feasible \(x_1,x_2\),
\[
    c^\top x_1\le c^\top x_2
    \Longleftrightarrow
    \theta_\beta c^\top x_1\le \theta_\beta c^\top x_2
    \Longleftrightarrow
    \beta\bar c^\top x_1\le \beta\bar c^\top x_2.
\]
Thus the two objectives induce the same ordering on the feasible set and
have the same minimizers.
\end{proof}

Therefore, replacing \(c\) by the positive scaled direction
\(\beta\bar c\) changes the magnitude of the objective drive but not the
set of primal optima. We can now analyze the controlled update as a DRS
iteration for this equivalent scaled problem.

The next lemma connects the implementation-level state update to the
operator-theoretic DRS map. The key shadow variable is
\(v^k=z^k+u^k\). If \(z^k\) is the cone resolvent of \(v^k\), then one
state-form update preserves this relation and advances \(v^k\) by the
relaxed DRS operator \(T_{\alpha,\beta}\).

\begin{lemma}[State-form realization of the \(\beta\)-DRS block]
\label{lem:state_form_realization}
Fix \(\alpha\in(0,2)\) and \(\beta>0\). Suppose that, at iteration \(k\),
\(v^k=z^k+u^k\) and \(z^k=J_g(v^k)\). Apply one step of
\eqref{eq:rc_drs_block} with the fixed pair \((\alpha,\beta)\). Then
\[
    v^{k+1}:=z^{k+1}+u^{k+1}=T_{\alpha,\beta}v^k,
    ~~
    z^{k+1}=J_g(v^{k+1}),
\]
and \(u^{k+1}=v^{k+1}-z^{k+1}\).
\end{lemma}

\begin{proof}
Since \(v^k=z^k+u^k\) and \(z^k=J_g(v^k)\), we have
\(z^k-u^k=2z^k-v^k=R_gv^k\). Hence the affine projection step gives
\[
    x^{k+1}
    =
    \Proj_{\mathcal A}(R_gv^k-\beta\bar c)
    =
    J_{f_\beta}(R_gv^k),
\]
where the second equality follows from Lemma~\ref{lem:beta_resolvent}.
Using the relaxation and dual-correction steps,
\[
\begin{aligned}
    v^{k+1}
    &=
    z^{k+1}+u^{k+1}
    =
    u^k+\bar x^{k+1}                                      \\
    &=
    v^k-z^k+\alpha x^{k+1}+(1-\alpha)z^k                  \\
    &=
    v^k+\alpha\bigl(J_{f_\beta}R_gv^k-J_gv^k\bigr)
    =
    T_{\alpha,\beta}v^k .
\end{aligned}
\]
Moreover, the cone projection step gives
\(z^{k+1}=\Proj_{\mathcal K}(\bar x^{k+1}+u^k)\). Since
\(\bar x^{k+1}+u^k=v^{k+1}\), we have
\(z^{k+1}=J_g(v^{k+1})\). The identity
\(u^{k+1}=v^{k+1}-z^{k+1}\) follows from the definition of \(v^{k+1}\).
\end{proof}

This establishes that the implemented block follows the same shadow
trajectory as relaxed DRS. The last ingredient needed for the fixed-block
interpretation is the averagedness of this relaxed operator.

\begin{proposition}[Averagedness]
\label{prop:averagedness}
Let \(\beta>0\) and \(\alpha\in(0,2)\). Then
\(T_{\alpha,\beta}\) is \(\alpha/2\)-averaged.
\end{proposition}

\begin{proof}
Since \(f_\beta\) and \(g\) are proper closed convex functions, their
subdifferentials are maximal monotone. Hence \(J_{f_\beta}\) and \(J_g\)
are firmly nonexpansive, and \(R_{f_\beta}\) and \(R_g\) are nonexpansive.
Thus \(N_\beta:=R_{f_\beta}R_g\) is nonexpansive. Since
\(T_\beta=\frac12 I+\frac12N_\beta\),
\[
    T_{\alpha,\beta}
    =
    (1-\alpha)I+\alpha T_\beta
    =
    \left(1-\frac{\alpha}{2}\right)I
    +
    \frac{\alpha}{2}N_\beta .
\]
Because \(\alpha/2\in(0,1)\), \(T_{\alpha,\beta}\) is
\(\alpha/2\)-averaged.
\end{proof}

We now have all ingredients needed for
Proposition~\ref{prop:fixed_block_interpretation}: solution-set
preservation, state-form realization, and averagedness.

\begin{proof}[Proof of Proposition~\ref{prop:fixed_block_interpretation}]
The first statement follows from
Proposition~\ref{prop:solution_preservation}, the second from
Lemma~\ref{lem:state_form_realization}, and the third from
Proposition~\ref{prop:averagedness}. This proves the proposition.
\end{proof}

At this point, the fixed-block claim is complete: a fixed admissible
\((\alpha,\beta)\) block preserves the primal solution set, realizes the
standard DRS shadow iteration, and remains an averaged operator. The next
subsection uses averagedness to derive the terminal residual bound stated in
Proposition~\ref{prop:finite_step_fixed_residual}.

\subsection{Terminal residual bounds for fixed admissible blocks}
\label{app:fixed_block_residual_bounds}

The previous subsection established that every fixed admissible block is an
averaged operator. The proof below applies the standard descent inequality
for averaged operators and the monotonicity of the fixed-point residual to
obtain the terminal \(O(1/N)\) squared-residual bound.

\begin{proof}[Proof of Proposition~\ref{prop:finite_step_fixed_residual}]
Let \(a=\alpha/2\). By Proposition~\ref{prop:averagedness},
\(T_{\alpha,\beta}\) is \(a\)-averaged, so
\(T_{\alpha,\beta}=(1-a)I+aN_\beta\) for some nonexpansive operator
\(N_\beta\). Let \(v^\star\in\Fix(T_{\alpha,\beta})\). Then
\(N_\beta v^\star=v^\star\). For any \(v\),
\[
\begin{aligned}
    \|T_{\alpha,\beta}v-v^\star\|_2^2
    &=
    \|(1-a)(v-v^\star)+a(N_\beta v-v^\star)\|_2^2        \\
    &\le
    \|v-v^\star\|_2^2
    -
    a(1-a)\|v-N_\beta v\|_2^2 .
\end{aligned}
\]
Since \(v-T_{\alpha,\beta}v=a(v-N_\beta v)\), this becomes
\[
    \|T_{\alpha,\beta}v-v^\star\|_2^2
    \le
    \|v-v^\star\|_2^2
    -
    \frac{1-a}{a}
    \|v-T_{\alpha,\beta}v\|_2^2 .
\]
Using \(a=\alpha/2\) and applying the inequality to
\(v^{k+1}=T_{\alpha,\beta}v^k\), we obtain
\[
    \|v^{k+1}-v^\star\|_2^2
    \le
    \|v^k-v^\star\|_2^2
    -
    \frac{2-\alpha}{\alpha}
    \|v^k-T_{\alpha,\beta}v^k\|_2^2 .
\]
Summing from \(k=0\) to \(N-1\) gives
\[
    \sum_{k=0}^{N-1}
    \|v^k-T_{\alpha,\beta}v^k\|_2^2
    \le
    \frac{\alpha}{2-\alpha}\|v^0-v^\star\|_2^2 .
\]
Let \(r_k:=\|v^k-T_{\alpha,\beta}v^k\|_2\). Since
\(T_{\alpha,\beta}\) is nonexpansive,
\[
    r_{k+1}
    =
    \|T_{\alpha,\beta}v^k-T_{\alpha,\beta}v^{k+1}\|_2
    \le
    \|v^k-v^{k+1}\|_2
    =
    r_k .
\]
Thus \(r_N^2\le N^{-1}\sum_{k=0}^{N-1}r_k^2\), and hence
\[
    \|v^N-T_{\alpha,\beta}v^N\|_2^2
    \le
    \frac{\alpha}{(2-\alpha)N}
    \|v^0-v^\star\|_2^2 .
\]
Taking the infimum over \(v^\star\in\Fix(T_{\alpha,\beta})\) proves the
claim.
\end{proof}

The bound controls the DRS fixed-point residual in the shadow variable
\(v^k\). Since the training objective does not directly penalize this
shadow residual, we next relate it to the terminal movement of the primal
candidate \(z^k\), which appears in the loss.

\begin{proposition}[Terminal movement as a residual surrogate]
\label{prop:movement_residual_surrogate}
Consider a fixed block with
\(v^{k+1}=T_{\alpha,\beta}v^k\) and \(z^k=J_g(v^k)\), where
\(\alpha\in(0,2)\), \(\beta>0\), and
\(\Fix(T_{\alpha,\beta})\neq\emptyset\). Then, for every \(k\ge1\),
\[
    \|z^k-z^{k-1}\|_2
    \le
    \|v^k-v^{k-1}\|_2
    =
    \|T_{\alpha,\beta}v^{k-1}-v^{k-1}\|_2 .
\]
Moreover, for every \(N\ge1\),
\[
    \|z^N-z^{N-1}\|_2^2
    \le
    \frac{\alpha}{(2-\alpha)N}
    \operatorname{dist}^2(v^0,\Fix(T_{\alpha,\beta})).
\]
\end{proposition}

\begin{proof}
Since \(J_g\) is firmly nonexpansive, it is nonexpansive. Hence
\[
    \|z^k-z^{k-1}\|_2
    =
    \|J_g(v^k)-J_g(v^{k-1})\|_2
    \le
    \|v^k-v^{k-1}\|_2 .
\]
The equality
\(\|v^k-v^{k-1}\|_2=\|T_{\alpha,\beta}v^{k-1}-v^{k-1}\|_2\) follows from
\(v^k=T_{\alpha,\beta}v^{k-1}\). Since \(T_{\alpha,\beta}\) is
nonexpansive, the sequence \(\|v^{k+1}-v^k\|_2\) is nonincreasing.
Therefore,
\[
    \|v^N-v^{N-1}\|_2^2
    \le
    \frac1N
    \sum_{k=0}^{N-1}
    \|T_{\alpha,\beta}v^k-v^k\|_2^2 .
\]
Combining this with the summation bound from the proof of
Proposition~\ref{prop:finite_step_fixed_residual} gives
\[
    \|v^N-v^{N-1}\|_2^2
    \le
    \frac{\alpha}{(2-\alpha)N}
    \operatorname{dist}^2(v^0,\Fix(T_{\alpha,\beta})).
\]
The claimed bound follows from
\(\|z^N-z^{N-1}\|_2\le\|v^N-v^{N-1}\|_2\).
\end{proof}

Thus the terminal movement loss used in training can be interpreted as a
computable surrogate for fixed-point residual decay. The fixed-block
analysis is now complete. We next move to the learned case, where the
controller changes \((\alpha_k,\beta_k)\) along the rollout.

\subsection{Safeguarded time-varying rollouts}
\label{app:time_varying_rollout_proofs}

The fixed-block proof cannot be applied directly when the controller produces
a different operator \(T_k\) at each iteration. This subsection concerns
\textsc{RCDRS-Env}; the unconstrained \textsc{RCDRS-NoEnv} variant is not
claimed to satisfy the following envelope guarantee. We compare the
safeguarded rollout with a limiting fixed operator \(T_\infty\). The argument
has two steps: first, show that
\(T_{\alpha,\beta}\) depends locally Lipschitz-continuously on
\((\alpha,\log\beta)\) on bounded trajectories; second, use the decaying
envelope to show that \(T_kv^k-T_\infty v^k\) is summable.

The learned controller produces a time-varying sequence
\(T_k=T_{\alpha_k,\beta_k}\). Under the safeguarded parameterization
in~\eqref{eq:envelope_param},
\(\alpha_k=\bar\alpha+s_\alpha\delta_k q_{\alpha,k}\) and
\(\log\beta_k=\log\bar\beta+\delta_kq_{\beta,k}\),
where \(q_{\alpha,k},q_{\beta,k}\in[-1,1]\), \(s_\alpha>0\),
\(\delta_k\ge0\), and \(\sum_{k=0}^{\infty}\delta_k<\infty\). Let
\(T_\infty:=T_{\bar\alpha,\bar\beta}\). The \(\rho_k\) coordinate is
omitted in this analysis because it is used only as a controller-side
scale memory and does not enter the DRS operator \(T_{\alpha,\beta}\).
By construction, the safeguarded actions lie in compact intervals
\([\alpha_{\min},\alpha_{\max}]\subset(0,2)\) and
\([\beta_{\min},\beta_{\max}]\subset(0,\infty)\).

\begin{lemma}[Local Lipschitz dependence on \(\alpha\) and \(\beta\)]
\label{lem:local_lipschitz_alpha_beta}
Fix
\(0<\alpha_{\min}\le\alpha,\alpha'\le\alpha_{\max}<2\) and
\(0<\beta_{\min}\le\beta,\beta'\le\beta_{\max}<\infty\). For every
bounded set \(B_R:=\{v:\|v\|_2\le R\}\), there exists \(L_R>0\) such that
\[
    \|T_{\alpha,\beta}v-T_{\alpha',\beta'}v\|_2
    \le
    L_R
    \left(
        |\alpha-\alpha'|
        +
        |\log\beta-\log\beta'|
    \right)
\]
for all \(v\in B_R\).
\end{lemma}

\begin{proof}
Using the averaged representation,
\(T_{\alpha,\beta}=(1-\alpha/2)I+(\alpha/2)N_\beta\), where
\(N_\beta:=R_{f_\beta}R_g\). Hence
\[
    T_{\alpha,\beta}v-T_{\alpha',\beta'}v
    =
    \frac{\alpha-\alpha'}{2}(N_\beta v-v)
    +
    \frac{\alpha'}{2}(N_\beta v-N_{\beta'}v).
\]
We first bound \(\|N_\beta v-v\|_2\) on \(B_R\), uniformly for
\(\beta\in[\beta_{\min},\beta_{\max}]\). Since \(N_\beta\) is
nonexpansive, \(\|N_\beta v\|_2\le\|v\|_2+\|N_\beta0\|_2\). Moreover,
by Lemma~\ref{lem:beta_resolvent} and nonexpansiveness of
\(\Proj_{\mathcal A}\),
\[
    \|J_{f_\beta}(y)-J_{f_{\beta'}}(y)\|_2
    \le
    |\beta-\beta'|\|\bar c\|_2
    \le
    |\beta-\beta'|.
\]
Taking any fixed \(\beta_0\in[\beta_{\min},\beta_{\max}]\), this implies
that \(J_{f_\beta}(y)\), and hence \(N_\beta0\), is uniformly bounded for
\(\beta\) in the compact interval. Therefore, there exists \(C_R>0\)
such that \(\|N_\beta v-v\|_2\le C_R\) for all \(v\in B_R\).

For the second term, let \(y=R_gv\). Since
\(N_\beta v=R_{f_\beta}y\), Lemma~\ref{lem:beta_resolvent} gives
\[
\begin{aligned}
    \|N_\beta v-N_{\beta'}v\|_2
    &=
    2\|J_{f_\beta}y-J_{f_{\beta'}}y\|_2                 \\
    &\le
    2|\beta-\beta'|\|\bar c\|_2
    \le
    2|\beta-\beta'|.
\end{aligned}
\]
Because \(\beta,\beta'\in[\beta_{\min},\beta_{\max}]\),
\[
    |\beta-\beta'|
    =
    |\exp(\log\beta)-\exp(\log\beta')|
    \le
    \beta_{\max}|\log\beta-\log\beta'|.
\]
Combining the above bounds and using \(\alpha'\le\alpha_{\max}\), we
obtain the claimed inequality for a suitable constant \(L_R>0\).
\end{proof}

Lemma~\ref{lem:local_lipschitz_alpha_beta} converts parameter deviations
into operator deviations on bounded trajectories. The safeguarded envelope
makes these parameter deviations decay summably, so the next lemma turns the
controller envelope into a summable perturbation bound.

\begin{lemma}[Summable perturbation induced by the envelope]
\label{lem:summable_perturbation}
Assume that the rollout \(v^{k+1}=T_kv^k\) remains in a bounded set
\(B_R\), and that \(\sum_{k=0}^{\infty}\delta_k<\infty\). Then there
exists a summable nonnegative sequence \(\{\epsilon_k\}\) such that
\[
    \|T_kv^k-T_\infty v^k\|_2
    \le
    \epsilon_k,
    ~~
    \sum_{k=0}^{\infty}\epsilon_k<\infty .
\]
\end{lemma}

\begin{proof}
The envelope gives
\(|\alpha_k-\bar\alpha|\le s_\alpha\delta_k\) and
\(|\log\beta_k-\log\bar\beta|\le\delta_k\). Since \(v^k\in B_R\),
Lemma~\ref{lem:local_lipschitz_alpha_beta} implies
\[
    \|T_kv^k-T_\infty v^k\|_2
    \le
    L_R(s_\alpha+1)\delta_k .
\]
Set \(\epsilon_k:=L_R(s_\alpha+1)\delta_k\). Since
\(\sum_k\delta_k<\infty\), the sequence \(\{\epsilon_k\}\) is summable.
\end{proof}

The time-varying learned rollout can therefore be written as
\(v^{k+1}=T_\infty v^k+e_k\), where the perturbations \(e_k\) are
summable. This puts the rollout in the standard perturbed averaged-operator
setting and allows the fixed-block descent argument to be reused up to
summable error terms.

\begin{proof}[Proof of Theorem~\ref{thm:finite_step_envelope}]
Since \(T_\infty=T_{\bar\alpha,\bar\beta}\) is
\(\bar\alpha/2\)-averaged, the averaged-operator descent inequality gives
\[
    \|T_\infty v-v^\star\|_2^2
    \le
    \|v-v^\star\|_2^2
    -
    c_\infty
    \|v-T_\infty v\|_2^2,
    ~~
    c_\infty=\frac{2-\bar\alpha}{\bar\alpha}>0.
\]
Define \(e_k:=T_kv^k-T_\infty v^k\). Then
\(v^{k+1}=T_\infty v^k+e_k\). By
Lemma~\ref{lem:summable_perturbation}, there exists a summable sequence
\(\{\epsilon_k\}\) such that \(\|e_k\|_2\le\epsilon_k\). Since
\(\{v^k\}\) is bounded and \(T_\infty\) is nonexpansive, there exists
\(M>0\) such that
\(\|T_\infty v^k-v^\star\|_2\le M\) for all \(k\). Therefore,
\[
\begin{aligned}
    \|v^{k+1}-v^\star\|_2^2
    &=
    \|T_\infty v^k+e_k-v^\star\|_2^2                  \\
    &\le
    \|T_\infty v^k-v^\star\|_2^2
    +
    2M\epsilon_k+\epsilon_k^2                         \\
    &\le
    \|v^k-v^\star\|_2^2
    -
    c_\infty\|v^k-T_\infty v^k\|_2^2
    +
    \xi_k,
\end{aligned}
\]
where \(\xi_k:=2M\epsilon_k+\epsilon_k^2\). Since
\(\{\epsilon_k\}\) is summable and nonnegative, \(\{\xi_k\}\) is also
summable. Define
\[
    D_\infty
    :=
    \|v^0-v^\star\|_2^2+\sum_{k=0}^{\infty}\xi_k<\infty .
\]
Rearranging and summing from \(k=0\) to \(N-1\) yields
\[
    c_\infty
    \sum_{k=0}^{N-1}
    \|v^k-T_\infty v^k\|_2^2
    \le
    \|v^0-v^\star\|_2^2
    +\sum_{k=0}^{N-1}\xi_k
    \le D_\infty .
\]
Taking the minimum over \(0\le k<N\) gives
\[
    \min_{0\le k<N}
    \|v^k-T_\infty v^k\|_2^2
    \le
    \frac{D_\infty}{c_\infty N}.
\]

It remains to control the terminal rather than only the best residual. Let
\(r_k:=\|v^k-T_\infty v^k\|_2\). Since
\(v^{k+1}=T_\infty v^k+e_k\), nonexpansiveness of \(T_\infty\) gives
\[
\begin{aligned}
    r_{k+1}
    &=
    \|v^{k+1}-T_\infty v^{k+1}\|_2 \\
    &\le
    \|T_\infty v^k-T_\infty^2v^k\|_2+2\|e_k\|_2
    \le r_k+2\epsilon_k .
\end{aligned}
\]
Choose \(j_N\in\{\lfloor N/2\rfloor,\ldots,N-1\}\) minimizing \(r_k\)
over this index set. The preceding summation bound implies
\[
    r_{j_N}
    \le
    \sqrt{\frac{2D_\infty}{c_\infty N}} .
\]
Iterating \(r_{k+1}\le r_k+2\epsilon_k\) from \(j_N\) to \(N-1\)
therefore yields
\[
    \|v^N-T_\infty v^N\|_2
    \le
    \sqrt{\frac{2D_\infty}{c_\infty N}}
    +
    2\sum_{k=\lfloor N/2\rfloor}^{\infty}\epsilon_k .
\]
Finally, Lemma~\ref{lem:summable_perturbation} gives
\(\epsilon_k=O(\delta_k)\). If
\(\delta_k=\delta_0/(1+(k/k_0)^p)\) with \(p>1\), an integral comparison
gives \(\sum_{k=\lfloor N/2\rfloor}^{\infty}\epsilon_k=O(N^{1-p})\).
Hence the terminal residual is
\(O(N^{-1/2}+N^{1-p})\); for \(p\ge3/2\), its square is \(O(1/N)\).
This proves the theorem.
\end{proof}

The theorem shows that the safeguarded trajectory remains a summable
perturbation of a limiting relaxed DRS operator. The best-iterate estimate
follows from perturbed Fej\'er descent, while the summable tail controls the
terminal residual used by a fixed-depth layer.

\begin{corollary}[Terminal residual convergence under the envelope]
\label{cor:envelope_residual_convergence}
Under the assumptions of Theorem~\ref{thm:finite_step_envelope},
\[
    \|v^k-T_\infty v^k\|_2\to0.
\]
\end{corollary}

\begin{proof}
The terminal bound in Theorem~\ref{thm:finite_step_envelope} tends to zero
because \(N^{-1/2}\to0\) and the tail of the summable sequence
\(\{\epsilon_k\}\) vanishes.
\end{proof}

\subsection{Terminal primal--dual diagnostics}
\label{app:primal_dual_diagnostic_proofs}

The previous results justify the controlled rollout as a DRS-type operator
sequence. We now show that the terminal state also retains the primal--dual
structure needed for diagnostics. The key observation is that the final cone
projection residual is a normal-cone vector. After rescaling by the last
objective-drive parameter, it gives a cone-dual-feasible diagnostic slack
for the original CLP. Exact dual stationarity is not asserted at finite
depth and is instead measured by the recovered stationarity residual.

The results in this subsection are stated for the conic setting, where
\(\mathcal K\) is a closed convex cone. For a \(K\)-step rollout, the
final transition is indexed by \(K-1\). Recall
\[
    \kappa_{K-1}
    :=
    \frac{\beta_{K-1}}{\|c\|_2+\varepsilon_c}.
\]
The terminal dual slack and equality multiplier are recovered as
\[
    s^K:=-\frac{1}{\kappa_{K-1}}u^K,
    ~~
    \lambda^K:=(AA^\top)^\dagger A(c-s^K).
\]

\begin{proposition}[Cone-side dual slack]
\label{prop:cone_side_dual_slack}
Assume that \(\mathcal K\) is a closed convex cone. For the terminal
state generated by \eqref{eq:rc_drs_block}, the recovered slack satisfies
\[
    s^K\in\mathcal K^\ast,
    ~~
    \langle z^K,s^K\rangle=0 .
\]
\end{proposition}

\begin{proof}
For the last transition, write
\(z^K=\Proj_{\mathcal K}(q^K)\), where
\(q^K:=\bar x^K+u^{K-1}\). The dual-correction step gives
\(u^K=u^{K-1}+\bar x^K-z^K=q^K-z^K\). Since projection residuals belong
to the normal cone, \(u^K\in N_{\mathcal K}(z^K)\).

For a closed convex cone, the normal cone satisfies
\(N_{\mathcal K}(z^K)=(-\mathcal K^\ast)\cap(z^K)^\perp\). Indeed, if
\(r\in N_{\mathcal K}(z^K)\), then using \(\xi=0\) and
\(\xi=2z^K\) in the normal-cone inequality gives
\(\langle r,z^K\rangle=0\), and then
\(\langle r,\xi\rangle\le0\) for all \(\xi\in\mathcal K\), i.e.,
\(r\in-\mathcal K^\ast\). The converse is immediate from the same
inequality. Therefore \(u^K\in-\mathcal K^\ast\) and
\(\langle z^K,u^K\rangle=0\). Since
\(s^K=-u^K/\kappa_{K-1}\) with \(\kappa_{K-1}>0\), we obtain
\(s^K\in\mathcal K^\ast\) and \(\langle z^K,s^K\rangle=0\).
\end{proof}

This proposition establishes cone dual feasibility and complementarity from
the final projection step alone. To measure stationarity, we still need an
equality multiplier. The next proposition defines the multiplier as the
least-squares fit to the stationarity equation.

\begin{proposition}[Least-squares equality multiplier]
\label{prop:least_squares_dual_recovery}
The multiplier
\[
    \lambda^K=(AA^\top)^\dagger A(c-s^K)
\]
solves
\[
    \lambda^K\in
    \argmin_{\lambda}
    \|A^\top\lambda+s^K-c\|_2^2 .
\]
\end{proposition}

\begin{proof}
For \(\psi(\lambda):=\|A^\top\lambda+s^K-c\|_2^2\), the normal equation
is \(AA^\top\lambda=A(c-s^K)\). The Moore--Penrose solution of this
normal equation is
\(\lambda^K=(AA^\top)^\dagger A(c-s^K)\), which is a least-squares
minimizer.
\end{proof}

Combining the cone-side slack with this least-squares equality multiplier
gives the terminal readout defined in
Section~\ref{app:terminal_diagnostics}. We now prove that every
cluster point of a bounded sequence of readouts is primal--dual optimal when
the associated residuals vanish.

\begin{proof}[Proof of Theorem~\ref{thm:terminal_primal_dual_certificate}]
The first part follows from Proposition~\ref{prop:cone_side_dual_slack}:
when \(\mathcal K\) is a closed convex cone and the cone projection is
evaluated exactly, \(s^K\in\mathcal K^\ast\) and
\(\langle z^K,s^K\rangle=0\).

Now fix a CLP and consider a bounded sequence of terminal readouts
\(\{(z_j,\lambda_j,s_j)\}\) whose corresponding residuals satisfy
\(r_{{\rm p},j}\to0\), \(r_{{\rm d},j}\to0\), and
\(r_{{\rm gap},j}\to0\). Let
\((z_{j_\ell},\lambda_{j_\ell},s_{j_\ell})\to
(z^\star,\lambda^\star,s^\star)\) be any convergent subsequence.
Since \(z_{j_\ell}\in\mathcal K\) and \(\mathcal K\) is closed,
\(z^\star\in\mathcal K\). Since
\(s_{j_\ell}\in\mathcal K^\ast\) and \(\mathcal K^\ast\) is closed,
\(s^\star\in\mathcal K^\ast\). The convergence
\(r_{{\rm p},j_\ell}\to0\) gives \(Az^\star=b\), and
\(r_{{\rm d},j_\ell}\to0\) gives
\(A^\top\lambda^\star+s^\star=c\). Thus \(z^\star\) is primal feasible
and \((\lambda^\star,s^\star)\) is dual feasible.

Finally, boundedness makes the denominator of the normalized gap residual
uniformly bounded, so \(r_{{\rm gap},j_\ell}\to0\) implies
\(c^\top z^\star=b^\top\lambda^\star\). For any primal feasible \(z\)
and dual feasible \((\lambda,s)\),
\[
    c^\top z
    =
    (A^\top\lambda+s)^\top z
    =
    b^\top\lambda+s^\top z
    \ge
    b^\top\lambda,
\]
because \(z\in\mathcal K\) and \(s\in\mathcal K^\ast\). Hence weak
duality holds. Since the primal feasible \(z^\star\) and dual feasible
\((\lambda^\star,s^\star)\) attain equal objective values, both are
optimal.
\end{proof}

This completes the proof of the terminal readout theorem. The argument does
not require any individual finite-depth output to be exactly optimal; it
states that bounded terminal readouts are consistent with the KKT system as
their primal, dual, and gap residuals vanish.

\subsection{Fixed points and exact CLP optimality}
\label{app:fixed_points_clp_optimality}

The terminal readout theorem concerns approximate finite-depth terminal
states. For completeness, we also record the exact fixed-point counterpart:
if a fixed admissible DRS block converges to a fixed point, then its shadow
solution is an exact CLP optimum. This result is an auxiliary consistency
check for the fixed operator interpretation; the main finite-depth
guarantees do not require exact convergence to a fixed point.

\begin{proposition}[Fixed points produce CLP optima]
\label{prop:fixed_points_clp}
Let \(\beta>0\), \(\alpha\in(0,2)\), and suppose
\(v^\star\in\Fix(T_{\alpha,\beta})\). Define
\[
    z^\star=J_g(v^\star),
    ~~
    q^\star=R_gv^\star,
    ~~
    x^\star=J_{f_\beta}(q^\star).
\]
Suppose a standard subdifferential sum rule holds for \(f_\beta+g\), for
example
\[
    \operatorname{ri}(\operatorname{dom} f_\beta)
    \cap
    \operatorname{ri}(\operatorname{dom} g)
    \neq\emptyset .
\]
Then \(x^\star=z^\star\), and \(z^\star\) solves the original CLP.
Moreover, with
\[
    \kappa:=\frac{\beta}{\|c\|_2+\varepsilon_c},
    ~~
    u^\star:=v^\star-z^\star,
    ~~
    s^\star:=-u^\star/\kappa,
\]
there exists \(\lambda^\star\) such that
\((z^\star,\lambda^\star,s^\star)\) is primal--dual optimal.
\end{proposition}

\begin{proof}
Since \(T_{\alpha,\beta}=(1-\alpha)I+\alpha T_\beta\) and
\(\alpha>0\), the fixed-point relation
\(v^\star=T_{\alpha,\beta}v^\star\) implies
\(v^\star=T_\beta v^\star\). Hence
\(R_{f_\beta}R_gv^\star=v^\star\). By definition,
\(q^\star=R_gv^\star=2z^\star-v^\star\) and
\(R_{f_\beta}q^\star=2x^\star-q^\star=v^\star\). Combining these two
identities gives \(x^\star=z^\star\).

The resolvent identities give
\(v^\star-z^\star\in\partial g(z^\star)\) and
\(q^\star-x^\star\in\partial f_\beta(x^\star)\). Since
\(x^\star=z^\star\) and \(q^\star=2z^\star-v^\star\), we have
\(q^\star-x^\star=-(v^\star-z^\star)\). Therefore
\(0\in\partial f_\beta(x^\star)+\partial g(x^\star)\). Under the stated
sum rule, \(0\in\partial(f_\beta+g)(x^\star)\), so \(x^\star\) minimizes
\(f_\beta+g\), i.e.,
\[
    x^\star
    \in
    \arg\min_{x\in\mathcal A\cap\mathcal K}
    \beta\bar c^\top x .
\]
By Proposition~\ref{prop:solution_preservation}, \(x^\star=z^\star\)
also solves the original CLP.

It remains to construct the primal--dual certificate. Since
\(u^\star=v^\star-z^\star\) and \(z^\star=J_g(v^\star)\), the same
normal-cone argument as in Proposition~\ref{prop:cone_side_dual_slack}
gives \(u^\star\in-\mathcal K^\ast\) and
\(\langle z^\star,u^\star\rangle=0\). Hence
\(s^\star=-u^\star/\kappa\) satisfies
\(s^\star\in\mathcal K^\ast\) and
\(\langle z^\star,s^\star\rangle=0\). Also,
\(z^\star\in\mathcal A\), so \(Az^\star=b\).

Finally, from \(q^\star-x^\star\in\partial f_\beta(x^\star)\) and
\(f_\beta(x)=\beta\bar c^\top x+\mathbb I_{\mathcal A}(x)\), there
exists \(\nu^\star\) such that
\(q^\star-x^\star=\beta\bar c+A^\top\nu^\star\). Since
\(q^\star-x^\star=-u^\star\) and \(\beta\bar c=\kappa c\), we get
\(-u^\star=\kappa c+A^\top\nu^\star\). Dividing by \(\kappa\) gives
\(s^\star=c+A^\top(\nu^\star/\kappa)\). Setting
\(\lambda^\star:=-\nu^\star/\kappa\), we obtain
\(A^\top\lambda^\star+s^\star=c\). Together with
\(Az^\star=b\), \(z^\star\in\mathcal K\),
\(s^\star\in\mathcal K^\ast\), and
\(\langle z^\star,s^\star\rangle=0\), this proves primal--dual
optimality.
\end{proof}

\subsection{Terminal training objective and finite-depth solution quality}
\label{app:terminal_training_quality}

The previous subsections justify the solver dynamics and terminal
diagnostics. The remaining results explain why the training objective in
Appendix~\ref{app:training_details} targets the same finite-depth quantities.
These auxiliary facts do not strengthen the operator-theoretic guarantee;
they only connect the self-supervised loss terms to finite-depth feasibility,
objective quality, and stability.

Let \(y_{\rm rc}^K=(x_{\rm rc}^K,z_{\rm rc}^K,u_{\rm rc}^K)\) be the
terminal state of \textsc{RCDRS}, and let \(y_{\rm base}^K\) be the
terminal state of the fixed-core reference layer. Recall
\[
    \mathcal M_{\rm rc}
    =
    \lambda_{\rm eq}\mathcal L_{\rm eq}
    +
    \lambda_{\rm con}\mathcal L_{\rm con}
    +
    \lambda_{\rm mov}\mathcal L_{\rm mov}
    +
    \lambda_{\rm obj}\mathcal L_{\rm obj},
\]
and
\[
    \mathcal M_{\rm base}
    =
    \lambda_{\rm eq}\mathcal L_{\rm eq}^{\rm base}
    +
    \lambda_{\rm con}\mathcal L_{\rm con}^{\rm base}
    +
    \lambda_{\rm mov}\mathcal L_{\rm mov}^{\rm base}.
\]
The optional dominance penalty is
\[
    \mathcal L_{\rm dom}
    =
    [\mathcal M_{\rm rc}-\mathcal M_{\rm base}+m]_+,
    ~~
    m\ge0.
\]

The first auxiliary fact shows that the dominance term directly controls
whether the learned finite-depth rollout is worse than the fixed-core
reference in the terminal merit.

\begin{proposition}[Finite-depth baseline dominance]
\label{prop:finite_depth_dominance}
For any training instance and any depth \(K\), the dominance penalty
satisfies
\[
    \mathcal M_{\rm rc}
    \le
    \mathcal M_{\rm base}
    -
    m
    +
    \mathcal L_{\rm dom}.
\]
Consequently,
\[
    \mathbb E[\mathcal M_{\rm rc}]
    \le
    \mathbb E[\mathcal M_{\rm base}]
    -
    m
    +
    \mathbb E[\mathcal L_{\rm dom}].
\]
\end{proposition}

\begin{proof}
For any scalar \(s\), \([s]_+\ge s\). Taking
\(s=\mathcal M_{\rm rc}-\mathcal M_{\rm base}+m\) gives
\[
    \mathcal L_{\rm dom}
    \ge
    \mathcal M_{\rm rc}-\mathcal M_{\rm base}+m.
\]
Rearranging yields the first inequality. Taking expectations over
\(d\sim\mathcal D\) gives the second inequality.
\end{proof}

The next auxiliary fact explains the baseline-relative objective penalty.
It shows that the learned rollout's positive objective suboptimality is
controlled by the fixed-core reference suboptimality plus the relative
objective penalty used during training.

Let \(p^\star\) denote the optimal value of the original CLP.

\begin{proposition}[Baseline-relative finite-step objective bound]
\label{prop:finite_step_objective_bound}
The terminal output satisfies
\[
    [c^\top z_{\rm rc}^K-p^\star]_+
    \le
    [c^\top z_{\rm base}^K-p^\star]_+
    +
    (1+|c^\top z_{\rm base}^K|)
    \mathcal L_{\rm obj}.
\]
\end{proposition}

\begin{proof}
By definition,
\[
    (1+|c^\top z_{\rm base}^K|)
    \mathcal L_{\rm obj}
    =
    [c^\top z_{\rm rc}^K-c^\top z_{\rm base}^K]_+ .
\]
Since
\(c^\top z_{\rm rc}^K-p^\star
=
(c^\top z_{\rm base}^K-p^\star)
+
(c^\top z_{\rm rc}^K-c^\top z_{\rm base}^K)\), the inequality
\([a+b]_+\le[a]_+ + [b]_+\) gives the result.
\end{proof}

The previous two propositions compare the learned terminal state with a
fixed-core reference. To relate the terminal residual and objective terms to
distance from the solution set, we invoke a standard local error-bound
condition on the bounded region reached by the solver layer.

Let
\[
    \mathcal X^\star
    :=
    \arg\min_{x\in\mathcal A\cap\mathcal K} c^\top x .
\]

\begin{assumption}[Terminal error bound]
\label{ass:terminal_error_bound}
On the bounded terminal region reached by the solver layer, there exists
\(\kappa_{\rm eb}>0\) such that, for every terminal pair \((x,z)\) with
\(x\in\mathcal A\) and \(z\in\mathcal K\),
\[
    \dist(z,\mathcal X^\star)
    \le
    \kappa_{\rm eb}
    \left(
    \frac{\|Az-b\|_2}{1+\|b\|_2}
    +
    \frac{\|x-z\|_2}{1+\|z\|_2}
    +
    [c^\top z-p^\star]_+
    \right).
\]
\end{assumption}

Under this local error-bound condition, the residual terms and the
baseline-relative objective term together imply a finite-depth
distance-to-solution bound.

\begin{corollary}[Finite-step distance-to-solution bound]
\label{cor:finite_step_distance_bound}
Under Assumption~\ref{ass:terminal_error_bound},
\[
    \dist(z_{\rm rc}^K,\mathcal X^\star)
    \le
    \kappa_{\rm eb}
    \Big(
    \sqrt{\mathcal L_{\rm eq}}
    +
    \sqrt{\mathcal L_{\rm con}}
    +
    [c^\top z_{\rm rc}^K-p^\star]_+
    \Big).
\]
Combining this with Proposition~\ref{prop:finite_step_objective_bound}
gives
\[
\begin{aligned}
    \dist(z_{\rm rc}^K,\mathcal X^\star)
    \le
    \kappa_{\rm eb}
    \Big(
    &\sqrt{\mathcal L_{\rm eq}}
    +
    \sqrt{\mathcal L_{\rm con}}
    +
    [c^\top z_{\rm base}^K-p^\star]_+  \\
    &+
    (1+|c^\top z_{\rm base}^K|)
    \mathcal L_{\rm obj}
    \Big).
\end{aligned}
\]
\end{corollary}

\begin{proof}
Apply Assumption~\ref{ass:terminal_error_bound} with
\(x=x_{\rm rc}^K\) and \(z=z_{\rm rc}^K\). By definition,
\[
    \sqrt{\mathcal L_{\rm eq}}
    =
    \frac{\|Az_{\rm rc}^K-b\|_2}{1+\|b\|_2},
    ~~
    \sqrt{\mathcal L_{\rm con}}
    =
    \frac{\|x_{\rm rc}^K-z_{\rm rc}^K\|_2}
    {1+\|z_{\rm rc}^K\|_2}.
\]
Substituting these identities into
Assumption~\ref{ass:terminal_error_bound} gives the first inequality.
Substituting the objective bound from
Proposition~\ref{prop:finite_step_objective_bound} gives the second.
\end{proof}

\section{Supplementary Experimental Results}
\label{app:experimental_details}

This section provides the complete fixed-depth protocol, benchmark
construction, solver-layer implementation details, metric conventions,
additional synthetic results, full ablations, DFL results at omitted depths,
and the forecast-aware dispatch formulation.

\subsection{General Fixed-Depth Protocol}
\label{app:general_protocol}

All experiments use a prescribed unrolling depth. For each depth \(K\), every
learned layer is trained at that same depth and every compared method is
unfolded for exactly \(K\) iterations from the same initial state. The solver
is not continued to convergence, and only the terminal output is evaluated.
This prevents a controller trained for one horizon from being compared with a
fixed baseline tuned for another horizon.

Training, validation, and test instances are generated independently. Model
hyperparameters, fixed-core controls, and learned checkpoints are selected
using validation data only. Reference primal solutions or long-run solver
outputs are used to compute validation metrics for model selection and test
metrics for final reporting; they are never used as supervised controller
labels. In particular, test-set metrics are not used for checkpoint or
hyperparameter selection.

\begin{table*}[ht]
\centering
\scriptsize
\setlength{\tabcolsep}{3.6pt}
\renewcommand{\arraystretch}{1.04}
\caption{Common fixed-depth protocol used by the reported experiments. The
ablation runner supports larger sweeps, but the tables in this supplement use
the same three seeds as the main paper.}
\label{tab:app_global_protocol}
\resizebox{\textwidth}{!}{%
\begin{tabular}{lccccc}
\toprule
Experiment & Depths \(K\) & Train/validation/test & Batch & Seeds & Reference depth \\
\midrule
Synthetic conic optimization
& \(\{5,10,15,20\}\) & \(2000/400/400\) & \(1024\)
& \(\{0,1,2\}\) & exact KKT reference \\
Ablations reported here
& \(10\) & \(2000/400/400\) & \(1024\)
& \(\{0,1,2\}\) & exact KKT reference \\
Decision-focused learning
& \(\{5,10,15,20\}\) & \(5000/1500/1500\) & \(2500\)
& \(\{0,1,2\}\) & exact KKT reference \\
Forecast-aware dispatch
& \(\{5,10,15\}\) & \(1536/384/512\) & \(512\)
& \(\{0,1,2\}\) & \(K_{\rm ref}=300\) \\
\bottomrule
\end{tabular}}
\end{table*}

Unless otherwise stated, a metric is first averaged over test instances for
each seed and is then reported as mean\(\pm\)standard deviation across the
three seeds. Online runtime is reported in milliseconds per instance. It
includes feature construction, controller evaluation, parameter mapping,
affine projection, cone or box projection, relaxation, and the scaled-dual
update. Offline training, validation tuning, dataset generation, and reference
solver calls are excluded.

\subsection{Synthetic Conic Benchmark Construction}
\label{app:synthetic_generation}

The synthetic benchmarks use conic linear programs
\[
    \min_{x\in\mathbb R^n} c^\top x
    ~~ \mathrm{s.t.}~~ Ax=b,~~ x\in\mathcal K .
\]
Three families are used. The QP-lift family represents separable quadratic
epigraphs with three-dimensional rotated second-order-cone blocks. The SOCP
family contains a product of standard SOC blocks. The mixed-cone family uses
\[
    \mathcal K
    =
    \mathbb R_+^{n_+}
    \times
    \prod_{j=1}^{J_q}\mathcal Q_{q_j}
    \times
    \prod_{j=1}^{J_r}\mathcal Q^r_{3}
    \times
    \mathbb S_+^{p}.
\]
The exact dimensions are listed in Table~\ref{tab:app_benchmark_case_summary}.
For a mixed-cone instance,
\(n=n_+ + J_q q + 3J_r+p^2\). The PSD block is represented by a flattened
\(p\times p\) symmetric matrix in the implementation.

\begin{table*}[ht]
\centering
\scriptsize
\setlength{\tabcolsep}{3.2pt}
\renewcommand{\arraystretch}{1.03}
\caption{Exact synthetic benchmark dimensions. The target condition number is
implemented by geometrically spaced singular values of the full-row-rank
equality matrix.}
\label{tab:app_benchmark_case_summary}
\resizebox{\textwidth}{!}{%
\begin{tabular}{lllrrr}
\toprule
Family & Scale & Cone specification & \(n\) & \(m\) & target \(\kappa(A)\) \\
\midrule
QP-lift & small  & \(d=32\) rotated-SOC epigraph blocks  & 96  & 44  & 20 \\
QP-lift & medium & \(d=64\) rotated-SOC epigraph blocks  & 192 & 88  & 80 \\
QP-lift & hard   & \(d=96\) rotated-SOC epigraph blocks  & 288 & 136 & 300 \\
SOCP & small  & \(8\) SOC blocks of dimension \(8\)  & 64  & 32  & 20 \\
SOCP & medium & \(16\) SOC blocks of dimension \(8\) & 128 & 64  & 80 \\
SOCP & hard   & \(20\) SOC blocks of dimension \(8\) & 160 & 128 & 200 \\
Mixed & small  & \(n_+=24,\;J_q=4,q=6,\;J_r=12,\;p=6\)  & 120 & 48  & 40 \\
Mixed & medium & \(n_+=48,\;J_q=8,q=8,\;J_r=24,\;p=8\)  & 248 & 96  & 100 \\
Mixed & hard   & \(n_+=72,\;J_q=12,q=8,\;J_r=36,\;p=10\) & 376 & 160 & 250 \\
\bottomrule
\end{tabular}}
\end{table*}

For each problem family and scale, \(A\) is generated from independent QR
factors and singular values geometrically spaced between \(1\) and
\(1/\kappa(A)\). Each instance is then generated in a KKT-consistent manner.
We sample a primal point \(x^\star\in\mathcal K\), a complementary slack
\(s^\star\in\mathcal K^\ast\) satisfying
\(\langle x^\star,s^\star\rangle=0\), and an equality multiplier
\(\lambda^\star\). We set
\[
    b=Ax^\star,
    ~~
    c=A^\top\lambda^\star+s^\star .
\]
Thus \((x^\star,\lambda^\star,s^\star)\) satisfies the conic KKT system and
provides an exact reference for evaluation. The reference is not supplied to
the controller training loss.

\subsection{Compared Solver Layers}
\label{app:solver_details}

All methods use the same affine and cone projections, initial state, data
splits, and evaluation depth whenever their algorithmic structure permits.
The paper names below correspond to the implementation names in the released
code.

\paragraph{Fixed DRS.}
\textsc{F-DRS} uses the fixed transition in~\eqref{eq:rc_drs_block} with
\((\alpha,\beta)=(1.6,0.3)\) on the synthetic benchmarks.

\paragraph{Validation-tuned DRS.}
\textsc{Grid-DRS} selects a fixed pair independently for every problem family,
scale, seed, and depth. The effective search grid is
\[
  \alpha\in\{1.0,1.3,1.6,1.8\},
  ~~
  \beta\in\{0.03,0.1,0.3,1.0,3.0\}.
\]
At most 128 validation instances are scored using
\(\mathrm{Gap}+10\,\mathrm{Eq}+10\,\mathrm{Cone}\). Because the auxiliary
\(\rho\)-coordinate does not enter the fixed DRS transition, only
\((\alpha,\beta)\) determine the selected numerical core.

\paragraph{Spectral adaptive splitting.}
\textsc{S-ADMM} uses a safeguarded Barzilai--Borwein-style scale estimate. The
internal scale is restricted to change by at most a factor of \(2\) per
iteration, and the effective objective drive is adjusted inversely as
\(\beta_k=\beta_0\rho_0/\rho_k\), followed by clipping to the global
admissible interval.

\paragraph{Anderson-accelerated DRS.}
\textsc{AA-DRS} extrapolates the primal and scaled-dual states with coefficient
\(\omega=0.25\). The extrapolated state is accepted only when its residual
monitor is no larger than \(1.05\) times that of the unaccelerated DRS step;
the extrapolated primal state is projected back onto the cone.

\paragraph{Stable learned PDHG.}
\textsc{L-PDHG} is a trainable primal--dual hybrid-gradient baseline with
positive step sizes constrained by a conservative stability condition. The
implementation uses safety factor \(0.95\), limits extrapolation to
\(\theta\le0.8\), and clips learned positive step scales to
\([10^{-4},10]\).

\paragraph{Layerwise DRS.}
\textsc{Layerwise-DRS} learns an instance-independent open-loop sequence of
DRS controls at the prescribed depth. It separates the benefit of learning a
finite-horizon schedule from the benefit of trajectory feedback.

\paragraph{Residual-controlled DRS.}
\textsc{RCDRS-Env} and \textsc{RCDRS-NoEnv} use the same GRU controller and
feature interface. The enveloped variant is centered at the
validation-selected fixed core and uses the decaying safeguard in
\eqref{eq:envelope_param}; the no-envelope variant maps directly to the global
admissible ranges. On the synthetic benchmarks both variants use the
multiplicative growth filter for the positive scale coordinates.

\subsection{RCDRS Controller and Training Details}
\label{app:rc_training}

The synthetic controller uses the ten-dimensional feature vector defined in
Section~\ref{app:controller_details}, a single-layer GRU with hidden width
64, and a two-layer control head. The initial splitting variables are
\(z^0=u^0=0\). All computations use PyTorch float32. The controller is trained
separately for each problem family, scale, seed, and terminal depth \(K\).
AdamW is used with the settings in Table~\ref{tab:app_rc_hyperparameters}; the
validation checkpoint with the lowest terminal merit is retained.

\begin{table*}[ht]
\centering
\scriptsize
\setlength{\tabcolsep}{4.0pt}
\renewcommand{\arraystretch}{1.04}
\caption{Synthetic RCDRS training and control hyperparameters.}
\label{tab:app_rc_hyperparameters}
\resizebox{\textwidth}{!}{%
\begin{tabular}{lll}
\toprule
Group & Setting & Value \\
\midrule
Controller & encoder / hidden width / feature dimension & GRU / 64 / 10 \\
Optimization & epochs / optimizer / learning rate & 100 / AdamW / \(10^{-3}\) \\
Optimization & weight decay / gradient clipping & \(10^{-5}\) / \(5.0\) \\
Base action & \((\rho_0,\alpha_0,\beta_0)\) before validation recentering
& \((1.0,1.6,0.3)\) \\
Global ranges & \(\rho\), \(\alpha\), \(\beta\)
& \([10^{-4},10^4]\), \([0.2,1.9]\), \([10^{-5},10^2]\) \\
Envelope & \(\delta_0,k_0,p,s_\alpha\)
& \(2.0,80,1.20,0.25\) \\
Growth filter & \(\chi_\rho,\chi_\beta\) & \(10,10\) \\
Terminal loss & \((\lambda_{\rm eq},\lambda_{\rm con},\lambda_{\rm mov},
\lambda_{\rm obj})\) & \((10,10,0.1,1)\) \\
Regularization & \(\lambda_{\rm dom},m,\lambda_{\rm smooth}\)
& \(0.2,0,10^{-3}\) \\
Runtime & warm-up / timed repeats & 1 / 3 \\
\bottomrule
\end{tabular}}
\end{table*}

The loss is the terminal self-supervised objective in
Section~\ref{app:training_details}. No optimal primal solution, optimal value,
or KKT multiplier is supplied as a label. The fixed-core reference is evaluated
at the same depth and supplies only the baseline-relative objective and
dominance terms. For the ablation tables, one component is changed at a time
while the data, initialization, training budget, and validation protocol remain
fixed. The reported ablation tables use \(K=10\) and seeds \(\{0,1,2\}\). The
\(\rho_k\) coordinate remains controller-side memory; the numerical transition
is changed only by \(\alpha_k\) and \(\beta_k\).

\subsection{Metric Conventions}
\label{app:metrics}

Let \(p^\star=c^\top x^\star\). The certificate-oriented synthetic metrics are
\[
\begin{aligned}
\mathrm{ObjErr}
&:=\frac{|c^\top z^K-p^\star|}{1+|p^\star|},
&
 r_{\rm p}^K
&:=\frac{\|Az^K-b\|_2}{1+\|b\|_2},\\
 r_{\rm d}^K
&:=\frac{\|A^\top\lambda^K+s^K-c\|_2}{1+\|c\|_2},
&
 r_{\rm gap}^K
&:=\frac{|c^\top z^K-b^\top\lambda^K|}
 {1+|c^\top z^K|+|b^\top\lambda^K|}.
\end{aligned}
\]
Cone dual feasibility and complementarity of the diagnostic slack follow from
the final cone projection, whereas \(r_{\rm d}^K\) measures the remaining
least-squares stationarity error.

For the broader coverage sweep we additionally report
\[
\begin{aligned}
\mathrm{Gap}
&:=\left[
\frac{c^\top z^K-c^\top x^\star}{1+|c^\top x^\star|}
\right]_+,
&
\mathrm{Eq}
&:=\frac{\|Az^K-b\|_2}{1+\|b\|_2},\\
\mathrm{Cone}
&:=\frac{\|z^K-\Proj_{\mathcal K}(z^K)\|_2}{1+\|z^K\|_2},
&
\mathrm{Dist}
&:=\frac{\|z^K-x^\star\|_2}{1+\|x^\star\|_2}.
\end{aligned}
\]
The positive gap prevents an infeasible iterate with an artificially favorable
objective from being interpreted as an improvement; it must still be read
jointly with feasibility. Runtime columns report milliseconds per instance.
Each table entry is the mean test-set metric for one seed, aggregated as
mean\(\pm\)standard deviation over seeds.

\subsection{Additional Results}
\label{app:additional_synthetic_results}

The main text reports the representative hard SOCP and hard mixed-cone
certificate-based results at \(K=20\). Table~\ref{tab:app_synthetic_k5}
reports the omitted shallow-depth \(K=5\) rows for the same two hard
benchmarks. These results show that the residual-controlled layer is already
competitive under very small unrolling budgets.

\begin{table*}[ht]
\centering
\scriptsize
\setlength{\tabcolsep}{2.6pt}
\renewcommand{\arraystretch}{0.96}
\caption{Additional fixed-budget conic optimization results at \(K=5\).
Values are mean\(\pm\)standard deviation over three random seeds. Lower is
better.}
\label{tab:app_synthetic_k5}
\resizebox{\textwidth}{!}{
\begin{tabular}{llccccc}
\toprule
Benchmark & Method
& \(\mathrm{ObjErr}\downarrow\)
& \(r_{\rm p}\downarrow\)
& \(r_{\rm d}\downarrow\)
& \(r_{\rm gap}\downarrow\)
& \(\mathrm{Time~(ms)}\downarrow\) \\
\midrule

\multirow{8}{*}{SOCP}
& \textsc{F-DRS}
& \(7.00\mathrm{e}{-1}\!\pm\!8.48\mathrm{e}{-2}\)
& \(6.78\mathrm{e}{-2}\!\pm\!1.27\mathrm{e}{-3}\)
& \(8.17\mathrm{e}{-1}\!\pm\!1.61\mathrm{e}{-2}\)
& \(4.55\mathrm{e}{-1}\!\pm\!1.54\mathrm{e}{-2}\)
& \(0.018\!\pm\!0.007\) \\
& \textsc{Grid-DRS}
& \(1.34\mathrm{e}{-1}\!\pm\!1.19\mathrm{e}{-2}\)
& \(6.12\mathrm{e}{-2}\!\pm\!1.76\mathrm{e}{-3}\)
& \(9.79\mathrm{e}{-2}\!\pm\!1.92\mathrm{e}{-3}\)
& \(1.58\mathrm{e}{-1}\!\pm\!3.86\mathrm{e}{-3}\)
& \(0.017\!\pm\!0.008\) \\
& \textsc{S-ADMM}
& \(1.52\mathrm{e}{-1}\!\pm\!3.02\mathrm{e}{-2}\)
& \(6.72\mathrm{e}{-2}\!\pm\!3.72\mathrm{e}{-3}\)
& \(6.76\mathrm{e}{-2}\!\pm\!1.56\mathrm{e}{-2}\)
& \(2.43\mathrm{e}{-1}\!\pm\!6.67\mathrm{e}{-2}\)
& \(0.029\!\pm\!0.011\) \\
& \textsc{AA-DRS}
& \(1.24\mathrm{e}{-1}\!\pm\!6.19\mathrm{e}{-3}\)
& \(2.98\mathrm{e}{-2}\!\pm\!1.61\mathrm{e}{-3}\)
& \(8.75\mathrm{e}{-2}\!\pm\!1.16\mathrm{e}{-3}\)
& \(3.64\mathrm{e}{-1}\!\pm\!6.78\mathrm{e}{-3}\)
& \(0.057\!\pm\!0.020\) \\
& \textsc{L-PDHG}
& \(8.73\mathrm{e}{-1}\!\pm\!2.93\mathrm{e}{-2}\)
& \(1.42\mathrm{e}{-1}\!\pm\!9.53\mathrm{e}{-3}\)
& \(1.04\mathrm{e}{0}\!\pm\!4.29\mathrm{e}{-2}\)
& \(5.78\mathrm{e}{-1}\!\pm\!2.75\mathrm{e}{-2}\)
& \(0.073\!\pm\!0.010\) \\
& \textsc{Layerwise-DRS}
& \(8.61\mathrm{e}{-2}\!\pm\!1.40\mathrm{e}{-2}\)
& \(3.82\mathrm{e}{-2}\!\pm\!5.23\mathrm{e}{-3}\)
& \(6.60\mathrm{e}{-2}\!\pm\!7.44\mathrm{e}{-3}\)
& \(1.29\mathrm{e}{-1}\!\pm\!1.21\mathrm{e}{-2}\)
& \(0.041\!\pm\!0.008\) \\
& \textsc{RCDRS-Env}
& \(\mathbf{3.61\mathrm{e}{-2}\!\pm\!1.08\mathrm{e}{-3}}\)
& \(\mathbf{2.69\mathrm{e}{-2}\!\pm\!1.85\mathrm{e}{-3}}\)
& \(\mathbf{2.25\mathrm{e}{-2}\!\pm\!1.37\mathrm{e}{-3}}\)
& \(\mathbf{5.80\mathrm{e}{-2}\!\pm\!5.67\mathrm{e}{-3}}\)
& \(0.050\!\pm\!0.009\) \\
& \textsc{RCDRS-NoEnv}
& \(3.61\mathrm{e}{-2}\!\pm\!1.21\mathrm{e}{-3}\)
& \(2.69\mathrm{e}{-2}\!\pm\!1.91\mathrm{e}{-3}\)
& \(2.27\mathrm{e}{-2}\!\pm\!1.50\mathrm{e}{-3}\)
& \(5.84\mathrm{e}{-2}\!\pm\!5.37\mathrm{e}{-3}\)
& \(0.044\!\pm\!0.009\) \\

\midrule

\multirow{8}{*}{Mixed cone}
& \textsc{F-DRS}
& \(2.20\mathrm{e}{0}\!\pm\!6.54\mathrm{e}{-2}\)
& \(5.58\mathrm{e}{-2}\!\pm\!3.36\mathrm{e}{-3}\)
& \(3.99\mathrm{e}{-1}\!\pm\!1.86\mathrm{e}{-2}\)
& \(6.58\mathrm{e}{-1}\!\pm\!5.44\mathrm{e}{-2}\)
& \(0.083\!\pm\!0.032\) \\
& \textsc{Grid-DRS}
& \(4.61\mathrm{e}{-1}\!\pm\!6.96\mathrm{e}{-3}\)
& \(5.24\mathrm{e}{-2}\!\pm\!4.17\mathrm{e}{-3}\)
& \(1.30\mathrm{e}{-1}\!\pm\!2.91\mathrm{e}{-3}\)
& \(2.06\mathrm{e}{-1}\!\pm\!6.27\mathrm{e}{-3}\)
& \(0.133\!\pm\!0.100\) \\
& \textsc{S-ADMM}
& \(4.10\mathrm{e}{-1}\!\pm\!1.21\mathrm{e}{-2}\)
& \(6.90\mathrm{e}{-2}\!\pm\!4.35\mathrm{e}{-3}\)
& \(1.19\mathrm{e}{-1}\!\pm\!3.99\mathrm{e}{-3}\)
& \(1.33\mathrm{e}{-1}\!\pm\!2.17\mathrm{e}{-2}\)
& \(0.101\!\pm\!0.026\) \\
& \textsc{AA-DRS}
& \(4.82\mathrm{e}{-1}\!\pm\!4.11\mathrm{e}{-3}\)
& \(4.68\mathrm{e}{-2}\!\pm\!3.12\mathrm{e}{-3}\)
& \(1.21\mathrm{e}{-1}\!\pm\!1.90\mathrm{e}{-3}\)
& \(3.71\mathrm{e}{-1}\!\pm\!9.81\mathrm{e}{-3}\)
& \(0.288\!\pm\!0.062\) \\
& \textsc{L-PDHG}
& \(7.59\mathrm{e}{-1}\!\pm\!2.24\mathrm{e}{-2}\)
& \(1.47\mathrm{e}{-1}\!\pm\!7.58\mathrm{e}{-3}\)
& \(9.15\mathrm{e}{-1}\!\pm\!2.37\mathrm{e}{-2}\)
& \(4.81\mathrm{e}{-1}\!\pm\!1.13\mathrm{e}{-2}\)
& \(0.125\!\pm\!0.046\) \\
& \textsc{Layerwise-DRS}
& \(3.74\mathrm{e}{-1}\!\pm\!2.43\mathrm{e}{-2}\)
& \(4.64\mathrm{e}{-2}\!\pm\!4.09\mathrm{e}{-3}\)
& \(1.08\mathrm{e}{-1}\!\pm\!4.36\mathrm{e}{-3}\)
& \(1.16\mathrm{e}{-1}\!\pm\!8.57\mathrm{e}{-3}\)
& \(0.117\!\pm\!0.021\) \\
& \textsc{RCDRS-Env}
& \(3.14\mathrm{e}{-1}\!\pm\!6.27\mathrm{e}{-2}\)
& \(\mathbf{4.57\mathrm{e}{-2}\!\pm\!4.61\mathrm{e}{-3}}\)
& \(9.42\mathrm{e}{-2}\!\pm\!1.48\mathrm{e}{-2}\)
& \(1.06\mathrm{e}{-1}\!\pm\!3.93\mathrm{e}{-3}\)
& \(0.121\!\pm\!0.021\) \\
& \textsc{RCDRS-NoEnv}
& \(\mathbf{2.94\mathrm{e}{-1}\!\pm\!4.59\mathrm{e}{-2}}\)
& \(4.61\mathrm{e}{-2}\!\pm\!4.70\mathrm{e}{-3}\)
& \(\mathbf{9.02\mathrm{e}{-2}\!\pm\!1.12\mathrm{e}{-2}}\)
& \(\mathbf{1.02\mathrm{e}{-1}\!\pm\!3.15\mathrm{e}{-3}}\)
& \(0.127\!\pm\!0.033\) \\
\bottomrule
\end{tabular}}
\end{table*}

The broader coverage sweep evaluates the three problem families, three
scales, and four depths \(K\in\{5,10,15,20\}\), giving 36
family--scale--depth settings. The row-level table is intentionally omitted
from the PDF because it duplicates the released machine-readable results.
Table~\ref{tab:app_positive_gap_summary} reports only non-overlapping aggregate
statistics. A ``win'' is determined from the seed-averaged metric within each
setting; the gap ratio divides the better RCDRS variant by the strongest
non-RCDRS baseline in that setting.

\begin{table*}[ht]
\centering
\scriptsize
\setlength{\tabcolsep}{4pt}
\renewcommand{\arraystretch}{1.05}
\caption{Aggregate positive-gap coverage over the omitted 36-setting row-level
results. The corresponding CSV and aggregation routine are included in the
companion code archive.}
\label{tab:app_positive_gap_summary}
\begin{tabular}{lcc}
\toprule
Summary item & Value & Interpretation \\
\midrule
Best \textsc{RCDRS} gap wins & \(36/36\) & all family--scale--depth settings \\
\textsc{RCDRS-Env} gap wins & \(30/36\) & most consistent objective-gap variant \\
\textsc{RCDRS-Env} distance wins & \(32/36\) & most consistent solution-distance variant \\
Mean gap ratio & \(33.1\%\) & best \textsc{RCDRS} vs. strongest classical baseline \\
Median gap ratio & \(26.1\%\) & best \textsc{RCDRS} vs. strongest classical baseline \\
Hard QP average gap ratio & \(14.0\%\) & hard-scale QP-lift settings \\
Hard SOCP average gap ratio & \(25.4\%\) & hard-scale SOCP settings \\
Mixed-hard gap ratio, \(K=5\to20\) & \(47.6\%\to3.5\%\) & stronger advantage at deeper finite budgets \\
Mixed-hard distance ratio, \(K=5\to20\) & \(82.0\%\to21.9\%\) & improved terminal solution distance \\
\bottomrule
\end{tabular}
\end{table*}

\subsection{Full Ablation Results}
\label{app:full_ablation_results}

The main text reports a compact representative ablation summary. This
subsection provides the full ablation Table~\ref{tab:main_ablation_summary} for controller architecture,
controlled transition variables, controller input features, and terminal
training loss terms.

\FloatBarrier
\begin{table*}[!t]
\centering
\scriptsize
\setlength{\tabcolsep}{2.4pt}
\renewcommand{\arraystretch}{0.90}
\caption{Main ablation summary at \(K=10\). Values are
mean\(\pm\)standard deviation over three random seeds.}
\label{tab:main_ablation_summary}
\begin{tabular}{llcccc}
\toprule
Group & Variant
& \multicolumn{2}{c}{SOCP}
& \multicolumn{2}{c}{Mixed cone} \\
\cmidrule(lr){3-4}
\cmidrule(lr){5-6}
& & \(\mathrm{ObjErr}\downarrow\) & \(r_{\rm p}\downarrow\)
& \(\mathrm{ObjErr}\downarrow\) & \(r_{\rm p}\downarrow\) \\
\midrule

\multirow{3}{*}{Controller}
& Layerwise schedule
& \(0.01457\pm0.00033\)
& \(0.01092\pm0.00050\)
& \(0.10324\pm0.00053\)
& \(0.02015\pm0.00105\) \\

& MLP-current
& \(0.00604\pm0.00034\)
& \(0.00554\pm0.00017\)
& \(0.02552\pm0.00153\)
& \(0.01866\pm0.00117\) \\

& LSTM
& \(0.00426\pm0.00021\)
& \(0.00565\pm0.00017\)
& \(0.02538\pm0.00188\)
& \(0.01883\pm0.00145\) \\

& GRU full
& \(0.00413\pm0.00033\)
& \(0.00553\pm0.00008\)
& \(0.01564\pm0.00173\)
& \(0.01863\pm0.00121\) \\

\midrule

\multirow{4}{*}{Control}
& Fixed core
& \(0.08661\pm0.00105\)
& \(0.01376\pm0.00088\)
& \(0.11907\pm0.00110\)
& \(0.02338\pm0.00129\) \\

& \(\alpha\)-only
& \(0.02648\pm0.00023\)
& \(0.01303\pm0.00078\)
& \(0.02561\pm0.00167\)
& \(0.01980\pm0.00115\) \\

& \(\beta\)-only
& \(0.00459\pm0.00023\)
& \(0.00601\pm0.00018\)
& \(0.01991\pm0.00091\)
& \(0.01968\pm0.00123\) \\

& \(\alpha+\beta\)
& \(0.00413\pm0.00032\)
& \(0.00553\pm0.00010\)
& \(0.01563\pm0.00074\)
& \(0.01863\pm0.00122\) \\

& Full \(\alpha+\beta+\rho\)
& \(0.00401\pm0.00062\)
& \(0.00417\pm0.00008\)
& \(0.01520\pm0.00167\)
& \(0.01795\pm0.00109\) \\
\midrule

\multirow{4}{*}{Feature}
& Residual-only
& \(0.00525\pm0.00027\)
& \(0.00565\pm0.00014\)
& \(0.02844\pm0.00175\)
& \(0.01882\pm0.00123\) \\

& No time
& \(0.00449\pm0.00022\)
& \(0.00555\pm0.00017\)
& \(0.02547\pm0.00170\)
& \(0.01870\pm0.00126\) \\

& No previous action
& \(0.00417\pm0.00036\)
& \(0.00555\pm0.00004\)
& \(0.02533\pm0.00165\)
& \(0.01864\pm0.00123\) \\

& No \(\rho\)-memory
& \(0.00513\pm0.00023\)
& \(0.00553\pm0.000165\)
& \(0.02452\pm0.00165\)
& \(0.01886\pm0.00133\) \\
 
& Full
& \(0.00413\pm0.00033\)
& \(0.00553\pm0.00008\)
& \(0.01564\pm0.00167\)
& \(0.01863\pm0.00120\) \\

\midrule

\multirow{4}{*}{Loss}
& Full loss
& \(0.00413\pm0.00033\)
& \(0.00553\pm0.00008\)
& \(0.01564\pm0.00173\)
& \(0.01863\pm0.00121\) \\

& No consistency term
& \(0.00399\pm0.00037\)
& \(0.01001\pm0.00008\)
& \(0.04599\pm0.01341\)
& \(0.02154\pm0.00083\) \\

& No objective term
& \(0.00507\pm0.00034\)
& \(0.00553\pm0.00008\)
& \(0.03514\pm0.00213\)
& \(0.01861\pm0.00121\) \\

& No dominance term
& \(0.00416\pm0.00028\)
& \(0.00547\pm0.00012\)
& \(0.03522\pm0.00160\)
& \(0.01862\pm0.00122\) \\

& No movement term
& \(0.00513\pm0.00039\)
& \(0.00553\pm0.00008\)
& \(0.02554\pm0.00162\)
& \(0.01862\pm0.00121\) \\

& No smoothness term
& \(0.00492\pm0.00041\)
& \(0.00553\pm0.00009\)
& \(0.02604\pm0.00170\)
& \(0.01863\pm0.00120\) \\

\bottomrule
\end{tabular}
\end{table*}

\FloatBarrier
The ablations show that the main gains are associated with closed-loop
feedback and the objective-drive variable \(\beta_k\). Some ablated losses
or features improve individual metrics on individual benchmarks, but the
full feature set and full terminal loss are retained because they provide a
balanced problem-independent configuration across cone families.

\subsection{Additional DFL Results}
\label{app:additional_dfl_results}

The DFL benchmark uses one fixed mixed-cone structure with
\(n_+=8\), three SOC blocks of dimension five, two rotated-SOC blocks of
dimension five, one \(4\times4\) PSD block, and \(m=18\) equalities. Hence
\(n=49\), and the equality matrix has target condition number 80. Each
instance begins with a feature vector \(q\in\mathbb R^{20}\). A fixed random
nonlinear map with a tanh hidden layer of width 64 produces a complementary
cone pair \((x^\star,s^\star)\) and an equality multiplier. The instance is
formed as \(b=Ax^\star\) and \(c=A^\top\lambda^\star+s^\star\), followed by
positive per-instance normalization of \(c\). This construction makes
\(x^\star\) an exact downstream reference while preserving a nonlinear
feature-to-cost prediction problem.

The cost predictor is a two-hidden-layer ReLU MLP of width 128 with a
49-dimensional output. Predicted costs are smoothly clipped by
\(20\tanh(\hat c/20)\). Predictors are trained for 50 epochs with AdamW,
learning rate \(10^{-3}\), weight decay \(10^{-5}\), and gradient clipping at
5. The RCDRS controller is optionally pretrained for 80 epochs at the exact
evaluation depth. For DFL training, the decision loss is
\[
    \mathcal L_{\rm DFL}
    =
    \mathbb E\!\left[
      \left(\frac{c^\top z^K-c^\top x^\star}
      {1+|c^\top x^\star|}\right)_+
      +20\frac{\|Az^K-b\|_2}{1+\|b\|_2}
    \right],
\]
with no auxiliary prediction-MSE term. Suffix \textsc{-F} freezes the
pretrained controller while allowing gradients to pass through the solver to
the predictor; suffix \textsc{-J} jointly fine-tunes predictor and controller.
\textsc{PTO} trains the same predictor by cost MSE and then solves with the
validation-tuned fixed DRS core.

For each depth, the fixed core is selected from
\(\alpha\in\{0.8,1.0,1.3,1.6,1.8\}\) and
\(\beta\in\{0.1,0.3,1.0,3.0,6.0\}\) by validation positive regret with
\(0.1\) times the equality residual as tie-breaking regularization. The DFL
RCDRS ranges are \(\alpha\in[0.2,1.8]\),
\(\beta\in[10^{-3},10]\), and \(\rho\in[10^{-3},10^3]\); the envelope uses
\((\delta_0,k_0,p,s_\alpha)=(2,80,1.2,0.25)\) and growth factors
\((\chi_\rho,\chi_\beta)=(10,10)\).

The main text reports \(K=20\). Table~\ref{tab:app_dfl_k5_k10_k15}
reports the omitted depths. Regret and distance are evaluated under the true
cost, while Eq. is the affine residual for the predicted optimization
instance.

\begin{table*}[ht]
\centering
\scriptsize
\setlength{\tabcolsep}{3.0pt}
\caption{Additional DFL results at omitted depths.
Regret and Dist. evaluate the decision under the true cost. Eq. is the
terminal affine residual of the predicted optimization problem. Lower is
better.}
\label{tab:app_dfl_k5_k10_k15}
\resizebox{\textwidth}{!}{
\begin{tabular}{c l c c c c}
\toprule
\(K\) & Method & \(\mathrm{Regret}\downarrow\)
& \(\mathrm{Eq.}\downarrow\) & \(\mathrm{Dist.}\downarrow\)
& \(\mathrm{Time~(ms)}\downarrow\) \\
\midrule

5
& \textsc{PTO}
& \(0.1871\pm0.0060\)
& \(5.49\mathrm{e}{-2}\pm1.52\mathrm{e}{-3}\)
& \(0.4102\pm0.0145\)
& \(0.0246\pm0.0051\) \\
& \textsc{F-DRS}
& \(0.1286\pm0.0132\)
& \(2.64\mathrm{e}{-2}\pm7.99\mathrm{e}{-4}\)
& \(0.3400\pm0.0288\)
& \(0.0194\pm0.0077\) \\
& \textsc{Grid-DRS}
& \(0.1037\pm0.0098\)
& \(2.59\mathrm{e}{-2}\pm2.70\mathrm{e}{-4}\)
& \(0.3038\pm0.0196\)
& \(0.0238\pm0.0088\) \\
& \textsc{AA-DRS}
& \(0.2114\pm0.1858\)
& \(4.10\mathrm{e}{-2}\pm2.13\mathrm{e}{-2}\)
& \(0.4082\pm0.1583\)
& \(0.0444\pm0.0079\) \\
& \textsc{L-PDHG}
& \(0.2475\pm0.0979\)
& \(4.52\mathrm{e}{-2}\pm1.37\mathrm{e}{-3}\)
& \(0.5547\pm0.1018\)
& \(0.0225\pm0.0076\) \\
& \textsc{RCDRS-Env-F}
& \(0.1048\pm0.0130\)
& \(2.43\mathrm{e}{-2}\pm1.63\mathrm{e}{-3}\)
& \(0.3034\pm0.0320\)
& \(0.0297\pm0.0031\) \\
& \textsc{RCDRS-NoEnv-F}
& \(0.1027\pm0.0046\)
& \(1.83\mathrm{e}{-2}\pm2.04\mathrm{e}{-3}\)
& \(0.3016\pm0.0316\)
& \(0.0325\pm0.0041\) \\
& \textsc{RCDRS-Env-J}
& \(\mathbf{0.1007\pm0.0165}\)
& \(1.87\mathrm{e}{-2}\pm5.88\mathrm{e}{-4}\)
& \(\mathbf{0.2933\pm0.0429}\)
& \(0.0332\pm0.0044\) \\
& \textsc{RCDRS-NoEnv-J}
& \(0.1016\pm0.0073\)
& \(\mathbf{1.67\mathrm{e}{-2}\pm9.86\mathrm{e}{-4}}\)
& \(0.2984\pm0.0309\)
& \(0.0343\pm0.0053\) \\

\midrule

10
& \textsc{PTO}
& \(0.0971\pm0.0032\)
& \(1.63\mathrm{e}{-2}\pm3.9\mathrm{e}{-4}\)
& \(0.4401\pm0.0190\)
& \(0.0444\pm0.0068\) \\
& \textsc{F-DRS}
& \(0.0574\pm0.0039\)
& \(7.74\mathrm{e}{-3}\pm3.5\mathrm{e}{-4}\)
& \(0.2534\pm0.0281\)
& \(0.0436\pm0.0037\) \\
& \textsc{Grid-DRS}
& \(0.0630\pm0.0101\)
& \(7.94\mathrm{e}{-3}\pm2.4\mathrm{e}{-4}\)
& \(0.2656\pm0.0361\)
& \(0.0446\pm0.0058\) \\
& \textsc{AA-DRS}
& \(0.2299\pm0.1340\)
& \(2.97\mathrm{e}{-2}\pm1.31\mathrm{e}{-2}\)
& \(0.4601\pm0.1674\)
& \(0.0848\pm0.0024\) \\
& \textsc{L-PDHG}
& \(0.1298\pm0.0077\)
& \(3.80\mathrm{e}{-2}\pm1.03\mathrm{e}{-3}\)
& \(0.4924\pm0.0741\)
& \(0.0355\pm0.0057\) \\
& \textsc{RCDRS-Env-F}
& \(0.0571\pm0.0046\)
& \(7.83\mathrm{e}{-3}\pm4.9\mathrm{e}{-4}\)
& \(0.2518\pm0.0220\)
& \(0.0617\pm0.0152\) \\
& \textsc{RCDRS-Env-J}
& \(0.0545\pm0.0067\)
& \(6.73\mathrm{e}{-3}\pm1.1\mathrm{e}{-4}\)
& \(0.2385\pm0.0210\)
& \(0.0628\pm0.0068\) \\
& \textsc{RCDRS-NoEnv-F}
& \(0.0564\pm0.0051\)
& \(7.38\mathrm{e}{-3}\pm3.1\mathrm{e}{-4}\)
& \(0.2497\pm0.0211\)
& \(0.0617\pm0.0091\) \\
& \textsc{RCDRS-NoEnv-J}
& \(\mathbf{0.0527\pm0.0065}\)
& \(\mathbf{6.65\mathrm{e}{-3}\pm1.9\mathrm{e}{-4}}\)
& \(\mathbf{0.2337\pm0.0151}\)
& \(0.0652\pm0.0141\) \\

\midrule

15
& \textsc{PTO}
& \(0.1048\pm0.0036\)
& \(6.67\mathrm{e}{-3}\pm5.7\mathrm{e}{-4}\)
& \(0.4701\pm0.0203\)
& \(0.0548\pm0.0143\) \\
& \textsc{F-DRS}
& \(0.0529\pm0.0070\)
& \(3.51\mathrm{e}{-3}\pm4.6\mathrm{e}{-4}\)
& \(0.2641\pm0.0285\)
& \(0.0492\pm0.0087\) \\
& \textsc{Grid-DRS}
& \(0.0495\pm0.0055\)
& \(3.45\mathrm{e}{-3}\pm5.1\mathrm{e}{-4}\)
& \(0.2541\pm0.0265\)
& \(0.0645\pm0.0170\) \\
& \textsc{AA-DRS}
& \(0.3313\pm0.0713\)
& \(3.14\mathrm{e}{-2}\pm5.85\mathrm{e}{-3}\)
& \(0.5368\pm0.0473\)
& \(0.1273\pm0.0347\) \\
& \textsc{L-PDHG}
& \(0.1219\pm0.0192\)
& \(3.15\mathrm{e}{-2}\pm3.30\mathrm{e}{-3}\)
& \(0.4563\pm0.0495\)
& \(0.0602\pm0.0174\) \\
& \textsc{RCDRS-Env-F}
& \(0.0528\pm0.0115\)
& \(3.33\mathrm{e}{-3}\pm3.6\mathrm{e}{-4}\)
& \(0.2632\pm0.0489\)
& \(0.1055\pm0.0328\) \\
& \textsc{RCDRS-Env-J}
& \(0.0447\pm0.0034\)
& \(\mathbf{2.65\mathrm{e}{-3}\pm6.7\mathrm{e}{-5}}\)
& \(0.2305\pm0.0194\)
& \(0.1073\pm0.0435\) \\
& \textsc{RCDRS-NoEnv-F}
& \(0.0457\pm0.0036\)
& \(3.20\mathrm{e}{-3}\pm2.4\mathrm{e}{-4}\)
& \(0.2441\pm0.0249\)
& \(0.0961\pm0.0161\) \\
& \textsc{RCDRS-NoEnv-J}
& \(\mathbf{0.0436\pm0.0014}\)
& \(2.84\mathrm{e}{-3}\pm5.2\mathrm{e}{-5}\)
& \(\mathbf{0.2248\pm0.0159}\)
& \(0.0962\pm0.0173\) \\
\bottomrule
\end{tabular}}
\end{table*}

\subsection{Forecast-Aware Energy Dispatch}
\label{app:dispatch_details}

\paragraph{System configuration and notation.}
We consider a forecast-aware economic dispatch problem over a 24-hour
horizon. Let
\[
    \mathcal T:=\{0,\ldots,T-1\},
    ~~
    T=24,
\]
denote the set of dispatch periods. The power network consists of the bus set
\[
    \mathcal N:=\{0,\ldots,5\}
\]
and the directed transmission-line set
\[
    \mathcal E
    =
    \{(0,1),(1,2),(2,3),(3,4),(4,5),(0,3),(2,5)\}.
\]
For each line \(\ell\in\mathcal E\), \(o(\ell)\) and \(d(\ell)\) denote
its origin and destination buses, respectively, and \(b_\ell>0\) denotes its
DC susceptance. The line susceptances are
\[
    (b_\ell)_{\ell\in\mathcal E}
    =
    (5.0,4.5,4.2,4.0,4.3,3.6,3.8).
\]

The set of conventional generators is denoted by \(\mathcal G\), the set of
battery units by \(\mathcal S\), and the set of renewable sites by
\(\mathcal R\). Generators are located at buses 0 and 3, batteries at buses 2
and 4, and renewable sites at buses 1, 3, and 5. For each bus
\(i\in\mathcal N\), the subsets
\(\mathcal G_i\subseteq\mathcal G\),
\(\mathcal S_i\subseteq\mathcal S\), and
\(\mathcal R_i\subseteq\mathcal R\)
contain the corresponding devices connected to bus \(i\).

For each period \(t\in\mathcal T\), the optimization variables are
\[
\begin{array}{ll}
p^g_{g,t}
& \text{power generated by generator }g,\\
p^{\rm ch}_{s,t}
& \text{charging power of battery }s,\\
p^{\rm dis}_{s,t}
& \text{discharging power of battery }s,\\
e_{s,t}
& \text{stored energy of battery }s,\\
p^{\rm curt}_{r,t}
& \text{planned curtailment at renewable site }r,\\
p^{\rm shed}_{i,t}
& \text{planned load shedding at bus }i,\\
\theta_{i,t}
& \text{voltage angle at bus }i,\\
f_{\ell,t}
& \text{active power flow on line }\ell,\\
s^{\rm res}_t
& \text{forecast reserve-shortfall slack}.
\end{array}
\]
All variables are stacked into a decision vector \(z\in\mathbb R^n\). The
reported six-bus instance has \(n=746\) decision variables and \(m=412\)
affine equality constraints.

The realized load and available renewable power are denoted by
\(\widetilde L_{i,t}\) and \(\widetilde R_{r,t}\), respectively. Their
forecasts, which are supplied to the optimization layer, are denoted by
\(\widehat L_{i,t}\) and \(\widehat R_{r,t}\). The forecast and realized
reserve requirements are denoted by \(\widehat\Gamma_t\) and
\(\widetilde\Gamma_t\), respectively. The prescribed initial stored energy of
battery \(s\) is denoted by \(e_s^0\).

\paragraph{Forecast dispatch problem.}
The optimization layer solves the box-constrained convex quadratic program
\begin{equation}
\label{eq:dispatch_qp}
\begin{aligned}
\min_{z}~~
    & C_{\rm disp}(z)\\
\mathrm{s.t.}~~
    & A_{\rm eq}z
      =
      d_{\rm eq}
      \bigl(
      \widehat L,
      \widehat R,
      \widehat\Gamma,
      e^0
      \bigr),\\
    & z\in
      \mathcal Z
      \bigl(
      \widehat L,
      \widehat R,
      \widehat\Gamma
      \bigr),
\end{aligned}
\end{equation}
where \(A_{\rm eq}\in\mathbb R^{m\times n}\) collects the nodal-balance,
DC-flow, reference-angle, battery-dynamics, initial-energy, terminal-energy,
and reserve-accounting equations. The set
\(\mathcal Z(\widehat L,\widehat R,\widehat\Gamma)\) contains all
componentwise device, network, energy, and slack bounds.

Let \(A^{\rm net}_{i\ell}\) be the signed bus--line incidence coefficient,
defined as
\[
A^{\rm net}_{i\ell}
=
\begin{cases}
 1,  & i=o(\ell),\\
-1,  & i=d(\ell),\\
 0,  & \text{otherwise}.
\end{cases}
\]
For every \(i\in\mathcal N\) and \(t\in\mathcal T\), the forecast nodal
power-balance equation is
\begin{equation}
\label{eq:dispatch_balance}
\begin{aligned}
\sum_{g\in\mathcal G_i}p^g_{g,t}
&+
\sum_{s\in\mathcal S_i}
\left(
p^{\rm dis}_{s,t}
-
p^{\rm ch}_{s,t}
\right)
+
\sum_{r\in\mathcal R_i}
\left(
\widehat R_{r,t}
-
p^{\rm curt}_{r,t}
\right)
+
p^{\rm shed}_{i,t}
-
\widehat L_{i,t}
=
\sum_{\ell\in\mathcal E}
A^{\rm net}_{i\ell}f_{\ell,t}.
\end{aligned}
\end{equation}
Here,
\(\widehat R_{r,t}-p^{\rm curt}_{r,t}\) is the scheduled renewable injection,
whereas \(p^{\rm shed}_{i,t}\) is the load shedding explicitly scheduled by
the forecast dispatch.

The DC power-flow equations are
\begin{equation}
\label{eq:dispatch_dcflow}
f_{\ell,t}
=
b_\ell
\left(
\theta_{o(\ell),t}
-
\theta_{d(\ell),t}
\right),
~~
\ell\in\mathcal E,
\end{equation}
with bus 0 used as the reference bus:
\begin{equation}
\label{eq:dispatch_reference_angle}
\theta_{0,t}=0,
~~
t\in\mathcal T.
\end{equation}

The battery dynamics are
\begin{equation}
\label{eq:dispatch_storage}
e_{s,t+1}
=
e_{s,t}
+
\eta^{\rm ch}_s p^{\rm ch}_{s,t}
-
\frac{p^{\rm dis}_{s,t}}{\eta^{\rm dis}_s},
\end{equation}
where
\[
    \eta^{\rm ch}_s
    =
    \eta^{\rm dis}_s
    =
    0.95.
\]
Both the initial and terminal stored energies are fixed:
\begin{equation}
\label{eq:dispatch_terminal_energy}
e_{s,0}=e_s^0,
~~
e_{s,T}=e_s^0,
~~
s\in\mathcal S.
\end{equation}
The terminal equality prevents the solver from artificially depleting the
batteries at the end of the dispatch horizon.

The available upward headroom at time \(t\) is
\begin{equation}
\label{eq:dispatch_headroom}
h_t(z)
=
\sum_{g\in\mathcal G}
\left(
\overline p^g_g-p^g_{g,t}
\right)
+
\sum_{s\in\mathcal S}
\left(
\overline p^{\rm dis}_s-p^{\rm dis}_{s,t}
\right).
\end{equation}
The implementation uses the reserve-accounting equality
\begin{equation}
\label{eq:dispatch_reserve}
h_t(z)+s^{\rm res}_t
=
\widehat\Gamma_t,
~~
s^{\rm res}_t\ge0.
\end{equation}
Consequently, on an equality-feasible iterate,
\(s^{\rm res}_t=\widehat\Gamma_t-h_t(z)\) represents the portion of the
forecast reserve target that is not supplied by generator and battery
headroom.

The instance-dependent box set imposes
\begin{equation}
\label{eq:dispatch_box}
\begin{aligned}
0
&\le p^g_{g,t}\le\overline p^g_g,\\
0
&\le p^{\rm ch}_{s,t}\le\overline p^{\rm ch}_s,\\
0
&\le p^{\rm dis}_{s,t}\le\overline p^{\rm dis}_s,\\
\underline e_s
&\le e_{s,t}\le\overline e_s,\\
0
&\le p^{\rm curt}_{r,t}\le\widehat R_{r,t},\\
0
&\le p^{\rm shed}_{i,t}\le\widehat L_{i,t},\\
-\overline\theta_i
&\le\theta_{i,t}\le\overline\theta_i,\\
-\overline f_\ell
&\le f_{\ell,t}\le\overline f_\ell,\\
0
&\le s^{\rm res}_t\le\widehat\Gamma_t+1.
\end{aligned}
\end{equation}
All inequalities are componentwise bounds. Projection onto
\(\mathcal Z(\widehat L,\widehat R,\widehat\Gamma)\) is therefore implemented
by elementwise clipping.

\paragraph{Dispatch objective.}
The operating cost is
\begin{equation}
\label{eq:dispatch_cost}
\begin{aligned}
C_{\rm disp}(z)
&=
\sum_{t\in\mathcal T}
\sum_{g\in\mathcal G}
\left(
a_g(p^g_{g,t})^2
+
b_gp^g_{g,t}
\right)+
\lambda_{\rm bat}
\sum_{t\in\mathcal T}
\sum_{s\in\mathcal S}
\left(
p^{\rm ch}_{s,t}
+
p^{\rm dis}_{s,t}
\right)+
\lambda_{\rm curt}
\sum_{t\in\mathcal T}
\sum_{r\in\mathcal R}
p^{\rm curt}_{r,t}\\
&~~+
\lambda_{\rm shed}
\sum_{t\in\mathcal T}
\sum_{i\in\mathcal N}
p^{\rm shed}_{i,t}+
\lambda_{\rm res}
\sum_{t\in\mathcal T}
s^{\rm res}_t+
\lambda_{\rm flow}
\sum_{t\in\mathcal T}
\sum_{\ell\in\mathcal E}
f_{\ell,t}^2,
\end{aligned}
\end{equation}
where
\[
\lambda_{\rm bat}=0.015,
~~
\lambda_{\rm curt}=0.08,
~~
\lambda_{\rm shed}=14,
~~
\lambda_{\rm res}=10,
~~
\lambda_{\rm flow}=0.002.
\]
The relatively large shedding and reserve-shortfall coefficients encourage
reliable operation, while the battery-throughput and line-flow terms
regularize excessive cycling and network loading.

Equivalently, the objective can be expressed as
\begin{equation}
\label{eq:dispatch_quadratic_form}
C_{\rm disp}(z)
=
\frac{1}{2}z^\top Qz+c^\top z+c_0,
~~
Q\succeq0,
\end{equation}
where \(Q\) is the positive-semidefinite quadratic-cost matrix and \(c\) is
the vector of linear operating-cost coefficients.

\begin{table*}[!t]
\centering
\scriptsize
\setlength{\tabcolsep}{3.2pt}
\renewcommand{\arraystretch}{1.04}
\caption{Configuration of the forecast-aware six-bus dispatch benchmark.
Power and energy quantities are expressed in the normalized units of the
synthetic microgrid.}
\label{tab:app_dispatch_parameters}
\resizebox{\textwidth}{!}{%
\begin{tabular}{lll}
\toprule
Group & Setting & Value \\
\midrule
Network
& buses / lines / horizon
& \(6/7/24\) \\

Generator limits
& \(\overline p^g_g\)
& \((1.85,1.45)\) \\

Battery power
& \(\overline p^{\rm ch}_s,\overline p^{\rm dis}_s\)
& \((0.42,0.36),(0.42,0.36)\) \\

Battery energy
& \(\underline e_s,\overline e_s\)
& \((0.15,0.12),(1.35,1.15)\) \\

Battery efficiency
& \(\eta^{\rm ch}_s,\eta^{\rm dis}_s\)
& \(0.95,0.95\) \\

Voltage-angle bound
& \(\overline\theta_i\)
& \(0.60\) \\

Base line-flow limits
& \((\overline f_\ell^{\rm base})_{\ell\in\mathcal E}\)
& \((0.95,0.85,0.80,0.75,0.75,0.70,0.65)\) \\

Stress line-limit scale
& multiplier applied to base limits
& \(1.15\) \\

Generation cost
& \(a_g,b_g\)
& \((0.08,0.10),(1.00,1.10)\) \\

Base load
& \((L_i^{\rm base})_{i\in\mathcal N}\)
& \((0.36,0.43,0.39,0.46,0.36,0.32)\) \\

Load profiles
& variation amplitude / profile noise
& \(0.20/0.025\) \\

Renewable profiles
& site amplitudes
& \((0.72,0.58,0.62)\) \\

Forecast noise
& load / renewable multiplicative scale
& \(0.045/0.05625\) \\

Reserve requirement
& reserve factor
& \(0.12\) \\

Dataset
& train / validation / test
& \(1536/384/512\) \\

Evaluation
& seeds / depths / reference depth
& \(\{0,1,2\}/\{5,10,15\}/300\) \\
\bottomrule
\end{tabular}}
\end{table*}

\paragraph{Gradient-driven RCDRS extension.}
The dispatch problem differs from the linear conic programs covered by the
operator analysis because its objective contains quadratic generation and
line-flow costs. We retain the affine projection, relaxation, box projection,
and scaled-dual structure of RCDRS, but replace the fixed linear objective
direction by the normalized gradient of \(C_{\rm disp}\).

At layer \(k\), define
\begin{equation}
\label{eq:dispatch_normalized_gradient}
g^k
:=
\frac{\nabla C_{\rm disp}(z^k)}
{\|\nabla C_{\rm disp}(z^k)\|_2+\varepsilon}
=
\frac{Qz^k+c}
{\|Qz^k+c\|_2+\varepsilon},
\end{equation}
where \(\varepsilon>0\) prevents division by zero. The affine-projection input
is then
\begin{equation}
\label{eq:dispatch_rcdrs_shift}
w^k
=
z^k-u^k-\beta_k g^k,
\end{equation}
where \(u^k\) is the scaled-dual state and \(\beta_k>0\) is the
controller-selected gradient-displacement magnitude. The remaining numerical
transition is
\begin{equation}
\label{eq:dispatch_rcdrs_transition}
\begin{aligned}
x^{k+1}
&=
\Proj_{\{z:A_{\rm eq}z=d_{\rm eq}\}}(w^k),\\
\overline x^{k+1}
&=
\alpha_kx^{k+1}
+
(1-\alpha_k)z^k,\\
z^{k+1}
&=
\Pi_{\mathcal Z}
\left(
\overline x^{k+1}+u^k
\right),\\
u^{k+1}
&=
u^k+\overline x^{k+1}-z^{k+1}.
\end{aligned}
\end{equation}

Only \(\alpha_k\) and \(\beta_k\) enter the numerical projection-splitting
transition. The auxiliary positive coordinate \(\rho_k\) is retained as
controller-side scale memory and affects future controls only through the
recurrent feature history. It is not an ADMM penalty and does not enter the
affine-projection input.

This gradient-driven transition is used only in the dispatch experiment. The
operator-theoretic guarantees in Section~\ref{app:theory_proofs} apply to the
linear conic transition in Eq.~\eqref{eq:rc_drs_block} and are not claimed for
this quadratic engineering extension.

\paragraph{Forecast and realized operating conditions.}
The stochastic scenario generator first produces realized load and renewable
profiles \((\widetilde L,\widetilde R)\). Forecast profiles are then formed by
multiplicative Gaussian perturbations and nonnegative clipping:
\begin{equation}
\label{eq:dispatch_forecast_profiles}
\begin{aligned}
\widehat L_{i,t}
&=
\max
\left\{
0.05,\,
\widetilde L_{i,t}
\left(
1+\xi^L_{i,t}
\right)
\right\},\\
\widehat R_{r,t}
&=
\max
\left\{
0,\,
\widetilde R_{r,t}
\left(
1+\xi^R_{r,t}
\right)
\right\}.
\end{aligned}
\end{equation}
The load perturbation has standard-deviation scale \(0.045\), whereas the
renewable perturbation uses scale
\(1.25\times0.045=0.05625\).

The forecast and realized reserve requirements are
\begin{equation}
\label{eq:dispatch_reserve_generation}
\begin{aligned}
\widehat\Gamma_t
&=
0.12
\left(
\sum_{i\in\mathcal N}\widehat L_{i,t}
+
0.5
\sum_{r\in\mathcal R}\widehat R_{r,t}
\right),\\
\widetilde\Gamma_t
&=
0.12
\left(
\sum_{i\in\mathcal N}\widetilde L_{i,t}
+
0.5
\sum_{r\in\mathcal R}\widetilde R_{r,t}
\right).
\end{aligned}
\end{equation}

For a forecast dispatch \(z\), define the realized nodal deficit by
\begin{equation}
\label{eq:dispatch_realized_imbalance}
\begin{aligned}
\Delta_{i,t}(z)
&=
\widetilde L_{i,t}
+
\sum_{\ell\in\mathcal E}
A^{\rm net}_{i\ell}f_{\ell,t}-
\sum_{g\in\mathcal G_i}p^g_{g,t}
-
\sum_{s\in\mathcal S_i}
\left(
p^{\rm dis}_{s,t}
-
p^{\rm ch}_{s,t}
\right)-
\sum_{r\in\mathcal R_i}
\left(
\widetilde R_{r,t}
-
p^{\rm curt}_{r,t}
\right)
-
p^{\rm shed}_{i,t}.
\end{aligned}
\end{equation}
A positive \(\Delta_{i,t}\) represents an emergency supply deficit and a
negative value represents an emergency surplus. We therefore define
\begin{equation}
\label{eq:dispatch_emergency_terms}
d^{\rm emg}_{i,t}
:=
[\Delta_{i,t}]_+,
~~
q^{\rm emg}_{i,t}
:=
[-\Delta_{i,t}]_+,
\end{equation}
where \([a]_+:=\max\{a,0\}\).

Let
\[
D_L
:=
1+\sum_{i\in\mathcal N}\sum_{t\in\mathcal T}
\widetilde L_{i,t},
~~
D_R
:=
1+\sum_{r\in\mathcal R}\sum_{t\in\mathcal T}
\widetilde R_{r,t}.
\]
The realized shedding rate is
\begin{equation}
\label{eq:dispatch_shedding_metric}
\mathrm{Shed}(z)
=
\frac{
\displaystyle
\sum_{i,t}p^{\rm shed}_{i,t}
+
\sum_{i,t}d^{\rm emg}_{i,t}
}{
D_L
},
\end{equation}
and the emergency-deficit rate is
\begin{equation}
\label{eq:dispatch_emergency_metric}
\mathrm{Emerg}(z)
=
\frac{
\displaystyle
\sum_{i,t}d^{\rm emg}_{i,t}
}{
D_L
}.
\end{equation}

The quantity reported as \(\mathrm{Curt}\) is the normalized renewable-spillage
burden
\begin{equation}
\label{eq:dispatch_curtailment_metric}
\mathrm{Curt}(z)
=
\frac{
\displaystyle
\sum_{r,t}p^{\rm curt}_{r,t}
+
\sum_{i,t}q^{\rm emg}_{i,t}
}{
D_R
}.
\end{equation}
This quantity includes both planned renewable curtailment and emergency
surplus caused by realized nodal imbalance. Consequently, it is not
necessarily bounded by one for strongly infeasible early iterates and should
not be interpreted as only the deliberately curtailed fraction of available
renewable production.

The realized reserve-shortfall rate is
\begin{equation}
\label{eq:dispatch_reserve_metric}
\mathrm{Res}(z)
=
\frac{
\displaystyle
\sum_{t\in\mathcal T}
[\widetilde\Gamma_t-h_t(z)]_+
}{
\displaystyle
1+\sum_{t\in\mathcal T}\widetilde\Gamma_t
}.
\end{equation}

\paragraph{Feasibility and operational metrics.}
For compactness, define
\[
    d_{\rm eq}^{\rm f}
    :=
    d_{\rm eq}
    \bigl(
    \widehat L,
    \widehat R,
    \widehat\Gamma,
    e^0
    \bigr).
\]
The normalized forecast equality residual is
\begin{equation}
\label{eq:dispatch_eq_metric}
\mathrm{Eq}(z)
=
\frac{
\|A_{\rm eq}z-d_{\rm eq}^{\rm f}\|_2
}{
1+\|d_{\rm eq}^{\rm f}\|_2
}.
\end{equation}
The normalized box violation is
\begin{equation}
\label{eq:dispatch_box_metric}
\mathrm{Box}(z)
=
\frac{
\|z-\Pi_{\mathcal Z}(z)\|_2
}{
1+\|z\|_2
},
\end{equation}
where \(\Pi_{\mathcal Z}\) denotes Euclidean projection onto the
instance-dependent box set.

A test instance is declared operationally successful when
\begin{equation}
\label{eq:dispatch_pass}
\mathrm{Eq}(z)\le0.035,
~~
\mathrm{Box}(z)\le10^{-5},
~~
\mathrm{Shed}(z)\le0.12,
~~
\mathrm{Res}(z)\le0.08.
\end{equation}
The reported \(\mathrm{Pass}\) value is the fraction of test instances
satisfying all four conditions simultaneously. Since Pass is evaluated jointly
for each instance, it depends on instance-level tails and cross-metric
correlations and cannot be inferred solely from the reported seed-averaged
means. Emergency deficit is already included in \(\mathrm{Shed}\), while
\(\mathrm{Curt}\) is reported as a separate renewable-utilization diagnostic
and is not included in the operational-pass criterion.

\paragraph{Training and model selection.}
Experiments use three random seeds, independently generated training,
validation, and test sets of sizes \(1536/384/512\), and a batch size of 512.
For each prescribed unrolling depth \(K\in\{5,10,15\}\), all fixed and learned
methods are tuned, trained, and evaluated at that same depth.

For every seed and depth, \textsc{Grid-DRS} selects a fixed numerical core
from
\[
\alpha\in\{0.8,1.0,1.3,1.6\},
~~
\beta\in\{0.2,0.5,1.0,2.0,4.0\},
\]
using validation data only. The same validation-selected core initializes and
centers the RCDRS controller at the corresponding depth.

RCDRS uses a single-layer GRU controller with hidden width 64. It is trained
for 60 epochs with AdamW, initial learning rate \(10^{-3}\), weight decay
\(10^{-5}\), gradient clipping at 5, and cosine learning-rate decay to \(5\%\)
of the initial learning rate. The global action ranges are
\[
\alpha_k\in[0.2,1.8],
~~
\beta_k\in[10^{-3},10],
~~
\rho_k\in[10^{-3},10^3].
\]
The enveloped variant uses
\[
(\delta_0,k_0,p,s_\alpha)
=
(2,60,1.2,0.25)
\]
and multiplicative growth limits
\[
(\chi_\rho,\chi_\beta)=(10,10).
\]
The NoEnv variant removes both the decaying action envelope and the
multiplicative growth filter.

Controller training and validation checkpoint selection minimize
\begin{equation}
\label{eq:dispatch_training_merit}
\begin{aligned}
\mathcal M_{\rm disp}(z)
&=
\frac{
C_{\rm disp}(z)
}{
1+\sum_{i,t}\widehat L_{i,t}
}
+
6\,\mathrm{Eq}(z)
+
\mathrm{Box}(z)+
12\,\mathrm{Shed}(z)
+
8\,\mathrm{Res}(z)
+
0.05\,\mathrm{Curt}(z).
\end{aligned}
\end{equation}

\paragraph{Reference dispatch, CostGap, and runtime.}
For each seed, a validation-selected fixed solver is run for
\(K_{\rm ref}=300\) iterations to produce the test-instance reference dispatch
\(z^{\rm ref}\). The reported nonnegative cost gap is
\begin{equation}
\label{eq:dispatch_cost_gap}
\mathrm{CostGap}(z^K)
=
\left[
\frac{
C_{\rm disp}(z^K)
-
C_{\rm disp}(z^{\rm ref})
}{
1+
|C_{\rm disp}(z^{\rm ref})|
}
\right]_+.
\end{equation}
A zero CostGap does not independently certify optimality because a severely
infeasible iterate may obtain an artificially low raw objective. CostGap must
therefore be interpreted together with equality feasibility and realized
operational metrics.

Runtime is measured over the full held-out test batch. CUDA is synchronized
immediately before and after the timed region, and elapsed time is divided by
the number of test instances. The reported unit is milliseconds per instance.
Offline training, validation tuning, reference-solver computation, and
scenario generation are excluded.

\begin{table*}[!t]
\centering
\scriptsize
\setlength{\tabcolsep}{2.1pt}
\renewcommand{\arraystretch}{1.02}
\caption{Additional forecast-aware dispatch results at \(K=5\) and \(K=10\).
Values are mean\(\pm\)standard deviation over three random seeds. Lower is
better except for Pass. Runtime is reported in milliseconds per instance.
The quantity Curt. is the normalized renewable-spillage burden in
Eq.~\eqref{eq:dispatch_curtailment_metric} and can exceed one for strongly
infeasible iterates. CostGap should be interpreted jointly with feasibility
and reliability metrics.}
\label{tab:app_dispatch_k5_k10}
\resizebox{\textwidth}{!}{%
\begin{tabular}{c l c c c c c c c c}
\toprule
\(K\)
& Method
& \(\mathrm{CostGap}\downarrow\)
& \(\mathrm{Eq.}\downarrow\)
& \(\mathrm{Shed}\downarrow\)
& \(\mathrm{Curt.}\downarrow\)
& \(\mathrm{Res.}\downarrow\)
& \(\mathrm{Emerg.}\downarrow\)
& \(\mathrm{Pass}\uparrow\)
& \(\mathrm{Time~(ms)}\downarrow\) \\
\midrule

5
& \textsc{F-DRS}
& \(0.412\pm0.006\)
& \(8.03\mathrm{e}{-2}\pm6.8\mathrm{e}{-5}\)
& \(6.89\mathrm{e}{-2}\pm5.9\mathrm{e}{-4}\)
& \(1.136\pm0.005\)
& \(4.95\mathrm{e}{-1}\pm5.5\mathrm{e}{-4}\)
& \(4.13\mathrm{e}{-2}\pm1.6\mathrm{e}{-4}\)
& \(0.000\pm0.000\)
& \(\mathbf{0.012\pm0.010}\) \\

& \textsc{Grid-DRS}
& \(0.294\pm0.005\)
& \(8.97\mathrm{e}{-2}\pm8.5\mathrm{e}{-5}\)
& \(6.54\mathrm{e}{-2}\pm4.3\mathrm{e}{-4}\)
& \(0.939\pm0.004\)
& \(1.89\mathrm{e}{-1}\pm4.3\mathrm{e}{-4}\)
& \(4.23\mathrm{e}{-2}\pm6.0\mathrm{e}{-5}\)
& \(0.000\pm0.000\)
& \(0.013\pm0.010\) \\

& \textsc{S-ADMM}
& \(0.537\pm0.006\)
& \(8.03\mathrm{e}{-2}\pm5.3\mathrm{e}{-5}\)
& \(7.88\mathrm{e}{-2}\pm5.6\mathrm{e}{-4}\)
& \(1.181\pm0.005\)
& \(5.04\mathrm{e}{-1}\pm6.7\mathrm{e}{-4}\)
& \(4.13\mathrm{e}{-2}\pm8.6\mathrm{e}{-5}\)
& \(0.000\pm0.000\)
& \(0.016\pm0.006\) \\

& \textsc{AA-DRS}
& \(0.279\pm0.005\)
& \(1.03\mathrm{e}{-1}\pm7.9\mathrm{e}{-5}\)
& \(5.46\mathrm{e}{-2}\pm4.9\mathrm{e}{-4}\)
& \(\mathbf{0.369\pm0.002}\)
& \(4.35\mathrm{e}{-1}\pm3.3\mathrm{e}{-4}\)
& \(3.62\mathrm{e}{-2}\pm1.2\mathrm{e}{-4}\)
& \(0.000\pm0.000\)
& \(0.021\pm0.006\) \\

& \textsc{RCDRS-Env}
& \(\mathbf{0.056\pm0.003}\)
& \(\mathbf{5.03\mathrm{e}{-2}\pm3.9\mathrm{e}{-5}}\)
& \(\mathbf{5.00\mathrm{e}{-2}\pm2.8\mathrm{e}{-4}}\)
& \(0.762\pm0.003\)
& \(\mathbf{6.85\mathrm{e}{-5}\pm2.2\mathrm{e}{-5}}\)
& \(\mathbf{3.34\mathrm{e}{-2}\pm2.7\mathrm{e}{-4}}\)
& \(0.000\pm0.000\)
& \(0.043\pm0.017\) \\

& \textsc{RCDRS-NoEnv}
& \(0.081\pm0.002\)
& \(5.42\mathrm{e}{-2}\pm6.4\mathrm{e}{-5}\)
& \(6.71\mathrm{e}{-2}\pm1.7\mathrm{e}{-3}\)
& \(0.754\pm0.006\)
& \(4.69\mathrm{e}{-4}\pm4.9\mathrm{e}{-5}\)
& \(3.73\mathrm{e}{-2}\pm1.5\mathrm{e}{-4}\)
& \(0.000\pm0.000\)
& \(0.031\pm0.006\) \\

\midrule

10
& \textsc{F-DRS}
& \(0.108\pm0.001\)
& \(3.33\mathrm{e}{-2}\pm1.2\mathrm{e}{-4}\)
& \(6.23\mathrm{e}{-2}\pm8.7\mathrm{e}{-4}\)
& \(0.927\pm0.001\)
& \(9.19\mathrm{e}{-2}\pm3.7\mathrm{e}{-4}\)
& \(5.66\mathrm{e}{-2}\pm8.0\mathrm{e}{-4}\)
& \(0.160\pm0.012\)
& \(\mathbf{0.017\pm0.006}\) \\

& \textsc{Grid-DRS}
& \(0.029\pm0.000\)
& \(3.48\mathrm{e}{-2}\pm9.3\mathrm{e}{-4}\)
& \(3.99\mathrm{e}{-2}\pm1.7\mathrm{e}{-4}\)
& \(0.812\pm0.003\)
& \(2.32\mathrm{e}{-2}\pm6.3\mathrm{e}{-4}\)
& \(3.95\mathrm{e}{-2}\pm1.6\mathrm{e}{-4}\)
& \(0.784\pm0.004\)
& \(0.018\pm0.004\) \\

& \textsc{S-ADMM}
& \(0.216\pm0.002\)
& \(3.45\mathrm{e}{-2}\pm1.8\mathrm{e}{-4}\)
& \(7.20\mathrm{e}{-2}\pm9.4\mathrm{e}{-4}\)
& \(0.967\pm0.001\)
& \(1.04\mathrm{e}{-1}\pm5.6\mathrm{e}{-4}\)
& \(5.92\mathrm{e}{-2}\pm8.5\mathrm{e}{-4}\)
& \(0.043\pm0.005\)
& \(0.029\pm0.002\) \\

& \textsc{AA-DRS}
& \(0.026\pm0.000\)
& \(3.50\mathrm{e}{-2}\pm1.1\mathrm{e}{-4}\)
& \(3.99\mathrm{e}{-2}\pm1.6\mathrm{e}{-4}\)
& \(0.813\pm0.003\)
& \(2.32\mathrm{e}{-2}\pm6.1\mathrm{e}{-4}\)
& \(3.95\mathrm{e}{-2}\pm1.4\mathrm{e}{-4}\)
& \(0.873\pm0.005\)
& \(0.031\pm0.002\) \\

& \textsc{RCDRS-Env}
& \(\mathbf{0.023\pm0.001}\)
& \(3.15\mathrm{e}{-2}\pm4.7\mathrm{e}{-5}\)
& \(\mathbf{5.97\mathrm{e}{-3}\pm1.6\mathrm{e}{-4}}\)
& \(\mathbf{0.740\pm0.005}\)
& \(\mathbf{3.17\mathrm{e}{-3}\pm2.5\mathrm{e}{-4}}\)
& \(\mathbf{5.96\mathrm{e}{-3}\pm1.6\mathrm{e}{-4}}\)
& \(\mathbf{0.951\pm0.027}\)
& \(0.074\pm0.010\) \\

& \textsc{RCDRS-NoEnv}
& \(0.027\pm0.001\)
& \(\mathbf{3.12\mathrm{e}{-2}\pm1.9\mathrm{e}{-5}}\)
& \(7.43\mathrm{e}{-3}\pm2.2\mathrm{e}{-4}\)
& \(0.756\pm0.001\)
& \(6.30\mathrm{e}{-3}\pm4.4\mathrm{e}{-4}\)
& \(7.42\mathrm{e}{-3}\pm2.2\mathrm{e}{-4}\)
& \(0.902\pm0.013\)
& \(0.078\pm0.013\) \\

\bottomrule
\end{tabular}}
\end{table*}

\paragraph{Results.}
At the very shallow depth \(K=5\), none of the methods satisfies the joint
operational-pass criterion because the forecast equality residual remains
above the threshold of \(0.035\). Nevertheless, \textsc{RCDRS-Env} already
achieves the smallest CostGap, equality residual, realized shedding,
reserve shortfall, and emergency-deficit rate. \textsc{AA-DRS} obtains the
smallest renewable-spillage burden at this depth, illustrating that low
spillage alone does not imply a reliable or equality-feasible dispatch.

At \(K=10\), the classical methods substantially improve forecast
feasibility, but their realized reliability remains less consistent.
\textsc{RCDRS-Env} obtains the lowest CostGap, shedding rate, spillage burden,
reserve shortfall, and emergency-deficit rate, together with the highest
operational pass rate of \(0.951\). \textsc{RCDRS-NoEnv} achieves the smallest
forecast equality residual, while remaining slightly worse than the
enveloped variant on the realized reliability metrics. These results indicate
that the residual-controlled policy primarily improves the operational quality
of the finite-depth dispatch rather than merely reducing a single algebraic
residual.

The improved reliability incurs additional controller overhead. At \(K=10\),
the RCDRS variants require approximately \(0.074\)--\(0.078\) milliseconds per
instance, compared with \(0.017\)--\(0.031\) milliseconds for the classical
splitting baselines. The absolute online cost nevertheless remains below
\(0.1\) milliseconds per instance in the reported batched GPU setting.


\subsection{Runtime Measurement and Reproducibility}
\label{app:runtime_reproducibility}

Runtime is measured in the online inference setting and reported in
milliseconds per instance. The implementation synchronizes CUDA before and
after timed regions when a 4060Ti GPU is used. Synthetic solver timing uses one
warm-up pass followed by three timed repetitions; DFL and dispatch timing use
the full held-out batch and divide elapsed time by the number of instances.
The timed region includes all operations required to produce the terminal
output: feature construction, recurrent-controller evaluation, action
mapping, affine projection, cone or box projection, relaxation, and the
scaled-dual update. Offline controller training, predictor training,
validation grid search, reference-solver calls, and dataset generation are
excluded.

All experiments use PyTorch 2.0 float32. Random seeds control problem generation,
feature-to-cost maps, model initialization, and minibatch order. For a fixed
seed and depth, all methods use identical training, validation, and test splits and the
same initial state. Learned layers are trained and evaluated at the same
terminal depth. The accompanying code archive stores seed-level CSV files and
seed-aggregated summaries.
\end{document}